\documentclass{amsart}
\usepackage{amsmath}
\usepackage{mathrsfs}
\usepackage{amsfonts}
\usepackage{amssymb}
\usepackage{graphicx}
\usepackage{amsthm}
\usepackage[toc,page]{appendix}
\usepackage{geometry}
\usepackage[all]{xy}
\usepackage{lmodern}
\usepackage[T1]{fontenc}

\usepackage[
    backend=bibtex,
    style=alphabetic,
    sorting=nyt,
    maxalphanames=99,
    maxbibnames=99,
    maxcitenames=99
]{biblatex}

\usepackage[dvipsnames]{xcolor}

\definecolor{grau}{gray}{0.9} %ehemals schwarz - Sgrau
\definecolor{gelb}{RGB}{225,225,0} %ehemals hellgrau - Sgruen
\definecolor{blau}{RGB}{0,0,255} %ehemals dunkelgrau - Sblau
\definecolor{rot}{RGB}{255,0,0} %ehemals weiss - Srot

\usepackage{graphics}
\usepackage{comment}
\usepackage{csquotes}
\usepackage{hyperref}
\usepackage{enumerate}
\usepackage{caption}
\usepackage{subcaption}
\usepackage{float}
\usepackage{multicol}
\usepackage{algpseudocodex}
\usepackage{algorithm}

\usepackage[normalem]{ulem}

\renewcommand{\arraystretch}{2}
\numberwithin{equation}{section}

\newtheorem{theorem}{Theorem}[section]	
\newtheorem*{theorem*}{theorem}
\newtheorem{proposition}[theorem]{Proposition} 
\newtheorem{lemma}[theorem]{Lemma}
	
\newtheorem{definition}[theorem]{Definition}
\newtheorem*{def*}{Definition}
\newtheorem{corollary}[theorem]{Corollary}
\newtheorem{remark}[theorem]{Remark}
\newtheorem{example}[theorem]{Example}

\usepackage{tikz}
\usetikzlibrary{decorations,arrows}
\usetikzlibrary{decorations.pathmorphing}
\usepgflibrary{decorations.pathreplacing} 
\usepackage{pgf}
\usetikzlibrary{patterns}

\tikzset{
	schraffiert/.style={pattern=horizontal lines,pattern color=#1},
	schraffiert/.default=black
}

\tikzset{
	ultra thin/.style= {line width=0.1pt},
	very thin/.style=  {line width=0.2pt},
	thin/.style=       {line width=0.4pt},% thin is the default
	semithick/.style=  {line width=0.6pt},
	thick/.style=      {line width=0.8pt},
	very thick/.style= {line width=1.2pt},
	ultra thick/.style={line width=2.4pt}
}

\newcommand{\C}{{\mathbb C}}

\newcommand{\N}{{\mathbb N}}

\newcommand{\R}{{\mathbb R}}

\newcommand{\Z}{{\mathbb Z}}

\newcommand{\Aa}{{\mathcal A}}
\newcommand{\Ee}{{\mathcal E}}

\newcommand{\Jj}{{\mathcal J}}

\newcommand{\Ll}{{\mathcal L}}
\newcommand{\Pp}{{\mathcal P}}
\newcommand{\Qq}{{\mathcal Q}}

\newcommand{\Vv}{{\mathcal V}}

\newcommand{\Orb}{{\mathrm Orb}}				%%orbit
\DeclareMathOperator{\supp}{{\mathrm supp}}		%%support
\newcommand{\Pat}{\mathrm{Pat}(\Aa^\Gamma)}				%%pattern
\newcommand{\PatZ}{\mathrm{Pat}(\Aa^{\Z^2})}				%%pattern
\definecolor{ColorR}{RGB}{220,50,47}   % r = red
\definecolor{ColorY}{RGB}{255,214,10}  % y = yellow
\definecolor{ColorB}{RGB}{38,139,210}  % b = blue
\definecolor{ColorG}{RGB}{160,160,160} % g = gray

\tikzset{
  Tiler/.style={draw=black, fill=ColorR, line width=0.25pt},
  Tiley/.style={draw=black, fill=ColorY, line width=0.25pt},
  Tileb/.style={draw=black, fill=ColorB, line width=0.25pt},
  Tileg/.style={draw=black, fill=ColorG, line width=0.25pt},
  Frame/.style={draw=black, line width=0.5pt}
}

\newcommand{\mathaxisbox}[1]{%
  \mathord{\vcenter{\hbox{#1}}}%
}
\newcommand{\letterbox}[1]{%
  \mathaxisbox{%
    \begin{tikzpicture}[scale=0.35]
      \filldraw[Tile#1] (0,0) rectangle (1,1);
      \draw[Frame] (0,0) rectangle (1,1);
    \end{tikzpicture}%
  }%
}

\newcommand{\block}[4]{%
  \mathaxisbox{%
    \begin{tikzpicture}[scale=0.35]
      \filldraw[Tile#1] (0,0) rectangle (1,1);
      \filldraw[Tile#2] (1,0) rectangle (2,1);
      \filldraw[Tile#3] (0,-1) rectangle (1,0);
      \filldraw[Tile#4] (1,-1) rectangle (2,0);
      \draw[Frame] (0,-1) rectangle (2,1);
    \end{tikzpicture}%
  }%
}
\newcommand{\vblock}[2]{%
  \raisebox{-0.5\height}{%
    \begin{tikzpicture}[scale=0.35]
      \filldraw[Tile#1] (0,0) rectangle (1,1);
      \filldraw[Tile#2] (0,-1) rectangle (1,0);
      \draw[Frame] (0,-1) rectangle (1,1);
    \end{tikzpicture}%
  }%
}
\newcommand{\hblock}[2]{%
  \raisebox{-0.2\height}{%
    \begin{tikzpicture}[scale=0.35]
      \filldraw[Tile#1] (0,0) rectangle (1,1);
      \filldraw[Tile#2] (1,0) rectangle (2,1);
      \draw[Frame] (0,0) rectangle (2,1);
    \end{tikzpicture}%
  }%
}
\newcommand{\blocktwothree}[6]{%
  \mathaxisbox{%
    \begin{tikzpicture}[scale=0.35]
      \filldraw[Tile#1] (0,2) rectangle (1,3);
      \filldraw[Tile#2] (1,2) rectangle (2,3);
      \filldraw[Tile#3] (2,2) rectangle (3,3);
      \filldraw[Tile#4] (0,1) rectangle (1,2);
      \filldraw[Tile#5] (1,1) rectangle (2,2);
      \filldraw[Tile#6] (2,1) rectangle (3,2);
      \draw[Frame] (0,1) rectangle (3,3);
    \end{tikzpicture}%
  }%
}
\newcommand{\blocknine}[9]{%
  \mathaxisbox{%
    \begin{tikzpicture}[scale=0.35]
      \filldraw[Tile#1] (0,2) rectangle (1,3);
      \filldraw[Tile#2] (1,2) rectangle (2,3);
      \filldraw[Tile#3] (2,2) rectangle (3,3);
      \filldraw[Tile#4] (0,1) rectangle (1,2);
      \filldraw[Tile#5] (1,1) rectangle (2,2);
      \filldraw[Tile#6] (2,1) rectangle (3,2);
      \filldraw[Tile#7] (0,0) rectangle (1,1);
      \filldraw[Tile#8] (1,0) rectangle (2,1);
      \filldraw[Tile#9] (2,0) rectangle (3,1);
      \draw[Frame] (0,0) rectangle (3,3);
    \end{tikzpicture}%
  }%
}
\usetikzlibrary{shapes}
\usetikzlibrary{positioning}

\title{Spectral pollution in substitution systems}
\author{Ram Band, Siegfried Beckus, Felix Pogorzelski and Lior Tenenbaum}

\address{Department of Mathematics\\
Technion - Israel Institute of Technology\\
Haifa
}
\email{ramband@technion.ac.il}

\address{Institute of Mathematics\\
University of Potsdam\\
Potsdam}
\email{beckus@uni-potsdam.de}

\address{Department of Mathematics and Computer Science\\
University of Leipzig, Leipzig}
\email{felix.pogorzelski@math.uni-leipzig.de}

\address{Department of Mathematics, Bar-Ilan University\\ 
Ramat Gan}
\email{lior.tenen.25@gmail.com}

\begin{document}
	\begin{abstract}
	We study spectral properties of Schrödinger operators associated with substitution dynamical systems in higher dimensions. Focusing on periodic approximations generated by iterating substitutions on initial configurations, we analyze how structural defects influence the limiting spectral behavior. In contrast to the one-dimensional setting, we show that such approximations may exhibit significant spectral pollution, including changes in the essential spectrum and the Lebesgue measure. 
	\end{abstract}
	
	\maketitle

\setcounter{tocdepth}{1} %Only sections are displayed in table of content not subsections
\tableofcontents
	
%%%%%%%%%%%%%%%%%%%%%%%%%%%%%%%%%%%%%%%%%%%%%%%%%%%%%%%%%%%%%%%%%%%%%%%%%%%%%%%%%%%%%%%%%%%%%%%%%%%%%%%%%%%%%%%%%%%%%%%%%%%%%%%%%%%%%%%%%%%%%%%%
%%%%%%%%%%%%%%%%%%%%%%%%%%%%%%%%%%%%%%%%%%%%%%%%%%%%%%%%%%%%%%%%%%%%%%%%%%%%%%%%%%%%%%%%%%%%%%%%%%%%%%%%%%%%%%%%%%%%%%%%%%%%%%%%%%%%%%%%%%%%%%%%
\section{Introduction}
\label{Sec:Intro}
%%%%%%%%%%%%%%%%%%%%%%%%%%%%%%%%%%%%%%%%%%%%%%%%%%%%%%%%%%%%%%%%%%%%%%%%%%%%%%%%%%%%%%%%%%%%%%%%%%%%%%%%%%%%%%%%%%%%%%%%%%%%%%%%%%%%%%%%%%%%%%%%
%%%%%%%%%%%%%%%%%%%%%%%%%%%%%%%%%%%%%%%%%%%%%%%%%%%%%%%%%%%%%%%%%%%%%%%%%%%%%%%%%%%%%%%%%%%%%%%%%%%%%%%%%%%%%%%%%%%%%%%%%%%%%%%%%%%%%%%%%%%%%%%%

Our results are motivated by two block substitutions $S:\Aa^{\Z^2}\to\Aa^{\Z^2}$ over a finite alphabet $\Aa$: the table tiling and the chair tiling \cite{Robin99}. Numerical computations reveal significantly different spectral behavior for these two substitutions. We consider a periodic configuration $\omega_0\in\Aa^{\Z^2}$ and study the associated Schr\"odinger operators $H_\omega$ along the sequence $\omega=S^n(\omega_0)$. Since $\omega_0$ is periodic, each iterate $S^n(\omega_0)$ gives rise to a periodic approximation of the dynamical system associated with $S$. 

Periodic approximations constitute a standard tool in spectral theory, both for numerical and analytical investigations, since their spectra can be accessed via Floquet--Bloch theory, see \cite{DaFi22-book_1,DaFi24-book_2}. Our results show that the choice of approximations is crucial: the spectral discrepancy between limits of such periodic approximations and the operator associated with $S$ is called a \emph{spectral pollution} and it strongly depends on the underlying substitution system and $\omega_0$. In contrast to the one-dimensional setting, where spectral pollution is limited to isolated eigenvalues, higher-dimensional substitution systems may exhibit a different behavior, including changes in the essential spectrum and even in the Lebesgue measure of the spectrum.

While our investigation is initiated by substitutions on $\Z^2$, several of our
results apply in a more general setting, where $\Z^2$ is replaced by $\Z^d$ or even a suitable
lattice $\Gamma$ in a homogeneous Lie group.
Such substitutions were recently introduced in \cite{BHP21-Symbolic} and include all
block substitutions \cite{Fra05,BaBePoTe24}, certain digit substitutions \cite{Vin00,FraMan22}, as
well as substitutions on the discrete Heisenberg group.

We adopt this general framework not merely for the sake of abstraction, but
because it allows for a systematic and efficient analysis of supports of
substitution iterates.
Moreover, this work highlights the \emph{testing domain} introduced in
\cite{BaBePoTe24} as a fundamental tool for the study of structural defects.

Let $\Aa$ be a finite set, called an \emph{alphabet}, and let $\Gamma$ be a countable group. We write
\[
\Pat:=\{P:M\to\Aa \mid M\subseteq\Gamma\}
\]
for the set of patches and $\supp(P):=M$ denotes the support of $P\in\Pat$.
Note that $\Aa^\Gamma\subseteq \Pat$.
The group $\Gamma$ acts on $\Pat$ by translations via
\[
(\gamma P)(x) := P(\gamma^{-1}x), \qquad \gamma,x\in\Gamma.
\]
Here we write the group operation multiplicatively, as $\Gamma$ does not need to be abelian.
For patches $P,Q\in\Pat$ we write $P\prec Q$ (and call $P$ a \emph{subpatch} of $Q$)
if there exists $\gamma\in\Gamma$ such that
\[
\gamma\supp(P)\subseteq \supp(Q)
\quad\text{and}\quad
\gamma P = Q\big|_{\gamma\supp(P)} .
\]
For $\omega\in\Aa^\Gamma$, define $W(\omega)$ as the set of all finitely supported patches $P$ with $P\prec \omega$. Let $F \subseteq \Gamma$ and $W(\omega)_F$ be the set of all $P\in W(\omega)$ such that $\supp(P)=F$.

We equip $\Aa^\Gamma$ with a complete metric $d$ inducing the product topology.
The main results presented here are independent of the particular choice of metric
\cite{BeckusThesis,BBdN18}, but fixing one simplifies the exposition.

A closed, non-empty, and $\Gamma$-invariant subset $\Omega\subseteq\Aa^\Gamma$ is called a \emph{subshift}. Then the $\Gamma$ action on subshifts defines a dynamical system. Typical examples are orbit closures, $\overline{\Orb(\omega)}$, where $\Orb(\omega):=\{\gamma\omega\mid \gamma\in\Gamma\}$ denotes the orbit of $\omega\in\Aa^\Gamma$. 
Following \cite{BeckusThesis,BBdN18}, the collection $\Jj$ of all subshifts, viewed as compact subsets of $\Aa^\Gamma$ and equipped with the Hausdorff metric, forms a compact metrizable space. Convergence in $\Jj$ can be characterized by the convergence of the associated dictionaries $W(\Omega):=\bigcup_{\omega\in\Omega}W(\omega)$ in the local pattern topology \cite{BeckusThesis}. This provides the natural setting for analyzing limit points of sequences of subshifts arising from iterating substitutions.

Here, a substitution is specified by a lattice $\Gamma$
and a substitution rule $S_0:\Aa\to\Aa^F$. 
Here $F\subseteq\Gamma$ is a finite set whose left translates by elements of a suitable subgroup $D(\Gamma)\subseteq\Gamma$ tessellate $\Gamma$.
Under suitable assumptions, the rule
$S_0$ extends to a substitution map $S:\Pat\to\Pat$ acting letterwise according to $S_0$, see Proposition~\ref{Prop:Exist_SubstMap}.

Let $W(S)$ denote the set of all $S$-legal patches $P$, that is, patches that occur in some iterate $S^n(a)$ for a letter $a\in\Aa$ and some $n\in\N$; see Definition~\ref{Def:LegalPatches}. Then there is a unique subshift $\Omega(S)$ associated with $S$, consisting of all $\omega\in\Aa^\Gamma$ such that $W(\omega)\subseteq W(S)$.
The examples of substitutions discussed here are moreover \emph{primitive}, namely there exists a $k\in\N$ such that $a\prec S^k(b)$ for all $a,b\in\Aa$.

For a finite set $B\subseteq \Gamma$ with $B=B^{-1}$ and a continuous function $V:\Aa^\Gamma\to\R$, consider the bounded self-adjoint operator family $H_\omega:\ell^2(\Gamma)\to\ell^2(\Gamma)$, $\omega\in\Aa^\Gamma$, defined by
\[
(H_{\omega}\psi)(\gamma) := \sum_{\eta\in B} \psi(\gamma\eta) + V(\gamma^{-1}\omega)\psi(\gamma),
\qquad 
\psi\in\ell^2(\Gamma),\ \gamma\in\Gamma .
\]
For $\Gamma=\Z^d$ choose $B=\{\pm e_j \mid 1\le j\le d\}$, where $e_j$ denotes the standard basis of $\Z^d$. In this case, $H_\omega$ is a discrete Schr\"odinger operator whose potential is given by the multiplication operator on $\ell^2(\Gamma)$ associated with the function $\Gamma\ni\gamma\mapsto V(\gamma^{-1}\omega)$, see also \eqref{Eq:SchrOp_SingleLetter} below for an explicit example.

Starting from a configuration $\omega_0\in\Aa^\Gamma$, we study the behavior of the iterated spectra $\sigma(H_{S^n(\omega_0)})$ as $n\to\infty$. In particular, we analyze the possible limit points of these spectra and compare them with the spectrum associated with the substitution system, namely $\sigma(H_\omega)$ for some $\omega\in\Omega(S)$. 
If $S$ is primitive, the spectrum $\sigma(H_\omega)$ and its essential spectrum $\sigma_{\mathrm{ess}}(H_\omega)$ coincide and are independent of $\omega\in\Omega(S)$, see \eqref{Eq:essSpectrum_Min} in Section~\ref{Sec:SpectralResults}.
In order to study the spectral pollution, we analyze limit points $\Omega$ of the sequence of subshifts $\bigl(\overline{\Orb(S^n(\omega_0))}\big)_n$ in $\Jj$. Then the patches in $W(\Omega)$ that are not $S$-legal are called the \emph{structural defects}, namely $W(\Omega)\setminus W(S)$.

%%%%%%%%%%%%%%%%%%%%%%%%%%%%%%%%%%%%%%%%%%%%%%%%%%%%%%%%%%%%%%%%%%%%%%%%%
\subsection{Table tiling}
\label{Subsec:TableTiling}
%%%%%%%%%%%%%%%%%%%%%%%%%%%%%%%%%%%%%%%%%%%%%%%%%%%%%%%%%%%%%%%%%%%%%%%%%

We start with the table tiling substitution defined by the substitution rule
$S_0:\Aa\to\Aa^F$ on the alphabet
$\Aa:=\{r,y,b,g\}=\{\letterbox{r},\letterbox{y},\letterbox{b},\letterbox{g}\}$ and $F=\{0,1\}^2\subseteq \Z^2=\Gamma$,
\[
\letterbox{r}\mapsto\block{y}{r}{g}{r},
\qquad
\letterbox{y}\mapsto\block{r}{b}{y}{y},
\qquad
\letterbox{b}\mapsto\block{b}{y}{b}{g},
\qquad
\letterbox{g}\mapsto\block{g}{g}{r}{b}.
\]
This substitution rule is componentwise bijective, meaning that at each position in the support $F$ it induces a bijection (i.e. permutation) on the alphabet.
Therefore it falls into the class of so-called $T$-bijective substitutions considered in Section~\ref{Sec:TwoClasses}, see Corollary~\ref{Cor:Table_Tbijective}.

The map $S_0$ naturally extends to a substitution map $S:\PatZ\to\PatZ$ by applying $S_0$ letterwise, see \cite{Fra05,BaakeGrimm13}. This works because $F$ is a fundamental domain of the subgroup\footnote{Using the notation from before we have $\Gamma=\Z^2$ and $D(\Gamma)=(2\Z)^2$.} $(2\Z)^2\subseteq\Z^2$. An illustration of this letterwise application is given in Figure~\ref{Fig:ApplSubstTable}.

\begin{figure}[htb]
\centering
\includegraphics[width=\textwidth]{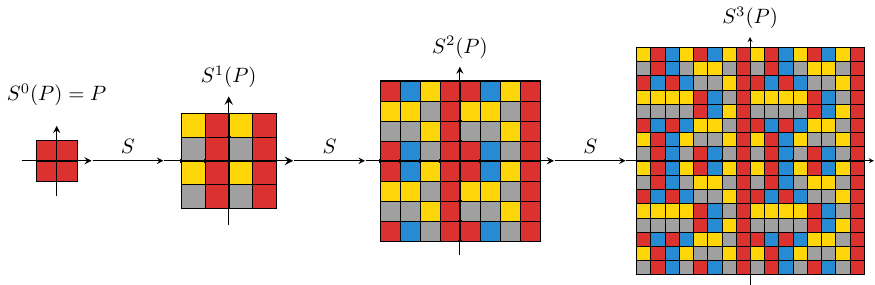}
\caption{Illustration of three successive letterwise applications of the Table tiling substitution
to an initial $2\times2$ patch fixed around the origin.}
\label{Fig:ApplSubstTable}
\end{figure}

Note that the table tiling substitution is primitive, namely $a\prec S^2(b)$ for all $a,b\in\Aa$.
Consider the periodic configuration
$\rho_{rb}\in\Aa^{\Z^2}$ defined by
\[
\rho_{rb}(n) :=
\begin{cases}
\letterbox{r}, & n\in (2\Z)^2\cup \bigl((1,1)+(2\Z)^2\bigr),\\
\letterbox{b}, & n\in \bigl((0,1)+(2\Z)^2\bigr)\cup \bigl((1,0)+(2\Z)^2\bigr).
\end{cases}
\]
Then the periodic subshifts $\overline{\Orb\big(S^n(\rho_{rb})\big)}$ converge to $\Omega(S)$ in $\Jj$, see \cite{BeckusThesis,BaBePoTe24}.
For the associated Schr\"odinger operators, this yields that the spectrum of the periodic operators $H_{S^n(\rho_{rb})}$ converges exponentially fast to $\sigma(H_\omega)$ for all $\omega\in\Omega(S)$, see \cite[Prop.~1.3]{BaBePoTe24}.

In contrast to the one-dimensional setting, where constant-letter approximations are commonly used, \cite[Cor.~1.2]{BaBePoTe24} shows that constant configurations do not converge in the present example. We therefore study the resulting structural defects and the dependence of spectral pollution on the initial configuration $\omega_0$.

Consider the operator
$H_\omega:\ell^2(\Z^2)\to\ell^2(\Z^2),\, \omega\in\Aa^{\Z^2}$,
with single-letter potential defined by
\begin{equation}
\label{Eq:SchrOp_SingleLetter}
(H_\omega\psi)(n)
:= \sum_{j=1}^2 \bigl(\psi(n+e_j)+\psi(n-e_j)\bigr)
   + v(\omega(n))\,\psi(n),
\qquad \psi\in\ell^2(\Z^2).
\end{equation}
Here $\{e_1,e_2\}$ denotes the standard basis of $\Z^2$, and
$v:\Aa\to\R$ is chosen such that their mutual differences exceed the operator norm of the adjacency operator.

\begin{figure}[htb]
\centering
\includegraphics[width=0.9\textwidth]{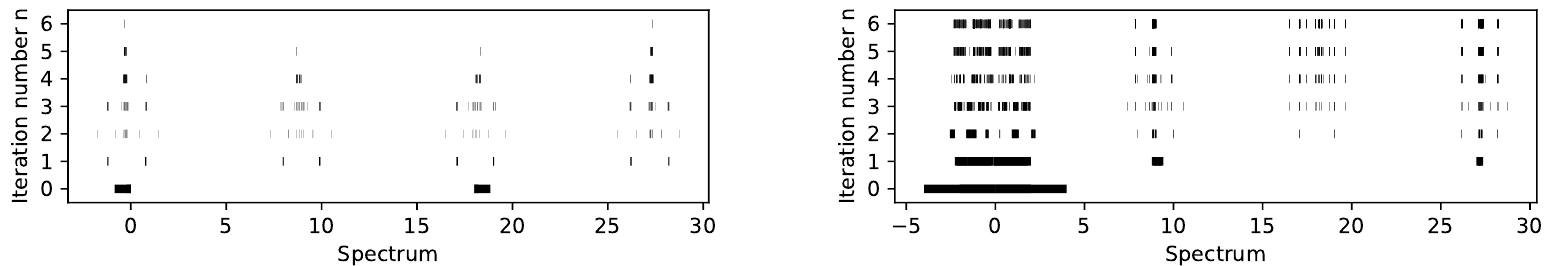}
\caption{
For the table tiling substitution, we compute the spectra of the Schr\"odinger operators with single-letter potential using Floquet--Bloch theory. 
The left panel shows $\sigma(H_{S^n(\rho_{rb})})$, while the right panel shows $\sigma(H_{S^n(\omega_0)})$ for a constant configuration, see \cite{BeTe26-Numer}.
}
\label{Fig:SpectraTable}
\end{figure}

For the choice
\[
v\bigl(\letterbox{r}\bigr) := 0,\qquad
v\bigl(\letterbox{y}\bigr) := 9,\qquad
v\bigl(\letterbox{b}\bigr) := 18,\qquad
v\bigl(\letterbox{g}\bigr) := 27,
\]
we performed numerical simulations for the constant configuration
$\omega_0\in\Aa^{\Z^2}$ with $\omega_0(n)=\letterbox{r}$, as well as for the periodic
configuration $\rho_{rb}$.
As illustrated in Figure~\ref{Fig:SpectraTable}, the spectral bands arising from
the sequence $(S^n(\omega_0))_{n\in\N}$ exhibit a significant different thickness%
\footnote{This phenomenon was observed for all constant-letter configurations.}
compared to those obtained from the convergent periodic approximations
$(S^n(\rho_{rb}))_{n\in\N}$.
This indicates the presence of spectral pollution in the limit and motivates the analysis presented here. 

By spectral pollution we mean that a limit point of the spectra $\sigma(H_{S^n(\omega_0)})$ (or the essential spectra $\sigma_{\mathrm{ess}}(H_{S^n(\omega_0)})$)
does not coincide with $\sigma(H_{\omega})$ (or $\sigma_{\mathrm{ess}}(H_{\omega})$) for some fixed $\omega\in\Omega(S)$. In general, the two sets differ. For the table tiling, we show later that even the Lebesgue measure of these sets does not need to coincide, see Corollary~\ref{Cor:TableTiling_Summary}.

To address this phenomenon, recall that $\Omega(S)$ consists of all
$\omega\in\Aa^\Gamma$ admitting only $S$-legal patches.
We show in Theorem~\ref{Thm:SubshiftPartialLimit} that it suffices to study the
$2\times2$-patches \(W(S^n(\omega_0))_{2\times2}\) when identified with $T$-patches for a so-called convenient testing domain $T\subseteq\Gamma$, see Definition~\ref{Def:TestingDomain} below.
Since our focus is on structural defects, we disregard the $S$-legal $2\times2$-patches given by
\begin{equation}
\label{Eq:LegalPatches_Table}
\begin{array}{cccccccccccc}
\block{b}{g}{b}{y} & \block{b}{g}{r}{b} & \block{b}{g}{y}{r} & \block{b}{r}{r}{b} & \block{b}{r}{y}{y} & \block{b}{y}{b}{g} &
\block{b}{y}{y}{g} & \block{g}{g}{b}{r} & \block{g}{g}{r}{b} & \block{g}{r}{b}{y} & \block{g}{r}{r}{b} & \block{g}{r}{y}{r} \\
\block{g}{y}{b}{g} & \block{g}{y}{y}{g} & \block{r}{b}{b}{r} & \block{r}{b}{b}{y} & \block{r}{b}{r}{b} & \block{r}{b}{y}{r} &
\block{r}{b}{y}{y} & \block{y}{g}{g}{r} & \block{y}{g}{g}{y} & \block{y}{r}{g}{r} & \block{y}{r}{g}{y} & \block{y}{y}{g}{g} 
\end{array}.
\end{equation}
We therefore study which sets of $S$-illegal $2\times 2$-patches stabilize under substitution iterates within $S^n(\omega_0)$.
Since every $2\times2$-patch occurring in $S^{n+1}(\omega_0)$ is contained in $S(P)$ for some $P\in W(S^n(\omega_0))_{2\times2}$, it follows that the
sets of $2\times2$-patches $W(S^n(\omega_0))_{2\times 2}$ are eventually periodic with some period $\ell_0$, see Lemma~\ref{Lem:EventualPeriodicityT} for details.
Indeed, a direct computation shows that for $j\in\{0,1\}$, $N_0=4$, and $\ell_0=2$,
\[
W(S^{N_0+j}(\omega_0))_{2\times2}
=
W(S^{N_0+j+\ell_0}(\omega_0))_{2\times2}.
\]
More precisely, the two distinct patch sets $\Qq_j := W\big(S^{N_0+j}(\omega_0)\big)_{2\times2}, \, j\in\{0,1\}$, are given by
\begin{equation}
\label{Eq:DefectsTable}
\Qq_0
=
\Bigl\{
\block{r}{r}{r}{r},
\block{r}{g}{r}{r},
\block{r}{y}{r}{g},
\block{r}{r}{r}{y}
\Bigr\}\cup  W(S)_{2\times 2}, \qquad
\Qq_1
=
\Bigl\{
\block{r}{g}{r}{y},
\block{r}{y}{r}{g},
\block{r}{r}{r}{y},
\block{r}{g}{r}{r}
\Bigr\}\cup  W(S)_{2\times 2}.
\end{equation}
Let us focus on the specific patches
\[
Q_0 := \block{r}{r}{r}{r} \in \Qq_0 \setminus W(S)
\qquad\text{and}\qquad
Q_1 := \block{r}{g}{r}{y} \in \Qq_1 \setminus W(S).
\]

As illustrated in Figure~\ref{Fig:ApplSubstTable} (see the $2\times 2$-patches around the origin), applying the substitution map $S$ twice to $Q_0$, respectively $Q_1$, reproduces the same patch at the origin. 
Since the substitution expands in all directions and we start from a $2\times2$-patch centered at the origin, the limits
\[
\rho_j := \lim_{n\to\infty} S^{2n+j}(\omega_j) \in \Aa^{\Z^2},
\qquad
\text{where } \omega_j|_{\{-1,0\}^2} = Q_j,\; j\in\{0,1\},
\]
exist. Moreover, they are independent of $\omega_j\in\Aa^{\Z^2}$, and satisfy $S^2(\rho_j)=\rho_j$, that is, $\rho_j$ is a fixed point of $S^2$.

Convergence in the product topology of $\Aa^{\Z^2}$ implies that
$S^{2n+j}(\omega_0)$ stabilizes on arbitrarily large balls for $n$ sufficiently large.
Iteratively applying the substitution as in Figure~\ref{Fig:ApplSubstTable}, one verifies that every $2\times2$-patch 
$P$ of $\rho_j$ supported inside the $y$-axis strip
\[
\Lambda := \{(i,k)\mid i\in\{-1,0\},\ k\in\Z\}
= \Z e_2 + T,
\qquad e_2:=(0,1),\ T:=\{-1,0\}^2,
\]
satisfies $P\in \Qq_j \setminus W(S)$. Moreover, every patch in $\Qq_j \setminus W(S)$ occurs in $\rho_j|_{\Lambda}$. 
The resulting infinite line of $S$-illegal $T$-patches of $\rho_j$ is called a \emph{pure line defect}.

We will show next that this pure line defect determines the structural defect
of the limit subshift $\Omega_j$ of the sequence $\overline{\Orb(S^n(\omega_0))}$, in the sense
that all $S$-illegal patches of $\rho_j$ arise from it.
To this end, note that the $2\times1$-patches
\[
R_1 := \vblock{r}{r},
\qquad
R_2 := \vblock{g}{y},
\]
are $S$-legal see \eqref{Eq:LegalPatches_Table}.
Using $S^2(\rho_j)=\rho_j$, a direct computation shows that every subpatch
$P$ of $\rho_j$ that is supported entirely in one of the half-planes
\[
\Upsilon_+ := \{(i,k)\in\Z^2 \mid i\geq 0\},
\qquad
\Upsilon_- := \{(i,k)\in\Z^2 \mid i<0\},
\]
occurs in $S^{2n}(R_l)$ for some $n\in\N$ and $l\in\{1,2\}$.
Hence, $P\in W(S)$ since $R_l\in W(S)$.

The present example falls within the framework of Theorem~\ref{Thm:SubshiftPartialLimit}, Proposition~\ref{Prop:Tperiodic_Inclusion}, and Proposition~\ref{Prop:T-bijective_FixedPoints}. Combined with the preceding considerations, this yields
\begin{equation}
\label{Eq:TableTiling_Limit}
\lim_{n\to\infty} \overline{\Orb\bigl(S^{2n+j}(\omega_0)\bigr)}
= \Omega_j
= \overline{\Orb(\rho_j)}
,\qquad \text{for } j\in\{0,1\}.
\end{equation}
The existence of two limit points comes from \eqref{Eq:DefectsTable} as discussed  in Section~\ref{Sec:SubshiftLimit}.
We show in Theorem~\ref{Thm:SubshiftPartialLimit} that convergence along arithmetic subsequences is a general phenomenon for substitutions. On the other hand a representation of limit points $\Omega_j$ by fixed points $\rho_j$ is false in general, see Example~\ref{Ex:Subst_NoFixedPoints}. 
However the existence of such fixed points allows us to classify the structural defects and derive the following spectral consequences in Section~\ref{Sec:SpectralResults}.

Consider a Schr\"odinger operator of finite range $H=(H_\omega)_{\omega\in\Aa^{\Z^2}}$, see Section~\ref{Sec:SpectralResults}. Denote by $\sigma(H_\omega)$ the \emph{spectrum} and by $\sigma_{\mathrm{ess}}(H_\omega)$ the \emph{essential spectrum} of the bounded self-adjoint operator $H_\omega$. The operator family is covariant with respect to the $\Gamma$-action and strongly continuous in $\omega\in\Aa^\Gamma$. 
Thus, using the second equality in \eqref{Eq:TableTiling_Limit} and semicontinuity of the spectrum \cite[Cor.~1.4.22]{DaFi22-book_1}, we obtain
\[
\sigma(H_\omega)\subseteq \sigma(H_{\rho_j}),
\qquad \textrm{for all } \omega\in\Omega_j \text{ and } j\in\{0,1\}.
\]
By Proposition~\ref{Prop:ConvSubsh-Spectra}, the first equality in \eqref{Eq:TableTiling_Limit} implies that the sequence of spectra $\bigl(\sigma(H_{S^n(\omega_0)})\bigr)_{n\in\N}$ has exactly two limit points with respect to the Hausdorff metric on compact subsets of $\R$, namely $\sigma(H_{\rho_0})$ and $\sigma(H_{\rho_1})$, see Proposition~\ref{Prop:ConvSubsh-Spectra} and Corollary~\ref{Cor:TableTiling_Summary}. 

Moreover, for suitable choices of the potential and using the presence of a pure line defect, we show
\[
\sigma_{\mathrm{ess}}(H_{\rho_j})\setminus\sigma_{\mathrm{ess}}(H_\omega)\neq\emptyset, \qquad \text{for all } \omega\in\Omega(S) \text{ and } j\in\{0,1\},
\]
and even that in more particular cases the Lebesgue measure can change
\[
\mathrm{Leb}\!\left(\sigma(H_{\rho_j})\setminus\sigma(H_\omega)\right)>0, \qquad \text{for all } \omega\in\Omega(S) \text{ and } j\in\{0,1\},
\]
see Corollary~\ref{Cor:TableTiling_Summary}. 
Note that changes in the essential spectrum arise from modifying the potential,
whereas changes in its Lebesgue measure arise from constructing specific potentials that assign large weights to $S$-illegal patches.
This indicates that, in general, pure line defects may indeed lead to such
spectral effects.
This observation is consistent with the numerical computations shown in
Figure~\ref{Fig:SpectraTable}, which were performed for Schrödinger operators with
single-letter potential, see \eqref{Eq:SchrOp_SingleLetter}.

All iterative constant-letter approximations $\omega_0$ (for any letter in $\Aa$) produce a pure line defect, either parallel to the $x$-axis or the $y$-axis.
Moreover, if $\omega_0\in\Aa^{\Z^2}$ contains the $2\times2$-patch $\block{r}{y}{g}{b}$, then a limit point exhibits a pure line defect in both directions.

The table tiling is a block substitution, that is a substitution for which the support $F\subseteq\Z^d$ is rectangular in the positive quadrant \cite{Fra05}.
For block substitutions, the rectangular shape of $F$ forces the boundaries of iterated supports to align with the coordinate axes, and therefore structural defects only arise along these axes.

In contrast, for more general substitutions, the boundary structure of iterated supports can be significantly more intricate. For instance, digit substitutions on $\Z^d$ \cite{Vin00,FraMan22} or substitutions on lattices in homogeneous Lie groups, such as the Heisenberg group \cite{BHP21-Symbolic}, may exhibit fundamentally different types of structural defects.

However, even within $\Z^d$, substitutions do not necessarily exhibit pure line defects. The chair tiling is a prototypical example for this phenomenon.

%%%%%%%%%%%%%%%%%%%%%%%%%%%%%%%%%%%%%%%%%%%%%%%%%%%%%%%%%%%%%%%%%%%%%%%%%
\subsection{Chair tiling}
\label{Subsec:ChairTiling}
%%%%%%%%%%%%%%%%%%%%%%%%%%%%%%%%%%%%%%%%%%%%%%%%%%%%%%%%%%%%%%%%%%%%%%%%%
We continue with the chair tiling substitution defined by the substitution rule
$S_0:\Aa\to\Aa^F$ on the alphabet
$\Aa:=\{r,y,b,g\}=\{\letterbox{r},\letterbox{y},\letterbox{b},\letterbox{g}\}$ and 
$F=\{0,1\}^2\subseteq\Z^2=\Gamma$,
\[
\letterbox{r}\mapsto\block{y}{r}{r}{g},
\qquad
\letterbox{y}\mapsto\block{y}{b}{r}{y},
\qquad
\letterbox{b}\mapsto\block{y}{b}{b}{g},
\qquad
\letterbox{g}\mapsto\block{g}{b}{r}{g}.
\]
Like the table tiling substitution, this map extends to a substitution map
$S:\PatZ\to\PatZ$ by applying $S_0$ letterwise \cite{Fra05,BaakeGrimm13}. 
It does belong to another class of substitutions, which we call \emph{purely $T$-illegal}, see Section~\ref{Subsec:Central_illegal}. Roughly speaking, after sufficiently many substitution steps, $S$-illegal $2\times 2$-patches can occur only at the origin at $T$. 
The computations are carried out in Appendix~\ref{App:ChairTilling} in Lemma~\ref{Lem:Chair_CentralIllegal} where we also provide a list of $S$-legal $2\times 2$-patches.
Note that the chair tiling substitution is primitive, namely $a\prec S^2(b)$ for all $a,b\in\Aa$.

\begin{figure}[htb]
\centering
\includegraphics[width=0.55\textwidth]{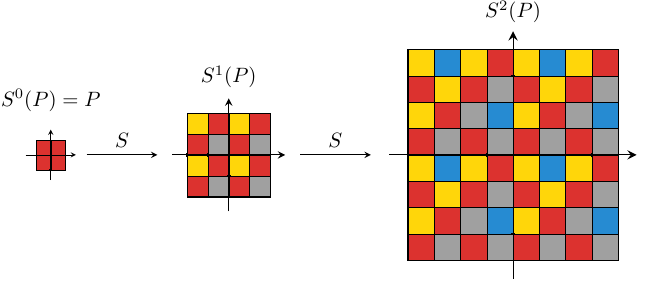}
\caption{Illustration of two successive letterwise applications of the chair tiling substitution
to an initial $2\times2$ patch fixed around the origin.}
\label{Fig:ApplSubstChair}
\end{figure}

As for the table tiling, we proceed by considering several initial configurations and by analyzing their convergence, both at the level of the dynamical systems and of the associated spectra.

Consider the periodic configuration $\rho_{\mathrm{leg}}\in\Aa^{\Z^2}$ obtained by periodizing the $2\times3$-patch $\blocktwothree{y}{b}{g}{r}{y}{b}$. A direct computation shows that $W(S(\rho_{\mathrm{leg}}))_{2\times 2}\subseteq W(S)$. Hence, the periodic subshifts $\overline{\Orb\big(S^n(\rho_{\mathrm{leg}})\big)}$, $n\in\N$, converge to $\Omega(S)$ in $\Jj$, see \cite[Thm.~2.14]{BaBePoTe24}. As in the table tiling, this implies exponentially fast convergence of the spectra of the associated Schr\"odinger operators, see \cite{BaBePoTe24}.

Let $\omega_0\in\Aa^{\Z^2}$ be the constant configuration defined by $\omega_0(n)=\letterbox{r}$ for all $n\in\Z^2$. From Figure~\ref{Fig:ApplSubstChair} we inductively obtain
\begin{equation}
\label{Eq:Chair_CentralPatch}
S^n(\omega_0)|_T=\block{g}{r}{r}{y}
\qquad\text{for all } n\in\N,\qquad T=\{-1,0\}^2 .
\end{equation}
This patch is $S$-illegal (see Appendix~\ref{App:ChairTilling}), and hence the subshifts $\overline{\Orb(S^n(\omega_0))}$, $n\in\N$, do not converge to $\Omega(S)$ in $\Jj$.

We performed the same numerical analysis with the same choice of potential as for the table tiling, see Figure~\ref{Fig:SpectraChair}. In contrast to the table tiling, we do not observe a significant thickening of the spectral bands for the poor approximations $S^n(\omega_0)$. Hence, the numerics do not indicate substantial spectral pollution as discussed next.

\begin{figure}[htb]
\centering
\includegraphics[width=0.9\textwidth]{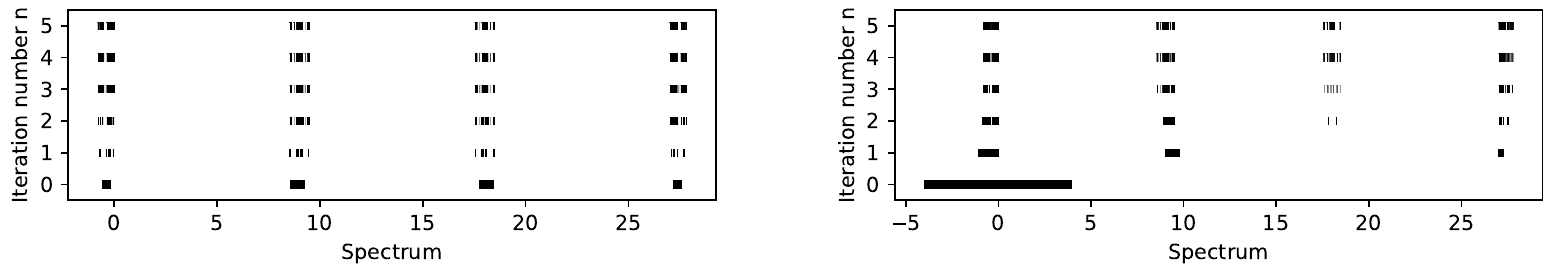}
\caption{
For the chair tiling substitution, we compute the spectra of the Schr\"odinger operators with single-letter potential using Floquet--Bloch theory. The left panel shows $\sigma(H_{S^n(\rho_{\mathrm{leg}})})$, while the right panel shows $\sigma(H_{S^n(\omega_0)})$, see \cite{BeTe26-Numer}.
}
\label{Fig:SpectraChair}
\end{figure}

We first observe that
\[
W(S^n(\omega_0))_{2\times 2} = W(S^{n+1}(\omega_0))_{2\times 2}
	= W(S)_{2\times 2}\cup\left\{ \block{g}{r}{r}{y} \right\}
\quad\text{for all } n\ge N_0=4.
\]
Since the chair tiling substitution expands in all directions, \eqref{Eq:Chair_CentralPatch} implies that $S^n(\omega_0)$ converges in $\Aa^{\Z^2}$. We denote its limit by $\rho_0$. 
The present example falls within the framework of Theorem~\ref{Thm:SubshiftPartialLimit}, Proposition~\ref{Prop:Tperiodic_Inclusion}, and Proposition~\ref{Prop:Central-illegal_FixedPoints}. Combined with the preceding considerations, this yields
\[
\lim_{n\to\infty} \overline{\Orb(S^n(\omega_0))}
= \Omega_0
= \overline{\Orb(\rho_0)},
\]
see Corollary~\ref{Cor:ChairTiling_Summary} for details. It is noteworthy that the iterative subshifts converge independently of $\omega_0\in\Aa^{\Z^2}$, see Example~\ref{Ex:Chair_NumberLimits} for details. Moreover, the limit generally differs from $\Omega(S)$.

By the computations in Appendix~\ref{App:ChairTilling}, every patch of $\rho_0$ that is contained in a half-space determined by the coordinate axes is $S$-legal\footnote{We argue as in the table tiling case, where patches $P\prec\rho_j|_{\Upsilon_\pm}$ were shown to be $S$-legal.}, and the only $S$-illegal $2\times2$-patch is the central patch $\rho_0|_T$ with $T=\{-1,0\}^2$.
Consequently, for all $\omega\in\Omega(S)$, we show
\[
\sigma_{\mathrm{ess}}(H_{\rho_0}) = \sigma_{\mathrm{ess}}(H_\omega),
\qquad\text{and}\qquad
\mathrm{Leb}\bigl(\sigma(H_{\rho_0})\Delta\sigma(H_\omega)\bigr)=0,
\]
see Corollary~\ref{Cor:ChairTiling_Summary}. Here $A\Delta B := (A\setminus B)\cup(B\setminus A)$ denotes the symmetric difference.

%%%%%%%%%%%%%%%%%%%%%%%%%%%%%%%%%%%%%%%%%%%%%%%%%%%%%%%%%%%%%%%%%%%%%%%%%
\subsection{Summary and outlook}
\label{Subsec:SummaryOutlook}
%%%%%%%%%%%%%%%%%%%%%%%%%%%%%%%%%%%%%%%%%%%%%%%%%%%%%%%%%%%%%%%%%%%%%%%%%

The numerical results for the table tiling (Figure~\ref{Fig:SpectraTable}) indicate that the limits of the (essential) spectra of $H_{S^n(\omega_0)}$ -- and even their Lebesgue measure -- may differ from the spectrum of $H_\omega$ for $\omega\in\Omega(S)$. This contrasts with the behavior of one-dimensional substitution models, where the limits of the essential spectrum, and hence the Lebesgue measure of iterated spectra, is independent of the initial configuration, see Remark~\ref{Rem-EssSpec-OneDim}. In contrast, the numerical results for the chair tiling (Figure~\ref{Fig:SpectraChair}) do not indicate such a change.

Since little is known about the spectral theory of $H_\omega$ for $\omega\in\Omega(S)$ in higher dimensions, we analyze these phenomena in a systematic way. Our results show that the choice of periodic approximations is significantly more delicate: unsuitable choices may lead to distinct limiting spectral behavior.

Our approach builds on \cite{BBdN18,BeBeCo19,BecTak25,BaBePoTe24}, which relate convergence of spectra to convergence of the underlying subshifts. In Section~\ref{Sec:SubshiftLimit}, we prove that the sequence
\[
\Theta_n := \overline{\Orb\big(S^n(\omega_0)\big)} \in\Jj, \qquad n\in\N,
\]
admits only finitely many limit points in $\Jj$ and converges along arithmetic subsequences. As illustrated by the table tiling (Section~\ref{Subsec:TableTiling}) and the chair tiling (Section~\ref{Subsec:ChairTiling}), a description of these limits via fixed points provides an effective tool to analyze spectral pollution.

The corresponding spectral consequences are derived in Section~\ref{Sec:SpectralResults} under suitable structural assumptions on these fixed points. Therefore, we provide sufficient conditions ensuring that the limit subshifts of $(\Theta_n)$ are generated by fixed points of the substitution in Section~\ref{Sec:Limits_FixedPoints}. These conditions are verified for two classes of substitutions, represented by the table and chair tiling, in Section~\ref{Sec:TwoClasses}. Further structural properties relevant for the analysis are developed in Section~\ref{Sec:T-patch_graphs}.
We refer the reader to the Phd thesis \cite{TenenbaumThesis25} for more examples and tools to study structural defects.

Beyond the spectral perspective, structural defects play a fundamental role in determining material properties in the physics literature \cite{QC2002,Kle03}, but have received comparatively little attention from a mathematical viewpoint. In the context of aperiodic order \cite{BaakeGrimm13,DaFi24-book_2}, the one-dimensional model of Sturmian Hamiltonians (or the Kohmoto model) plays a central role. For this model, spectral pollution has recently been analyzed in \cite{BecBelTho25} via a detailed study of structural defects, leading to further insights into the associated Kohmoto butterfly.

\section*{Acknowledgements}
We are grateful to Philipp Bartmann, Alan Lew and Yannik Thomas for insightful discussions. S.B. expresses gratitude to Daniel Riveline for insightful discussions on defects in physics and for pointing out relevant references related to this subject. 
The research for this article was conducted at the Israel Institute for Advanced Studies, as part of the Research Group \textit{Group Analysis, Geometry, and Spectral Theory of Graphs} (2025). R.B., S.B. and F.P. are grateful for the excellent working conditions.
This work was partially supported by the Deutsche Forschungsgemeinschaft [BE 6789/1-1 to S.B.] and [PO 2383/2-1 to F.P.]. R.B.\@ and L.T.\@ were supported by the Israel Science Foundation (ISF Grants No. 844/19 and 2362/25). L.T. was supported by the Israel Science Foundation grant No. 1647/23 (PI B. Solomyak).

%%%%%%%%%%%%%%%%%%%%%%%%%%%%%%%%%%%%%%%%%%%%%%%%%%%%%%%%%%%%%%%%%%%%%%%%%
\section{Spectral results}
\label{Sec:SpectralResults}
%%%%%%%%%%%%%%%%%%%%%%%%%%%%%%%%%%%%%%%%%%%%%%%%%%%%%%%%%%%%%%%%%%%%%%%%%

In this section we derive spectral consequences for Schr\"odinger operators associated with subshifts. 
Assuming that structural defects are described by particular fixed points $\rho\in\Aa^\Gamma$, as illustrated by the table and chair tilings, we analyze the behavior of the essential spectrum and the Lebesgue measure of the spectrum. In particular, we show that both may change or remain stable.

Throughout the remainder of this work, we replace $\Z^d$ by a general countable discrete group $\Gamma$ and write the group operation multiplicatively. 
We further fix a left-invariant proper metric $d$ on $\Gamma$ and 
\[
B(\gamma,r):=\{\eta\in\Gamma \mid d(\gamma,\eta)<r\}
\]
denotes the open ball in $\Gamma$ with center $\gamma\in\Gamma$ and radius $r>0$.

A {\em Schrödinger operator $H$ with finite range} is an operator family $H=(H_\omega)_{\omega\in\Aa^\Gamma}$ of self-adjoint operators 
$H_\omega : \ell^2(\Gamma)\to\ell^2(\Gamma)$, $\omega\in\Aa^\Gamma$, defined by
\begin{equation}\label{eq:Hamiltonian}
(H_\omega\psi)(\gamma)
=
\sum_{\eta\in B}
t_\eta(\gamma^{-1}\omega)\psi(\gamma\eta),
\qquad \psi\in\ell^2(\Gamma),\ \gamma\in\Gamma,
\end{equation}
where $B\subseteq\Gamma$ is finite with $B=B^{-1}$ and $t_\eta:\Aa^\Gamma\to\C$ are continuous and satisfy
$
t_\eta(\omega)
=
\overline{t_{\eta^{-1}}(\eta\omega)}.
$
This guarantees self-adjointness of $H_\omega$ for all $\omega$. We call $(t_\eta)_{\eta\in B}$, the \emph{coefficients} of $H$. 

The operators defined in \eqref{Eq:SchrOp_SingleLetter} fit into this framework by setting $t_\eta\equiv 1$ for $\eta=\pm e_j$, where $e_j$ denotes the standard basis of $\Z^2$, and by defining $t_0(\omega)=v(\omega(e))$ with $v:\Aa\to\R$. 
This particular potential $t_0$ is not only continuous, but also belongs to the class of so-called pattern equivariant functions.
A coefficient $t_\eta:\Aa^\Gamma\to\C$ is called \emph{pattern equivariant} if there exists $R_\eta>0$ such that
\[
\text{for all } \omega,\rho\in\Aa^\Gamma : \quad
\omega|_{B(e,R_\eta)}=\rho|_{B(e,R_\eta)}
\quad \Longrightarrow\quad
t_\eta(\omega)=t_\eta(\rho).
\]
If all coefficients $t_\eta$, $\eta\in B$, are pattern equivariant, then $(H_\omega)_{\omega\in\Aa^\Gamma}$ is called a \emph{pattern equivariant Schrödinger operator with finite range}.

Define the unitary operator $U_\gamma:\ell^2(\Gamma)\to\ell^2(\Gamma)$ by $(U_\gamma\psi)(x):=\psi(\gamma^{-1}x)$. Then $H$ is covariant, i.e. 
\[
U_\gamma^{*}H_{\gamma\omega}U_\gamma=H_\omega
\qquad\text{for all }\omega\in\Aa^\Gamma,\ \gamma\in\Gamma.
\]
Moreover, continuity of the coefficients $t_\eta$ implies that the family $H=(H_\omega)_{\omega\in\Aa^\Gamma}$ is strongly continuous in $\omega\in\Aa^\Gamma$. For a subshift $\Omega\subseteq\Aa^\Gamma$, define the compact set
\[
\sigma(H_\Omega):=\overline{\bigcup_{\omega\in\Omega}\sigma(H_\omega)} \subseteq \R.
\]
If $\Omega=\overline{\Orb(\omega)}$, then covariance and strong continuity yield $\sigma(H_\Omega)=\sigma(H_\omega)$ by standard arguments, see e.g. \cite[Cor 1.4.22]{DaFi22-book_1} or \cite[Lem~3.2]{BecTak25}.

The table and the chair tiling substitutions are primitive and hence their associated subshifts are minimal \cite[Ch.~4]{BaakeGrimm13}.
Here, a subshift $\Omega$ is called \emph{minimal} if $\overline{\Orb(\omega)}=\Omega$ for every $\omega\in\Omega$. Moreover, there exist many substitution families beyond $\Z^d$ that share this minimality property if the substitution is primitive, see \cite{BHP21-Symbolic}. 
For a minimal subshift $\Omega$, \cite[Cor.~4.5]{BeLeLiSe17} asserts
\begin{equation}
\label{Eq:essSpectrum_Min}
\sigma(H_\Omega)=\sigma(H_\omega)=\sigma_{\mathrm{ess}}(H_\omega), \qquad \text{for all } \omega\in\Omega(S).
\end{equation}

The convergence of subshifts in the Hausdorff metric on $\Jj$ is tightly connected to the spectral convergence as discovered in \cite{BBdN18}.

\begin{proposition}[\cite{BBdN18}, Cor.~1]
\label{Prop:ConvSubsh-Spectra}
Let $\Gamma$ be a countable and amenable group. Then the following are equivalent.
\begin{enumerate}[(i)]
\item The subshifts $(\Omega_n)_{n\in\N}$ converge in $\Jj$ to $\Omega$.
\item For each Schrödinger operator of finite range $H=(H_\omega)_{\omega\in\Aa^\Gamma}$, the spectra $(\sigma(H_{\Omega_n}))_{n\in\N}$ converge in the Hausdorff metric on $\R$ to $\sigma(H_\Omega)$.
\end{enumerate}
\end{proposition}

Under suitable additional assumptions, quantitative results are available \cite{BeBeCo19,BecTak25}. In particular, \cite{BaBePoTe24} proves exponential convergence of $\bigl(\overline{\Orb(S^n(\omega_0))}\bigr)_n$ to $\Omega(S)$ whenever convergence occurs. Here we consider the complementary case where $\bigl(\overline{\Orb(S^n(\omega_0))}\bigr)_n$ does not converge to $\Omega(S)$ but admits finitely many limit points. Straightforward modifications of \cite{BaBePoTe24} show that the corresponding subsequences converge exponentially fast to these limit points, and hence so do the associated spectra by \cite{BecTak25}. Consequently, the spectral pollution stabilizes exponentially fast along the approximations.

In this spirit, we continue by showing how the description of structural defects via suitable $S$-fixed points leads to spectral consequences for the associated Schrödinger operators.

\subsection{The essential spectrum}
\label{Subsec:EssentialSpectr}

Our first result is applicable to the chair tiling and shows that, despite the presence of structural defects, the essential spectrum remains unchanged.

For this statement, we assume that $\Gamma$ is of \emph{strict polynomial growth} with respect to the left-invariant proper metric $d$ on $\Gamma$. More precisely, there are constants $c_0,c_1>0$ and $d\ge 1$ such that
\[
c_0 r^d \le \sharp B(e,r) \le c_1 r^d,\qquad r\ge 1.
\]
This allows us to employ the cut-off functions introduced in \cite{BecTak25}. It can be relaxed to amenable groups by adapting the construction of the cut-off functions accordingly, see e.g. \cite[Sec.~4.6]{BeckusHabil}.

\begin{theorem} 
\label{Thm:essspec-unchanged}
	Let $H=(H_{\omega})_{\omega\in\Aa^\Gamma}$ be a pattern equivariant Schrödinger operator of finite range and $\Gamma$ be of strict polynomial growth.
	Let $\Omega\in\Jj$ be minimal and suppose that $\rho \in \mathcal{A}^{\Gamma}$ is an $S$-fixed point that satisfies the following condition: there is a finite set $T \subseteq \Gamma$ such that for all $x \in \Gamma$ and $r \geq 1$ with $B(x,r) \cap T = \emptyset$, we have $\rho|_{B(x,r)} \in W(\Omega)$. Then 
\[
	\sigma(H_\Omega)=\sigma_{ess}(H_{w}) = \sigma_{ess}(H_{\rho}),\qquad \text{for all } \omega\in\Omega.
\] 
\end{theorem}

For primitive block substitutions on $\Z^d$, purely $T$-illegality (Definition~\ref{Def:central illegal}) yields a description of the iterated limit subshifts by fixed points $\rho$, see Proposition~\ref{Prop:Central-illegal_FixedPoints}. By Lemma~\ref{Lem:Central-illegal_condition}, these fixed points satisfy the assumption of Theorem~\ref{Thm:essspec-unchanged}, which we now illustrate for the chair tiling before proving the theorem.

\begin{corollary}
\label{Cor:ChairTiling_Summary}
Let $S:\PatZ\to\PatZ$ be the chair tiling substitution, with subshift $\Omega(S)$, and let $T=\{-1,0\}^2$. Let $H=(H_\omega)_{\omega\in\Aa^{\Z^2}}$ be a finite-range Schr\"odinger operator. Then for every $\omega_0\in\Aa^{\Z^2}$ the following assertions hold. 
\begin{enumerate}[(a)]
\item There exist a $k=k(\omega_0)\in\N$ and finitely many $S$-fixed points $\rho_1,\ldots,\rho_k\in \Aa^{\Z^2}$ satisfying 
\[
x\in\Z^2,\ r\ge 1,\ B(x,r)\cap T=\emptyset \ \Longrightarrow\  \rho_j|_{B(x,r)}\in W(S) \, \text{ for all } j\in\{ 1,\ldots, k \},
\]
and
\[
\lim_{n\to\infty} \overline{\Orb(S^n(\omega_0))} = \Omega_0= \bigcup_{j=1}^k\overline{\Orb(\rho_j)}. 
\]

\item For each $\rho\in\Omega_0$ and all $\omega\in\Omega(S)$,
\[
\sigma_{\mathrm{ess}}(H_{\rho})=\sigma_{\mathrm{ess}}(H_\omega)
\qquad\text{and}\qquad
\mathrm{Leb}\bigl(\sigma(H_{\rho})\Delta\sigma(H_\omega)\bigr)=0.
\]
\end{enumerate}
\end{corollary}

\begin{remark}
Although every patch of $\rho_0\in\Aa^{\Z^2}$ supported on a ball disjoint from $T$ is $S$-legal for the chair tiling, patches of $\rho_0$ supported on annuli $B(0,r_1)\setminus B(0,r_2)$ with $r_1>r_2$ need not be $S$-legal. A priori, such patches could contribute to the essential spectrum, which is, however, not the case.
\end{remark}

\begin{proof}
Since the chair tiling substitution is primitive, the subshift $\Omega(S)$ is minimal \cite[Cor.~4.9]{BaakeGrimm13}. Consequently, both the spectrum and the essential spectrum of $H_\omega$ are independent of $\omega\in\Omega(S)$, see \eqref{Eq:essSpectrum_Min}.
By Lemma~\ref{Lem:Chair_CentralIllegal}, the substitution is purely $T$-illegal. Assertions (a) therefore follows from Theorem~\ref{Thm:SubshiftPartialLimit}, Proposition~\ref{Prop:Central-illegal_FixedPoints}, Lemma~\ref{Lem:Central-illegal_condition} and Example~\ref{Ex:Chair_NumberLimits}. For (b) note first that it suffices to restrict to $\rho=\rho_j$ for each $j\in\{ 1,\ldots, k \}$. Thus, (b) follows from (a) together with Theorem~\ref{Thm:essspec-unchanged} and the fact that the discrete spectrum of a self-adjoint operator has Lebesgue measure zero.
\end{proof}

We now show two auxiliary lemmas from which Theorem~\ref{Thm:essspec-unchanged} follows.

\begin{lemma}
\label{Lem:EssSpectr_InclInOmega(S)}
Let $H$ be a pattern equivariant Schrödinger operator of finite range and $\Gamma$ of strict polynomial growth.
Suppose $\omega,\rho\in\Aa^\Gamma$ are such that for all $x\in\Gamma$ and $r\ge 1$ with $B(x,r)\cap T=\emptyset$, $\rho|_{B(x,r)}\in W(\omega)$. Then
\[
\sigma_{\text{ess}}(H_\rho) \subseteq \sigma(H_\omega).
\]
\end{lemma}

\begin{proof}
Let $R_0>0$ be such that $T\subseteq B(e,R_0)$. Since $H$ is a pattern equivariant Schr\"odinger operator of finite range, there is an $R_H>0$ such that for all $\xi,\zeta\in\Aa^\Gamma$ and $\psi\in\ell^2(\Gamma)$, 
\[
\xi|_{B(\gamma,R_H)}=\zeta|_{B(\gamma,R_H)}
\ \Rightarrow\
(H_\xi\psi)(\gamma)=(H_\zeta\psi)(\gamma).
\]
In the following we use that, in the terminology of \cite{BecTak25}, $H$ is a dynamically-defined operator with kernel $k:\Gamma\times\Aa^\Gamma\to\C$ given by
\[
k(\eta,\omega)
=
\begin{cases}
t_\eta(\omega), & \eta\in B,\\
0, & \text{otherwise}.
\end{cases}
\]
Let $E\in\sigma_{\mathrm{ess}}(H_\rho)$. By Weyl's criterion there exist $0\neq\psi_n\in\ell^2(\Gamma)$ and $\varepsilon_n\to 0$ such that
\[
\supp(\psi_n)\cap B(e,2n+R_H+R_0)=\emptyset,
\qquad
\|(H_\rho-E)\psi_n\|\leq \varepsilon_n\|\psi_n\|,
\]
where $\|\cdot\|$ denotes the $\ell^2$-norm on $\ell^2(\Gamma)$. Fix $n$ and set
\[
\chi:=\chi_n:\Gamma\to[0,1], 
\qquad
\chi(\gamma)
:=
\left(\frac{n-d(e,\gamma)}{n}\right)\mathbf{1}_{B(e,n)}(\gamma).
\]
For $\eta\in\Gamma$, define the shifted functions $\chi_\eta(g):=\chi(\eta^{-1}g)$ and $\chi^\eta(g):=\chi(g\eta)$. Since $\Gamma$ is of strict polynomial growth, there exists a constant $C>0$ such that
\begin{equation}
\label{Eq:CutOff_Estimate}
\frac{\|\chi^\eta-\chi\|}{\|\chi\|}
\le 
d(e,\eta)\frac{C}{n},
\end{equation}
where $C>0$ is independent of $n$, see \cite[Lem.~4.5]{BecTak25}. We use $\chi$ as a cut-off function by implementing it as the multiplication operator $\hat\chi\in\Ll(\ell^2(\Gamma))$ defined by $(\hat\chi\psi)(\gamma)=\chi(\gamma)\psi(\gamma)$.

Writing $[A,B]=AB-BA$ for the commutator of operators $A$ and $B$, we obtain
\[
\sum_{\gamma\in\Gamma}\|(H_\rho-E)\hat\chi_\gamma\psi_n\|^2
\leq
2\sum_{\gamma\in\Gamma}\|[H_\rho,\hat\chi_\gamma]\psi_n\|^2
+
2\sum_{\gamma\in\Gamma}\|\hat\chi_\gamma(H_\rho-E)\psi_n\|^2.
\]
By \cite[Lem.~4.2]{BecTak25},
\[
\sum_{\gamma\in\Gamma}\left\|\hat{\chi}_\gamma(H_\rho-E)\psi_n\right\|^2
=
\|\chi\|^2 \|(H_\rho-E)\psi_n\|^2
\le
\varepsilon_n^2 \|\chi\|^2 \|\psi_n\|^2
=
\varepsilon_n^2
\sum_{\gamma\in\Gamma}\|\hat{\chi}_\gamma\psi_n\|^2.
\]
Using again \cite[Lem.~4.2]{BecTak25} and the kernel representation,
\[
\sum_{\gamma\in\Gamma}\|[H_\rho,\hat\chi_\gamma]\psi_n\|^2
\le
\left(
\sum_{\eta\in B}
\|k(\eta^{-1},(\cdot)^{-1}\rho)\|_\infty
\frac{\|\chi^\eta-\chi\|}{\|\chi\|}
\right)^2
\sum_{\gamma\in\Gamma}\|\hat\chi_\gamma\psi_n\|^2,
\]
where 
$\|k(\eta^{-1},(\cdot)^{-1}\rho)\|_\infty
:=
\sup_{\gamma\in\Gamma}|t_{\eta^{-1}}(\gamma^{-1}\rho)|
<\infty$.
Hence,
\[
\sum_{\gamma\in\Gamma}\|(H_\rho-E)\hat\chi_\gamma\psi_n\|^2
\le
2\left(
\left(
\sum_{\eta\in B}
\|k(\eta^{-1},(\cdot)^{-1}\rho)\|_\infty
\frac{\|\chi^\eta-\chi\|}{\|\chi\|}
\right)^2
+
\varepsilon_n^2
\right)
\sum_{\gamma\in\Gamma}\|\hat\chi_\gamma\psi_n\|^2.
\]
Set 
\[
\|H\|_S
:=
\sum_{\eta\in B}
\|k(\eta^{-1},(\cdot)^{-1}\rho)\|_\infty d(e,\eta)
<\infty.
\]
Using \eqref{Eq:CutOff_Estimate}, the previous considerations imply that there exists $\gamma\in\Gamma$ with $\|\hat\chi_\gamma\psi_n\|>0$ and
\[
\|(H_\rho-E)\hat\chi_\gamma\psi_n\|
\le
\delta_n\|\hat\chi_\gamma\psi_n\|,
\qquad
\delta_n
=
\sqrt{2}\left(\frac{C\|H\|_S}{n}+\varepsilon_n\right).
\]

Since $\|\hat\chi_\gamma\psi_n\|>0$, the support $\supp(\chi_\gamma)$ intersects $\supp(\psi_n)$. Recalling that
\[
\supp(\psi_n)\cap B(e,2n+R_H+R_0)=\emptyset 
\qquad\text{and}\qquad
\supp(\chi_\gamma)\subseteq B(\gamma,n),
\]
we obtain $B(\gamma,n+R_H)\cap T=\emptyset$. Hence, $\rho|_{B(\gamma,n+R_H)} \in W(\omega)$ follows from our assumptions. Thus, there exists $\gamma_n\in\Gamma$ with 
\[
(\gamma_n\omega)|_{B(\gamma,n+R_H)}
=
\rho|_{B(\gamma,n+R_H)}.
\]
By the choice of $R_H$ and $\supp(\hat\chi_\gamma\psi_n)\subseteq B(\gamma,n)$, we conclude
\[
\|(H_{\gamma_n\omega}-E)\hat\chi_\gamma\psi_n\|
=
\|(H_\rho-E)\hat\chi_\gamma\psi_n\|
\le
\delta_n\|\hat\chi_\gamma\psi_n\|.
\]
Thus, standard arguments \cite[Cor.~1.4.21]{DaFi22-book_1} and covariance of the operator family lead to
\[
\bigl\{ y\in\R \mid |E-y|<\delta_n \bigr\} \cap\sigma(H_\omega)
=
\bigl\{ y\in\R \mid |E-y|<\delta_n \bigr\} \cap\sigma(H_{\gamma_n\omega})
\neq\emptyset.
\]
Since $\delta_n\to 0$, we conclude $E\in\sigma(H_\omega)$.
\end{proof}

\begin{lemma}
\label{Lem:EssSpectr_InclInRho}
Let $H$ be a pattern equivariant Schr\"odinger operator of finite range. Let $\Omega\in\Jj$ be a minimal subshift. Suppose that $\rho\in\Aa^\Gamma$ satisfies: there exists a finite set $T\subseteq\Gamma$ such that for all $x\in\Gamma$ and $r\ge1$ with $B(x,r)\cap T=\emptyset$ one has $\rho|_{B(x,r)}\in W(\Omega)$. Then
\[
\sigma_{\mathrm{ess}}(H_\omega)\subseteq \sigma_{\mathrm{ess}}(H_\rho), \qquad \text{for all }\omega\in\Omega.
\]
\end{lemma}

\begin{proof}
Since $H$ is a pattern equivariant Schr\"odinger operator of finite range, there is an $R_H>0$ such that for all $\xi,\zeta\in\Aa^\Gamma$ and $\psi\in\ell^2(\Gamma)$,
\[
\xi|_{B(\gamma,R_H)}=\zeta|_{B(\gamma,R_H)}
\ \Rightarrow\
(H_\xi\psi)(\gamma)=(H_\zeta\psi)(\gamma).
\]

Fix $E\in\sigma_{\mathrm{ess}}(H_\omega)$. By Weyl's criterion there exist $\psi_n\in C_c(\Gamma)$ with $\|\psi_n\|=1$, $\|(H_\omega-E)\psi_n\|\to0$ and $(\psi_n)$ tends weakly to zero. Choose $x_n\in\Gamma$ and $r_n\ge0$ such that $\supp(\psi_n)\subseteq B(x_n,r_n)$.

We further note that minimality of $\Omega$ implies that for every patch $P\in W(\Omega)$, there is a $R_P>0$ such that every $Q\in W(\Omega)_{B(e,R_P)}$ satisfies $P\prec Q$. Since every patch $P$ of $\rho$ supported on a ball disjoint from $T$ satisfies $P\in W(\Omega)$, there exists $\gamma_n\in \Gamma \setminus B\bigl(e,\,n+r_n+d(e,x_n)+R_H\bigr)$ such that
\[
(\gamma_n^{-1}\rho)|_{B(x_n,r_n+R_H)}
=
\omega|_{B(x_n,r_n+R_H)}.
\]
Since $\supp(\psi_n)\subseteq B(x_n,r_n)$, the finite range property yields $H_{\gamma_n^{-1}\rho}\psi_n=H_\omega\psi_n$.
By covariance,
\[
\|(H_\rho-E)U_{\gamma_n}\psi_n\|
%=\|U_{\gamma_n}(H_{\gamma_n^{-1}\rho}-E)\psi_n\|
=
\|(H_{\gamma_n^{-1}\rho}-E)\psi_n\|
=\|(H_\omega-E)\psi_n\|\to0,
\]
where $U_\eta$ denotes the unitary operator on $\ell^2(\Gamma)$ given by $(U_\eta\psi)(\gamma):=\psi(\eta^{-1}\gamma)$.
Moreover,
\[
\supp(U_{\gamma_n}\psi_n)=\gamma_n\supp(\psi_n)\subseteq \gamma_nB(x_n,r_n),
\]
and since $\gamma_n\notin B(e,n+r_n+d(e,x_n)+R_H)$, the supports escape to infinity. In particular, $(U_{\gamma_n}\psi_n)$ tends weakly to zero. Thus, $E\in\sigma_{\mathrm{ess}}(H_\rho)$.
\end{proof}

\begin{proof}[Proof of Theorem~\ref{Thm:essspec-unchanged}]
The inclusion $\subseteq$ follows from Lemma~\ref{Lem:EssSpectr_InclInRho}. Since $\Omega$ is minimal, we have $\sigma(H_\Omega)=\sigma_{\mathrm{ess}}(H_\omega)$ for all $\omega\in\Omega$ by \eqref{Eq:essSpectrum_Min}. Thus, the converse inclusion follows from Lemma~\ref{Lem:EssSpectr_InclInOmega(S)}.
\end{proof}

\begin{remark}
\label{Rem-EssSpec-OneDim}
For one-dimensional models with $\Gamma=\Z$, Theorem~\ref{Thm:essspec-unchanged} implies that all limit points of a primitive substitution share the same essential spectrum: Primitive substitutions over an alphabet with at least two letters are expanding, and the iterates $S^n(a)$ are $S$-legal by construction. Thus, non-$S$-legal patches can only occur at boundaries. Consequently, every limit point $\Omega_j$ of $\bigl(\overline{\Orb(S^n(\omega_0))}\bigr)_n$ satisfies the assumptions of Theorem~\ref{Thm:essspec-unchanged} with $T=\{-1,0\}$.

Hence, for primitive one-dimensional substitutions, the essential spectrum of the Schr\"odinger operator associated with any limit point coincides with the essential spectrum associated with $\Omega(S)$.
\end{remark}

In contrast to Theorem~\ref{Thm:essspec-unchanged}, we show that the essential spectrum may change in the presence of an infinite defect. Combined with Section~\ref{Subsec:TableTiling}, this demonstrates that for substitutions on $\Z^d$ with $d>1$, new spectral phenomena can occur: the essential spectrum of the limit subshift may differ from that of $\Omega(S)$. 

\begin{proposition}
\label{Prop:ChangeEssSpectrum}
Let $\Omega\in\Jj$. Then for every Schr\"odinger operator with finite range 
$(H_\omega)_{\omega\in\Aa^\Gamma}$, there exists a Schr\"odinger operator with finite range 
$(\tilde H_\omega)_{\omega\in\Aa^\Gamma}$ such that $\tilde H_\omega = H_\omega$ for all $\omega\in\Omega$ and the following holds:
for each finite set $T\subseteq\Gamma$ and each configuration $\rho\in\Aa^\Gamma$ with an infinite defect, i.e.
\[
\rho|_{\gamma T}\notin W(\Omega)
\qquad\text{for infinitely many } \gamma\in\Gamma,
\]
we have
\[
%\tilde H_\omega = H_\omega \quad \text{for all }\omega\in\Omega,
%\qquad\text{and}\qquad
\sigma_{\mathrm{ess}}(\tilde H_\omega)\neq \sigma_{\mathrm{ess}}(\tilde H_\rho)
\quad \text{for all }\omega\in\Omega.
\]
\end{proposition}

\begin{proof}
Let $t_\eta$, $\eta\in B$, be the coefficients of $H$ and assume $e\in B$ (otherwise enlarge $B$ by adding $e$ with coefficient $0$). Define $\tilde H$ by $\tilde t_\eta=t_\eta$ for $\eta\neq e$ and
\[
\tilde t_e(\omega)
=
t_e(\omega)
+
2\lambda_0\Bigl(1-1_{W(\Omega)}(\omega|_T)\Bigr),
\qquad\text{ for some }\lambda_0>\sup_{\omega\in\Omega}\|H_\omega\|.
\]
If $\omega\in\Omega$, then $\omega|_T\in W(\Omega)$ and hence $\tilde t_\eta(\omega)=t_\eta(\omega)$ for all $\eta\in B$, so $\tilde H_\omega=H_\omega$.

By assumption there exist infinitely many distinct $\gamma_n\in\Gamma$ with $\rho|_{\gamma_n T}\notin W(\Omega)$. For such $\gamma_n$, 
\[
\langle \tilde H_\rho \delta_{\gamma_n},\delta_{\gamma_n}\rangle
= \langle H_\rho \delta_{\gamma_n},\delta_{\gamma_n}\rangle
	+ 2\lambda_0 \langle \delta_{\gamma_n},\delta_{\gamma_n}\rangle
> \lambda_0.
\]
Since the vectors $\delta_{\gamma_n}$ are pairwise orthogonal, it follows from \cite[Thm.~4.15]{Teschl14} that
\[
(\lambda_0,\infty)\cap\sigma_{\mathrm{ess}}(\tilde H_\rho)\neq\emptyset.
\]
On the other hand, for $\omega\in\Omega$ one has $\|\tilde H_\omega\|=\|H_\omega\|<\lambda_0$, hence $\sigma_{\mathrm{ess}}(\tilde H_\omega)\subseteq(-\lambda_0,\lambda_0)$.
Thus, $\sigma_{\mathrm{ess}}(\tilde H_\omega)\neq \sigma_{\mathrm{ess}}(\tilde H_\rho)$ follows for all $\omega\in\Omega$.
\end{proof}

\subsection{The Lebesgue measure of the spectrum}
\label{Subsec:LebesgueSpectr}

For the table tiling, we demonstrated above that the limiting subshift admits a pure line defect. Restricting to $\Gamma=\Z^d$, we now prove that a pure line defect may lead to a spectral pollution with positive Lebesgue measure.

In the following, $e_j\in\Z^d$ denotes the $j$-th standard basis vector for $j\in\{1,\ldots,d\}$.

\begin{theorem}
\label{Thm:ChangeLebesgueMeas}
Let $\Omega\in\Jj$. Then there exists a Schr\"odinger operator of finite range with $B=\{\pm e_j:1\le j\le d\}\cup\{0\}$ such that the following holds:
If for some finite set $T\subseteq\Z^d$ a configuration $\rho\in\Aa^{\Z^d}$ has a pure line defect, that is, there exists $1\le j\le d$ such that
\[
\rho|_{mT}\notin W(\Omega)
\quad\text{if and only if}\quad
m=te_j \text{ for some } t\in\Z,
\]
then
\[
\mathrm{Leb}\bigl(\sigma(H_\rho)\setminus\sigma(H_\omega)\bigr)>0, 
\qquad \text{ for all } \omega\in\Omega.
\]
\end{theorem}

For the proof of the theorem, we use some ideas of \cite[Sec.~6]{DaGo11}. We need some preparation on tensor products of operators, see \cite[Sections~II.4 and~VIII.10]{ReedSimonI} and \cite[Sec.~4.6]{Teschl14}. 
Let ${H}^a: \ell^2(\Z^a) \to \ell^2(\Z^a)$ and $H^b: \ell^2(\Z^b) \to \ell^2(\Z^b)$, $a,b \in \N$ be bounded self-adjoint operators. There is some unitary map $\mathcal{U}: \ell^2(\Z^a) \otimes \ell^2(\Z^b) \to \ell^2(\Z^{a+b})$ which is defined on elementary tensor elements by
\[
\mathcal{U}\big( \psi_a \otimes \psi_b \big)(n_a, n_b) \, = \, \psi_a(n_a) \psi_b(n_b), \quad n_a \in \Z^a, \quad n_b \in \Z^b.
\]
Via $\mathcal{U}$, we define a bounded self-adjoint operator $H: \ell^2(\Z^{a+b}) \to \ell^2(\Z^{a+b})$ by
\begin{align} \label{eqn:tensorop}
H = \mathcal{U} \left( H^a \otimes \mathrm{Id}_b \, + \, \mathrm{Id}_a \otimes H^b  \right) \mathcal{U}^{*},
\end{align}
where $\mathrm{Id}_a$ and $\mathrm{Id}_b$ are the identity operators on $\ell^2(\Z^a)$ and $\ell^2(\Z^b)$, respectively.  
 Recall that the tensor product of two bounded self-adjoint operators $A:\ell^2(\Z^a) \to \ell^2(\Z^a)$ and $B:\ell^2(\Z^b) \to \ell^2(\Z^b)$
 is uniquely defined by the equality
\[
\big(A \otimes B \big)(\psi_a \otimes \psi_b) = A \psi_a \otimes B \psi_b, \quad \psi_a \in \ell^2(\Z^a), \quad \psi_b \in \ell^2(\Z^b).
\]
see \cite[Section~VIII.10]{ReedSimonI}. 

\begin{proof}[Proof of Theorem~\ref{Thm:ChangeLebesgueMeas}] 
It is no loss of generality in assuming $j=1$.
We define the Schrödinger operator of finite range $H=(H_\xi)_{\xi\in\Aa^\Gamma}$ by the coefficients $t_\gamma\equiv 1$ for $\gamma\in\{\pm e_j:1\le j\le d\}$ and 
\[
t_0:\mathcal{A}^{\Z^d} \to \R,\qquad
t_0(\omega) = \lambda_0 1_{\Aa^T\setminus W(\Omega)}(\omega|_T), 
\]
where $1_{\Aa^T\setminus W(\Omega)}$ is the characteristic function of the set $\Aa^T\setminus W(\Omega)$. We will chose $\lambda_0 > 0$ later in dependence of $d$. 
Using the assumption on $\rho$ to have a pure line defect, we obtain 
\begin{align*}
	t_0(\gamma^{-1}\rho) \psi (\gamma) = 
	\begin{cases}
		\lambda_0 \psi(\gamma), & \gamma \in \Z e_1 \\
		0, & \mathrm{else}.
	\end{cases}
\end{align*}
With this at hand, it is not hard to see that  $H_{\rho}$ is equivalent to the tensor product $ H_{\rho}^1 \otimes H_{\rho}^2$, 
where
\begin{align*}
H^{1}_{\rho}: \ell^2(\Z) \to \ell^2(\Z), \quad H^{1}_{\rho} \psi_1(t)  &= \psi_1(t+1) + \psi_1(t-1), \\
 H^{2}_{\rho}: \ell^2(\Z^{d-1}) \to \ell^2(\Z^{d-1}), \quad H^{2}_{\rho} \psi_{d-1}(m)  &= \sum_{l=1}^{d-1} \Big( \psi_{d-1}(m+f_{l}) + \psi_{d-1}(m-f_{l}) \Big)  + \lambda \delta_0(m) \psi_{d-1}(m),
%\begin{cases}
%		\varphi(n+1) + \varphi(n-1), & i = j, \\
%	\varphi(n+1) + \varphi(n-1) + \lambda_0 \delta_0(n) \varphi(n), & i \neq j.
%\end{cases}
\end{align*}
and the $f_{l}$ are the standard unit vectors in $\Z^{d-1}$. 
By \cite[Theorem~VIII.33]{ReedSimonI}, we obtain that 
\[
\sigma(H_{\rho}) = \sigma(H_{\rho}^1) + \sigma(H_{\rho}^2).
\]
It is well-known that $\sigma(H_{\rho}^1) = [-2,2]$, and it is easy to see that $\lambda_0$ can be chosen large enough such that $\sigma(H_{\rho}^2)$ contains an element $\theta > 2d+2$. Moreover, $H_{\omega}$ is clearly the unperturbed adjacency operator, and it is well-known that $\sigma(H_{\omega}) = [-2d,2d]$. Putting everything together, we arrive at
\[
	\sigma(H_{\rho}) \setminus \sigma(H_{\omega}) \supseteq  [\theta-2, \theta + 2] \setminus [-2d,2d] = [\theta-2, \theta+ 2].
	\hfill\qedhere
\]
\end{proof}

This result connects to \cite[Thm.~1.1]{Fil25}. It will be interesting to investigate this further.

\begin{corollary}
\label{Cor:TableTiling_Summary}
Let $S:\PatZ\to\PatZ$ be the table tiling substitution, with subshift $\Omega(S)$, and let $T=\{-1,0\}^2$. Let $\omega_a\in\Aa^{\Z^2}$ be the constant configuration $\omega_0(n)=a$ where $a\in\Aa$ is some fixed letter. Let $H=(H_\omega)_{\omega\in\Aa^{\Z^2}}$ be a finite-range Schr\"odinger operator. Then the following assertions hold.
\begin{enumerate}[(a)]
\item There exist subshifts $\Omega_0$ and $\Omega_1$ and $\rho_j\in\Omega_j, \, j \in\{0,1\}$, such that  for each $j\in\{0,1\}$
\[
\lim_{n\to\infty}\Theta_{2n+j}=\Omega_j = \overline{\Orb(\rho_j)}
\]
and $S^2(\rho_j)=\rho_j$ and $\rho_j$ has a pure line defect.

\item There exists a finite-range Schrödinger operator $\tilde H$ such that for each $j\in \{0,1\}$ and $\omega\in\Omega(S)$,
\[
\sigma_{\mathrm{ess}}(\tilde H_{\rho_j}) \neq \sigma_{\mathrm{ess}}(\tilde H_\omega) = \sigma_{\mathrm{ess}}(H_\omega).
\]
\item There exists a finite-range Schrödinger operator $\hat H$ such that for each $j\in \{0,1\}$ and $\omega\in\Omega(S)$,
\[
\mathrm{Leb}\bigl(\sigma(\hat H_{\rho_j})\setminus \sigma(\hat H_\omega)\bigr)>0.
\]
\end{enumerate}
\end{corollary}

\begin{proof}
Since the table tiling substitution is primitive, the subshift $\Omega(S)$ is minimal \cite[Ch.~4]{BaakeGrimm13}. Consequently, both the spectrum and the essential spectrum of $H_\omega$ are independent of $\omega\in\Omega(S)$, see \eqref{Eq:essSpectrum_Min}.

Since $S_0$ is componentwise bijective, the table tiling substitution is $T$-bijective by Corollary~\ref{Cor:Table_Tbijective} for $T=\{-1,0\}^2$.
Therefore, (a) follows from Theorem~\ref{Thm:SubshiftPartialLimit}, Proposition~\ref{Prop:Tperiodic_Inclusion}, and Proposition~\ref{Prop:T-bijective_FixedPoints} combined with the considerations in Section~\ref{Subsec:TableTiling}. One representative case is treated there for $a=\letterbox{r}$, and the remaining cases are analogous. Moreover, similar computations show that for each $a\in\Aa$ the configuration $\rho_j$ admits a pure line defect either parallel to the $x$-axis or along the $y$-axis.

Statement (b) follows from (a) together with Proposition~\ref{Prop:ChangeEssSpectrum}. Finally, (c) follows from (a) and Theorem~\ref{Thm:ChangeLebesgueMeas}, using that $\rho_j$ exhibits a pure line defect.
\end{proof}

%%%%%%%%%%%%%%%%%%%%%%%%%%%%%%%%%%%%%%%%%%%%%%%%%%%%%%%%%%%%%%%%%%%%%%%%%%%%%%%%
\section{Structural defects for substitutions}
\label{Sec:SubshiftLimit}
%%%%%%%%%%%%%%%%%%%%%%%%%%%%%%%%%%%%%%%%%%%%%%%%%%%%%%%%%%%%%%%%%%%%%%%%%%%%%%%%

We collect basic terminologies on substitutions needed for the questions raised in
Section~\ref{Sec:Intro} and study the limit points of iterative subshifts. 

For readability, the reader may keep block substitutions on
$\Gamma=\Z^d$ in mind. The present notation is tailored to the more general framework of
\cite{BHP21-Symbolic} and covers further examples beyond block substitutions (e.g. particular digit substitutions) and even beyond the abelian setting.

A \emph{substitution rule} on a finite alphabet $\Aa$ with support $F\subseteq\Gamma$
is a map
\[
S_0:\Aa\to\Aa^F.
\]
We restrict to substitution rules that extend to patches by letterwise application
compatible with the group structure of $\Gamma$. Let $e\in\Gamma$ denote the neutral element.

\begin{definition}
\label{Def:SubstitutionMap}
A map $S:\Pat\to\Pat$ is called a \emph{substitution map} with respect to $S_0$ if:
\begin{enumerate}[(a)]
\item $S(P_a)=S_0(a)$ for all $a\in\Aa$, where $P_a:\{e\}\to\Aa$ is the single-letter patch
given by $P_a(e)=a$,
\item $S(\gamma P)=D(\gamma)\,S(P)$ for all $\gamma\in\Gamma$ and $P\in\Pat$, where
$D:\Gamma\to\Gamma$ is a group homomorphism such that 
\begin{equation}
\label{Eq:FundCell_D-F}
\Gamma=\bigsqcup_{\gamma\in\Gamma} D(\gamma)F,
\end{equation}
\item for all $P\in\Pat$ and all finite $M\subseteq\supp(P)$, the support $F(n,M)\subseteq\Gamma$ of $S^n\bigl(P|_M\bigr)$ satisfies
\[
\bigl(S^n(P)\bigr)\big|_{F(n,M)} = S^n\bigl(P|_M\bigr),
\qquad n\in\N_0.
\]
\end{enumerate}
\end{definition}

For the two examples from Section~\ref{Sec:Intro} with $\Gamma=\Z^2$, we take
$F:=\{0,1\}^2$ and $D:\Z^2\to\Z^2$ given by $D(x,y)=(2x,2y)$. In fact every block substitution admits such a unique substitution map \cite[Prop.~3.1]{BaBePoTe24} as well as particular digit substitutions \cite{Vin00,FraMan22}. We refer to \cite{BHP21-Symbolic,BeckusHabil} for further examples beyond the
abelian setting such as the Heisenberg group.

\begin{proposition}
\label{Prop:Exist_SubstMap}
Let $S_0:\Aa\to\Aa^F$ be a substitution rule with $F\subseteq \Gamma$ finite. If there is a group homomorphism $D:\Gamma\to\Gamma$ satisfying \eqref{Eq:FundCell_D-F}, then there exists a unique substitution  map $S:\Pat\to\Pat$ and 
\[
S:\Pat\to\Pat,
\qquad
S(P)(\gamma)
=
S_0\bigl(P(\eta_\gamma)\bigr)\bigl(D(\eta_\gamma)^{-1}\gamma\bigr),
\quad 
\gamma\in F\bigl(1,\supp(P)\bigr).
\]
Here $F\bigl(1,\supp(P)\bigr)=\bigsqcup_{\eta\in\supp(P)} D(\eta)F$ and $\eta_\gamma\in\Gamma$ is the unique element satisfying $D(\eta_\gamma)^{-1}\gamma\in F$. Moreover, the restriction $S:\Aa^\Gamma\to\Aa^\Gamma$ is continuous.
\end{proposition}

\begin{proof}
This is a straightforward modification of \cite[Proposition~3.2 and 3.7]{BHP21-Symbolic}.
\end{proof}

\begin{definition}[Legal patches and subshift]
\label{Def:LegalPatches}
A finite patch $P$ is \emph{$S$-legal} if $P\prec S^n(a)$ for some $n\in\N$ and
$a\in\Aa$. Let $W(S)$ denote the set of all $S$-legal finite patches.
The associated subshift is
\[
\Omega(S):=\{\omega\in\Aa^\Gamma \mid W(\omega)\subseteq W(S)\}.
\]
\end{definition}

The subset $\Omega(S)\subseteq\Aa^\Gamma$ is $\Gamma$-invariant and compact. 
We require $\Omega(S)$ to be non-empty so that it defines a subshift.
Typically, this is ensured by assuming that the support $F$ of $S_0$ expands in all directions. More precisely, for every radius $r$ there exists $n\in\N$ such that the support $F(n,\{e\})$ of $S^n$ applied to a single letter contains a ball of radius $r$.
In the examples considered in Section~\ref{Sec:Intro}, this property is immediate, and we assume $\Omega(S)\neq\emptyset$ without further discussion of the rather subtle conditions. We refer to \cite[Prop.~3.14 and 4.3]{BHP21-Symbolic} for sufficient criteria.

To analyze possible limit points of the iterated subshifts $\bigl(\overline{\Orb(S^n(\omega_0))}\bigr)_n$ in the space $\Jj$,
we introduce a restricted notion of testing domains, originally
introduced in \cite{BaBePoTe24}.

\begin{definition}[Convenient testing domain]
\label{Def:TestingDomain}
A finite set $T\subseteq\Gamma$ is called a \emph{convenient testing domain} if the following
conditions hold: 
\begin{enumerate}[(a)]
\item \label{enu:TestingDomain_simplified}
For every $r>0$ there exists $n_r\in\N$ such that for all $n\ge n_r$ and all
$x\in\Gamma$ there exists $\gamma=\gamma(n,x)\in\Gamma$ with
\[
xB(e,r)\subseteq F(n,\gamma T).
\]
Moreover, we can choose $\gamma(n,e)=e$ for all $n\in\N$.
\item \label{enu:NT_simplified}
For every $x\in\Gamma$ there exists $\gamma=\gamma(x)\in\Gamma$ such that
\[
xT\subseteq F(1,\gamma T).
\]
Moreover, we can choose $\gamma(e)=e$.
\end{enumerate}
\end{definition}

Recalling that $F(n,T)$ is the support of $S^n(P)$ for $P\in\Aa^T$
(Definition~\ref{Def:SubstitutionMap}), these conditions suggest the following
intuition.
Condition~\eqref{enu:TestingDomain_simplified} ensures that all finite patches on
large balls are, up to translation, eventually inside iterates of
$T$-patches, while \eqref{enu:NT_simplified} guarantees that a single substitution
step already 
is enough to contain a translate of $T$.

The notion of a convenient testing domain corresponds to a testing tuple $(T,N_T)$ with $N_T=1$, as defined in \cite[Def.~2.11]{BaBePoTe24}. We restrict our considerations for simplicity. All results extend to general testing tuples at the expense of additional bookkeeping and by replacing $S$ with $S^{N_T}$. 

Many standard examples admit a convenient testing domain, including block substitutions on $\Z^d$ and certain substitutions on the Heisenberg group. In the table and chair tilings discussed in Section~\ref{Sec:Intro}, the set $T=\{-1,0\}^2$ is a convenient testing domain, which follows from straightforward computations using \cite[Prop.~3.2]{BaBePoTe24}. In particular, the role of the $2\times2$-patches arises from their identification with $T$-patches.

%%%%%%%%%%%%%%%%%%%%%%%%%%%%%%%%%%%%%%%%%%%%%%%%%%%%%%%%%%%%%%%%%%%%%%%%%%%%%%%%
\subsection{Limit points of the iterated subshift sequence}
\label{Subsec:SubshiftLimit}
%%%%%%%%%%%%%%%%%%%%%%%%%%%%%%%%%%%%%%%%%%%%%%%%%%%%%%%%%%%%%%%%%%%%%%%%%%%%%%%%
We continue studying limit points of \(\bigl(\Theta_n:=\overline{\Orb(S^n(\omega_0))}\bigr)_{n\in\N}\) in $\Jj$.
For a limit point $\Omega$, we call the patches in $W(\Omega)\setminus W(S)$ the \emph{structural defects} of $\Omega$. 

\begin{definition}[$\Qq$-legality]
\label{Def:Qlegal}
Let $\Qq$ be a nonempty set of finite patches and $\ell\in\N$. 
A patch $P$ is called \emph{$(\Qq,\ell)$-legal} if there exist $n\in\N_0$ and $Q\in\Qq$ such that $P\prec S^{n\ell}(Q)$.
The set of all $(\Qq,\ell)$-legal patches is denoted by $W(\Qq,\ell)$.
A configuration $\omega\in\Aa^\Gamma$ is called \emph{$(\Qq,\ell)$-legal} if $W(\omega)\subseteq W(\Qq,\ell)$. 
\end{definition}

If $\Qq=\{P_a\mid a\in\Aa\}$, then $(\Qq,1)$-legality coincides with the notion of
$S$-legality from Definition~\ref{Def:LegalPatches} and $W(\Qq,1)=W(S)$.

Using this concept, we show that the sequence $(\Theta_n)_{n\in\N}$ admits only finitely
many limit points $\Omega_0,\ldots,\Omega_{\ell_0-1}$, each of which is characterized by $(\Qq,\ell_0)$-legality, in direct
analogy with the subshift $\Omega(S)$ characterized by $S$-legal patches.

\begin{theorem}[Finite set of partial limits]
\label{Thm:SubshiftPartialLimit}
Let $S$ be a substitution map on $\Aa^\Gamma$ with a convenient testing domain $T$.
For each $\omega_0\in\Aa^\Gamma$, there exist $\ell_0\in\N$ and 
$\Qq_0,\ldots,\Qq_{\ell_0-1}\subseteq \Aa^T$ such that
\[
\Theta_{n\ell_0+j}
\;\overset{n\to\infty}{\longrightarrow}\;
\Omega_j
:=
\bigl\{
\omega\in\Aa^\Gamma
\;\big|\;
\omega \text{ is $(\Qq_j,\ell_0)$-legal}
\bigr\}
\]
for each $j\in\{0,\ldots,\ell_0-1\}$.
\end{theorem}

As a consequence of the theorem the structural defects are given by $W(\Qq_j,\ell_0)\setminus W(S)$. 

\begin{corollary}
\label{Cor:FiniteLimit_Spectra}
Let $S$ be a substitution map on $\Aa^\Gamma$ with a convenient testing domain $T$ and let $H=(H_\omega)_{\omega\in\Aa^\Gamma}$ be a finite-range Schr\"odinger operator. For every $\omega_0\in\Aa^\Gamma$, there exist compact sets $\Sigma_0,\ldots,\Sigma_{\ell_0-1}\subseteq\R$ such that
\[
\lim_{n\to\infty} \sigma\bigl(H_{S^{n\ell_0+j}(\omega_0)}\bigr) = \Sigma_j
\qquad\text{for all } 0\le j\le \ell_0-1
\]
in the Hausdorff metric on compact subsets of $\R$.
\end{corollary}

\begin{proof}
This follows directly from Theorem~\ref{Thm:SubshiftPartialLimit} and Proposition~\ref{Prop:ConvSubsh-Spectra}.
\end{proof}

We continue proving Theorem~\ref{Thm:SubshiftPartialLimit}. The key observation is that the dictionary on the convenient testing domain $T$ is eventually
periodic.

\begin{lemma}[Eventual periodicity of the $T$-dictionary]
\label{Lem:EventualPeriodicityT}
Let $S$ be a substitution map on $\Aa^\Gamma$ with a convenient testing domain $T$.
For every $\omega_0\in\Aa^\Gamma$, there exist integers $N_0,\ell_0\in\N$ such that
\[
W\big(S^n(\omega_0)\big)_T
=
W\big(S^{n+\ell_0}(\omega_0)\big)_T
\qquad\text{for all } n\ge N_0.
\]
Moreover, $N_0$ and $\ell_0$ are uniformly bounded by the cardinality $\#\,2^{\Aa^T}$.
\end{lemma}

\begin{definition}
\label{Def:Q_j}
Let $S$ be a substitution map on $\Aa^\Gamma$ with a convenient testing domain $T$.
For $\omega_0\in\Aa^\Gamma$ and integers $N_0,\ell_0\in\N$ as in
Lemma~\ref{Lem:EventualPeriodicityT}, define the \emph{associated patch sets} by
\[
\Qq_j
:=
W\big(S^{N_0\ell_0+j}(\omega_0)\big)_T,
\qquad j=0,\ldots,\ell_0-1.
\]
\end{definition}

\begin{remark}
\label{Rem:Period_and_Q_j}
The following considerations show that the period $\ell_0\in\N$ in Lemma~\ref{Lem:EventualPeriodicityT} and the associated patch sets $\Qq_j$ coincide with the corresponding quantities in Theorem~\ref{Thm:SubshiftPartialLimit}.
In particular, any eventual period of the $T$-dictionaries as in Lemma~\ref{Lem:EventualPeriodicityT} provides an upper bound on the number of limit points of the iterated subshifts $(\Theta_n)_{n\in\N}$.
\end{remark}

\begin{proof}[Proof of Lemma~\ref{Lem:EventualPeriodicityT}]
We encode the evolution of $T$-dictionaries by a self-map on a finite space.
Define
\[
\Phi:2^{\Aa^T}\longrightarrow 2^{\Aa^T},\qquad
\Phi(\Pp)
:=
\bigl\{
Q\in\Aa^T \;:\; \exists\,P\in\Pp \text{ with } Q\prec S(P)
\bigr\}.
\]
For $n\in\N_0$, set
\[
\Pp_n := W\bigl(S^n(\omega_0)\bigr)_T \subseteq \Aa^T .
\]

We first show that $\Pp_{n+1}=\Phi(\Pp_n)$ for all $n\in\N_0$.

\begin{itemize}
\item[$\subseteq$:]
Let $Q\in\Pp_{n+1}$. Then
$Q\prec S^{n+1}(\omega_0)=S\bigl(S^n(\omega_0)\bigr)$.
Choose $x\in\Gamma$ such that $xQ = S^{n+1}(\omega_0)\vert_{xT}$.
Since $T$ is a convenient testing domain (Definition~\ref{Def:TestingDomain}~\eqref{enu:NT_simplified}), there exists $\gamma\in\Gamma$ such that
$xT\subseteq F(1,\gamma T)$.
Definition~\ref{Def:SubstitutionMap} implies
\[
xQ
\prec
S\bigl(S^n(\omega_0)\bigr)\vert_{F(1,\gamma T)}
=
S\bigl(S^n(\omega_0)\vert_{\gamma T}\bigr).
\]
Since $S^n(\omega_0)\vert_{\gamma T}\in\Pp_n$, we conclude
$Q\in\Phi(\Pp_n)$.

\item[$\supseteq$:]
Let $Q\in\Phi(\Pp_n)$. Then there exists $P\in\Pp_n$ with
$Q\prec S(P)$.
Since $P\prec S^n(\omega_0)$, it follows that
$S(P)\prec S^{n+1}(\omega_0)$, and hence
$Q\in\Pp_{n+1}$.
\end{itemize}

As $2^{\Aa^T}$ is finite, every orbit of $\Phi$ is eventually periodic.
Specifically, for $\Pp_0=W(\omega_0)_T$, there exist integers $N_0,\ell_0\in\N$,
bounded by $\#\,2^{\Aa^T}$, such that
\[
\Pp_n=\Pp_{n+\ell_0}
\qquad\text{for all } n\ge N_0.
	\hfill \qedhere
\]
\end{proof}

As an immediate consequence of the previous result we obtain the following stabilization result for patches under iterations of $S^{\ell_0}$.

\begin{lemma}
\label{Lem:Qq-patches-stabilize}
Let $S$ be a substitution map on $\Aa^\Gamma$ with a convenient testing domain $T$.
Let $\omega_0\in\Aa^\Gamma$ with associated patch sets $\{\Qq_j\}_{j=0}^{\ell_0-1}$.
Then the following holds for each $j\in\{0,\ldots,\ell_0-1\}$.
\begin{enumerate}[(a)]
\item With the convention $\Qq_{-1}=\Qq_{\ell_0-1}$, we have
\begin{align*}
\Qq_j = &\{ P\in\Aa^T \mid \exists Q\in\Qq_{j-1} \text{ such that } P\prec S(Q)\}\\ 
	= &\{ P\in\Aa^T \mid \exists Q\in\Qq_j \text{ such that } P\prec S^{\ell_0}(Q)\}
\end{align*}
\item For all $r\ge 1$, there exists an $n_r\in\N$ such that
\begin{align*}
W(\Qq_j,\ell_0)_{B(r)} 
= &
\bigl\{
	P\in\Aa^{B(r)}
	\;\big|\;
	\exists\,Q\in\Qq_j \text{ with } P\prec S^{n_r\ell_0}(Q)
\bigr\}\\
= &
\bigl\{
	P\in\Aa^{B(r)}
	\;\big|\;
	\forall n\in\N,\, 
	\exists\,Q\in\Qq_j \text{ with } P\prec S^{(n+n_r)\ell_0}(Q)
\bigr\}.
\end{align*}
\end{enumerate}
\end{lemma} 

\begin{proof}
(a) 
By Lemma~\ref{Lem:EventualPeriodicityT} there is an $N_0\in\N$ such that for $n\in \N$
\[
\Qq_j
=
W\big(S^{(N_0+n)\ell_0+j}(\omega_0)\big)_T.
\]
We first note that the first equality immediately implies the second.
Let $P\in\Qq_j$. Thus, there exists an $x\in \Gamma$ such that $xP= S^{(N_0+1)\ell_0+j}(\omega_0)|_{xT}$. Since $T$ is a convenient testing domain (Definition~\ref{Def:TestingDomain}~\eqref{enu:NT_simplified}), there exists a $\gamma\in \Gamma$ such that $xT\subseteq F(1,\gamma T)$. Hence,
\[
P\prec S\Big(S^{(N_0+1)\ell_0 -1+j}(\omega_0)\Big)|_{F(1,\gamma T)}
	= S\Big(S^{(N_0+1)\ell_0+(j-1)}(\omega_0)|_{\gamma T}\Big).
\]
Since $S^{(N_0+1)\ell_0 + (j-1)}(\omega_0)|_{\gamma T}$ is a shift of a patch $Q\in\Qq_{j-1}$, we conclude $P\prec S(Q)$. 

If conversely $P\in\Aa^T$ and $Q\in\Qq_{j-1}$ are such that $P\prec S(Q)$. Then $Q\prec S^{(N_0+1)\ell_0+j-1}(\omega_0)$ holds by definition implying $P\prec S^{(N_0+1)\ell_0+j}(\omega_0)$ and $P\in\Aa^T$, i.e. $P\in\Qq_j$.

\medskip

(b) By definition of $(\Qq_j,\ell_0)$-legal, we have
\[
W(\Qq_j,\ell_0)_{B(r)}
	= \{ P\in\Aa^{B(r)} \mid \exists n\in\N, \, \exists Q\in\Qq_j \text{ such that } P\prec S^{n\ell_0}(Q) \}.
\]
Combined with (a), the equalities in (b) follow immediately.
\end{proof}

Following \cite[Sec.~3.3]{BeckusThesis}, it suffices for
Theorem~\ref{Thm:SubshiftPartialLimit} to show that for every $r>0$ the
$B(r)$-dictionaries of $S^{(N_0+n)\ell_0+j}(\omega_0)$ stabilize for $n$ large and are
described by $(\Qq_j,\ell_0)$-legal patches.

\begin{lemma}[Stabilization of patches supported on balls]
\label{Lem:DictBallStabilizeGeneral}
Let $S$ be a substitution map on $\Aa^\Gamma$ with a convenient testing domain $T$ and let $\omega_0\in\Aa^\Gamma$ with associated patch sets $\{\Qq_j\}_{j=0}^{\ell_0-1}$.
For every $r\ge 1$ there exist integers $n_r\in\N$ such that for all
$n\ge n_r$ and all $j\in\{0,\dots,\ell_0-1\}$,
\[
W\bigl(S^{(N_0+n)\ell_0+j}(\omega_0)\bigr)_{B(r)}
= W(\Qq_j,\ell_0)_{B(r)}
\]

\end{lemma}

\begin{proof}
Fix $r\ge 1$.
Since $T$ is a convenient testing domain (Definition~\ref{Def:TestingDomain}~\eqref{enu:TestingDomain_simplified}), there exists $n_r\in\N$ such that
for all $n\ge n_r$ and all $x\in\Gamma$ there exists $\gamma=\gamma(n,x)\in\Gamma$
with
\[
xB(r)\subseteq F(n,\gamma T).
\]
Fix $j\in\{0,\dots,\ell_0-1\}$ and $n\ge n_r$.
By Lemma~\ref{Lem:EventualPeriodicityT}, we have
\begin{equation}
\label{eq:Pf-DictBallStabilizeGeneral}
W\bigl(S^{(N_0+n-n_r)\ell_0+j}(\omega_0)\bigr)_T=\Qq_j.
\end{equation}

\underline{Step 1:} Using \eqref{eq:Pf-DictBallStabilizeGeneral}, we first prove the identity
\[
W\bigl(S^{(N_0 + n)\ell_0+j}(\omega_0)\bigr)_{B(r)}
= W(j,r):=
\bigl\{
	P\in\Aa^{B(r)}
	\;\big|\;
	\exists\,Q\in\Qq_j \text{ such that } P\prec S^{n_r\ell_0}(Q)
\bigr\}.
\]

\noindent\underline{$\subseteq$:}
Let $P\in W\bigl(S^{(N_0 + n)\ell_0+j}(\omega_0)\bigr)_{B(r)}$ and $x\in \Gamma$ such that $xP=S^{(N_0+n)\ell_0+j}(\omega_0)|_{xB(r)}$.
By the choice of $n_r$, there exist $\gamma=\gamma(n_r\ell_0,x)\in\Gamma$ such that
\begin{align*}
xP
=
S^{(N_0 + n)\ell_0+j}(\omega_0)|_{xB(r)}
&\prec
S^{n_r\ell_0}\bigl(S^{(N_0 + n-n_r)\ell_0+j}(\omega_0)\bigr)\vert_{F(n_r\ell_0,\gamma T)}\\
&=
S^{n_r\ell_0}\bigl(S^{(N_0 + n-n_r)\ell_0+j}(\omega_0)\vert_{\gamma T}\bigr),
\end{align*}
and hence $P\prec S^{n_r\ell_0}(Q)$ for some $Q\in\Qq_j$ by Equation~\eqref{eq:Pf-DictBallStabilizeGeneral}.

\noindent\underline{$\supseteq$:}
If $P\in \Aa^{B(r)}$ with $P\prec S^{n_r\ell_0}(Q)$ for some $Q\in\Qq_j$, then
$Q\prec S^{(N_0 + n-n_r)\ell_0+j}(\omega_0)$ follows by Equation~\eqref{eq:Pf-DictBallStabilizeGeneral}.
Thus, $P\prec S^{(N_0 + n)\ell_0+j}(\omega_0)$ follows proving $P\in\Qq_j$.

\medskip

\underline{Step 2:} Combining Lemma~\ref{Lem:Qq-patches-stabilize}~(b) with Step~1, we conclude $W(j,r)=W(\Qq_j,\ell_0)_{B(r)}$.
\end{proof}

\begin{proof}[Proof of Theorem~\ref{Thm:SubshiftPartialLimit}]
Fix $j\in\{0,\dots,\ell_0-1\}$.
By Lemma~\ref{Lem:DictBallStabilizeGeneral}, for every $r>0$, the $B(r)$-dictionary of
$S^{(N_0+n)\ell_0+j}(\omega_0)$ stabilizes for $n$ large and equals $W(\Qq_j,\ell_0)_{B(r)}$.
Hence, the subshifts $\Theta_{n\ell_0+j}$ converge in $\Jj$ \cite{BeckusThesis}, and their limit equals
\[
\Omega_j
	= \{\omega\in\Aa^\Gamma \mid \omega \text{ is $(\Qq_j,\ell_0)$-legal}\}.
	\hfill\qedhere
\]
\end{proof}

%%%%%%%%%%%%%%%%%%%%%%%%%%%%%%%%%%%%%%%%%%%%%%%%%%%%%%%%%%%%%%%%%%%%%%%%%%%%%%%%
\section{Structural defects via fixed points}
\label{Sec:Limits_FixedPoints}
%%%%%%%%%%%%%%%%%%%%%%%%%%%%%%%%%%%%%%%%%%%%%%%%%%%%%%%%%%%%%%%%%%%%%%%%%%%%%%%%

We now establish a general mechanism to represent $\Omega_j$ as a finite union of orbit
closures provided suitable fixed points of the substitution map $S$ exist.
A configuration $\rho\in\Aa^\Gamma$ is an \emph{$S$-fixed point} if $S^L(\rho)=\rho$ for
some $L\in\N$.

\begin{proposition}
\label{Prop:LimitSubshifts_FiniteUnion}
Let $S$ be a substitution map on $\Aa^\Gamma$ with a convenient testing domain $T$ and let $\omega_0\in\Aa^\Gamma$ with associated patch sets $\{\Qq_j\}_{j=0}^{\ell_0-1}$ and limiting subshifts $\{\Omega_j\}_{j=0}^{\ell_0-1}$ as in Theorem~\ref{Thm:SubshiftPartialLimit}.
Fix $j\in\{0,\dots,\ell_0-1\}$. Assume that for each $Q\in\Qq_j$ there exists a
configuration $\rho_Q\in\Omega_j$ such that $Q\prec \rho_Q$ and $\rho_Q$ is an
$S$-fixed point. Then
\[
\Omega_j
=
\bigcup_{Q\in\Qq_j} \overline{\Orb(\rho_Q)}.
\]
If $S$ is additionally primitive and $\Omega_j\neq\Omega(S)$, then we can replace $\Qq_j$ by $\Qq_j\setminus W(S)$.
\end{proposition}

\begin{proof}
Since $\rho_Q \in \Omega_j$ and $\Omega_j$ is a subshift, we conclude $\overline{\Orb(\rho_Q)} \subseteq \Omega_j$ for each $Q \in \Qq_j$.
For the converse inclusion it suffices to show that for each $r\ge 1$ and each $P \in W(\Omega_j)_{B(r)}$, there exists $Q \in \Qq_j$ such that $P \prec \rho_Q$.

Let $r\ge 1$ and $P \in W(\Omega_j)_{B(r)}=W(\Qq_j,\ell_0)_{B(r)}$ where the equality follows from Theorem~\ref{Thm:SubshiftPartialLimit}.
By Lemma~\ref{Lem:Qq-patches-stabilize}~(a) and Lemma~\ref{Lem:DictBallStabilizeGeneral} there exists $n_r \in \N$ such that for all $n \in\N_0$
there exists $Q_n \in \Qq_j$ with $P \prec S^{(n+n_r)\ell_0}(Q_n)$.

Choose $\rho_Q\in \Omega_j$ and $L_Q\in\N$ for each $Q \in \Qq_j$ such that $Q\prec \rho_Q$ and $S^{L_Q}(\rho_Q)=\rho_Q$.
Then
\[
S^{L_0 \ell_0}(\rho_Q) = \rho_Q
\quad \text{for all } Q \in \Qq_j \text{ and } L_0 := n_r \cdot \prod_{Q \in \Qq_j} L_Q.
\]

Since $L_0 \ge n_r$, the previous considerations imply that there exists $Q=Q_{L_0-n_r} \in \Qq_j$ such that
\[
P \prec S^{L_0 \ell_0}(Q) 
	\prec S^{L_0 \ell_0}(\rho_Q) 
	= \rho_Q.
\]
Suppose now that $S$ is additionally primitive. Primitivity together with the fact that for $Q\in \Qq_j\setminus W(S)$, $\rho_Q$ is an $S$-fixed point implies that $W(S)\subseteq W(\rho_Q)$. Thus, for $P\in\Qq_j\cap W(S)$, we can choose $\rho_P\in\Omega(S)$ such that $\overline{\Orb(\rho_P)}\subseteq \overline{\Orb(\rho_Q)}$  for any $Q\in\Qq_j\setminus W(S)$.
\end{proof}

In Proposition~\ref{Prop:LimitSubshifts_FiniteUnion}, we assumed the existence of $S$-fixed points $\rho_Q\in\Omega_j$ for $Q\in\Qq_j$. We now provide a sufficient condition ensuring this.
To this end, the map
\[
\Psi:\Aa^T\to\Aa^T,\qquad \Psi(P):=S(P)\big|_T,
\]
plays a central role. A patch $Q\in\Aa^T$ is called \emph{$T$-periodic} (with respect to $S$) if there exists a $p\in \N$ such that $\Psi^p(Q)=Q$.
For $Q\in\Aa^T$, we call $p\in\N$ the \emph{minimal $T$-period} of $Q$ if $p$ is the
smallest integer $n\in\N$ such that $\Psi^{n}(Q)=Q$.

\begin{lemma}\label{Lem:Psi_and_S}
Let $S$ be a substitution map on $\Aa^\Gamma$ with a convenient testing domain $T$.
For all $n\in\N$ and all $Q\in\Aa^T$, we have
$\Psi^n(Q)=S^n(Q)\big|_T$.
\end{lemma}

\begin{proof}
The case $n=1$ holds by definition. Assume $\Psi^n(Q)=S^n(Q)|_T$ for some $n\in\N$. Thus, Definition~\ref{Def:SubstitutionMap} yields
\[
\Psi^{n+1}(Q)=\Psi(\Psi^n(Q))=S\bigl(S^n(Q)|_T\bigr)\big|_T
=
S^{n+1}(Q)\big|_{T\cap F(1,T)}.
\]
Since $T\subseteq F(1,T)$ by Definition~\ref{Def:TestingDomain}~\eqref{enu:NT_simplified}, the induction is completed.
\end{proof}

\begin{proposition}
\label{Prop:SuffCond_Q-fixed_point}
Let $S$ be a substitution on $\Aa^\Gamma$ with convenient testing domain $T$, and let $Q\in\Aa^T$ be $T$-periodic with $\Psi^p(Q)=Q$ for some $p\in\N$. 
Then, for every $\rho\in\Aa^\Gamma$ with $\rho|_T=Q$, the limit
\[
\rho_Q := \lim_{n\to\infty} S^{np}(\rho)
\]
exists in $\Aa^\Gamma$ and is independent of the choice of $\rho$. Moreover, $\rho_Q|_T=Q$ and $S^p(\rho_Q)=\rho_Q$.
\end{proposition}

\begin{proof}
Let $\rho_n:=S^{np}(\rho)$. By Lemma~\ref{Lem:Psi_and_S},
\[
\rho_n|_T
= S^{np}(\rho)\big|_{T\cap F(np,T)}
= S^{np}(\rho|_T)\big|_T
= S^{np}(Q)\big|_T
= \Psi^{np}(Q)
= Q
= \rho|_T.
\]
Fix $r\ge 1$. Since $T$ is a convenient testing domain (Definition~\ref{Def:TestingDomain}~\eqref{enu:TestingDomain_simplified}), there exists an $n_r\in\N$ such that $B(r):=B(e,r)\cap\Gamma\subseteq F(n_r,T)$.
Thus, we obtain for all $n\in\N$
\[
\rho_{n_r+n}\big|_{B(r)}
	= \big(S^{n_rp}(\rho_n)\big)\big|_{B(r)\cap F(n_rp,T)} 
	= \big(S^{n_rp}(\rho_n|_T)\big)\big|_{B(r)} 
	= \big(S^{n_rp}(\rho|_T)\big)\big|_{B(r)}
	= \rho_{n_r}\big|_{B(r)} .
\]
Since $B(r)$ exhausts $\Gamma$, this shows that $(\rho_n)_{n\in\N}$ is a Cauchy sequence in the product topology of $\Aa^\Gamma$ and hence converges to a unique configuration $\rho_Q\in\Aa^\Gamma$ with $\rho_Q|_T=Q$.

Finally, continuity of $S$ on $\Aa^\Gamma$ (Definition~\ref{Def:SubstitutionMap}) gives
\[
S^p(\rho_Q)
	= \lim_{n\to\infty} S^p(\rho_n)
	= \lim_{n\to\infty}\rho_{n+1}
	= \rho_Q.
	\hfill\qedhere
\]
\end{proof}

In the following, $\rho_Q\in\Aa^\Gamma$ denotes the $S$-fixed point associated with a $T$-periodic patch $Q$, as obtained in Proposition~\ref{Prop:SuffCond_Q-fixed_point}.

\begin{proposition}
\label{Prop:Tperiodic_Inclusion}
Let $S$ be a substitution map on $\Aa^\Gamma$ with a convenient testing domain $T$.
Let $P,Q\in\Aa^T$ be $T$-periodic patches, and suppose that $P \prec \rho_Q$.
Then
\[
\overline{\Orb(\rho_P)} \subseteq \overline{\Orb(\rho_Q)}.
\]
\end{proposition}

\begin{proof}
Let $p,q\in\N$ be such that $\Psi^p(P)=P$ and $\Psi^q(Q)=Q$.
Since $P\prec\rho_Q$, there exists $\gamma\in\Gamma$ with $P = (\gamma\rho_Q)\big|_T$.
By Proposition~\ref{Prop:SuffCond_Q-fixed_point}, the limit
\[
\rho_P := \lim_{n\to\infty} S^{npq}(\gamma\rho_Q)
\]
exists and defines the $S$-fixed point associated with $P$.
Since $S^q(\rho_Q)=\rho_Q$, the orbit closure $\overline{\Orb(\rho_Q)}$ is
$S^q$-invariant.
As $\gamma\rho_Q\in\overline{\Orb(\rho_Q)}$ and $\overline{\Orb(\rho_Q)}$ is
closed, it follows that $\rho_P\in\overline{\Orb(\rho_Q)}$.
Taking orbit closures of $\rho_P$ yields the claim.
\end{proof}

The following example shows that, in general, a substitution does not need to admit such fixed points for all $T$-patches.

\begin{example}
\label{Ex:Subst_NoFixedPoints}
Consider the substitution map $S$ on $\Aa^{\Z^2}$ over the alphabet $\Aa:=\{\letterbox{r},\letterbox{g},\letterbox{b}\}$ defined by the block substitution rule $S_0:\Aa\to\Aa^F$ with $F:=\{0,1,2\}^2$ and
\[
\letterbox{r}\mapsto \blocknine{r}{r}{r}{r}{r}{g}{r}{r}{r},\quad
\letterbox{g}\mapsto \blocknine{r}{r}{r}{b}{r}{r}{r}{r}{r},\quad
\letterbox{b}\mapsto \blocknine{r}{r}{r}{g}{r}{b}{r}{r}{r}.
\]
Then $T=\{-1,0\}^2$ is a convenient testing domain of $S$, see \cite{BaBePoTe24}. From the structure of $S_0$ it follows that $\Psi^n(P)=\block{r}{r}{r}{r}$ for all $n\in\N$ and all $P\in\Aa^T$. Hence, $S$ has a unique fixed point $\rho=\lim\limits_{n\to\infty} S^n(\omega)$ for each $\omega\in\Aa^{\Z^2}$ and $\rho\in\Omega(S)$.

However, if $\omega_0\in\Aa^{\Z^2}$ contains the patch $\hblock{b}{b}$, then $(\overline{\Orb(S^n(\omega_0)})_{n\in\N}$ admits a limit point $\Omega\neq\Omega(S)$. Thus, this limit point cannot be represented as a union of orbit closures of fixed points. The example is a straightforward extension of the substitution given in \cite[Eq.~(2.2)]{Ten24}.
\end{example}

This should be contrasted with \cite[Proposition~5.15]{BSTY19}, which provides a characterization of limit configurations of $S$ in one dimension. 
A limit configuration of $S$ is a configuration of the form $\omega=\lim_{n\to\infty} S^n(\rho_n)$ for some sequence $(\rho_n)_n$. 
Every $S$-fixed point is, in particular, a limit configuration of $S$.

%%%%%%%%%%%%%%%%%%%%%%%%%%%%%%%%%%%%%%%%%%%%%%%%%%%%%%%%%%%%%%%%%%%%%%%%%%%%%%%%%%%%%%%%%%%%%%%%%%%%%%%%%%%%%%%%%%%%%%%%%%%%%%%%%%%%%%%%%%%%%%%%
\section{Two classes of substitutions}
\label{Sec:TwoClasses}

Combining Proposition~\ref{Prop:LimitSubshifts_FiniteUnion} and
Proposition~\ref{Prop:SuffCond_Q-fixed_point}, we consider two classes of
substitutions generalizing the table tiling and the chair tiling.
Both classes share the property that the map
\[
\Psi:\Aa^T\to\Aa^T, \qquad \Psi(P):=S(P)\vert_T,
\]
when restricted to the sets $\Qq_j\setminus W(S) \subseteq\Aa^T$, is bijective where $\Qq_j$ was defined in Definition~\ref{Def:Q_j}.

\begin{lemma}
\label{Lem:SuffCond_Q-fixed_point}
Let $S$ be a substitution map on $\Aa^\Gamma$ with a convenient testing domain $T$.
If for a finite set $\Qq\subseteq\Aa^T$ the restriction
\[
\Psi:\Qq\to\Aa^T, \qquad \Psi(Q)=S(Q)\vert_T,
\]
is bijective onto the image $\Qq$, then for every $Q\in\Qq$ there exists a $p\in\N$ such that $\Psi^{p}(Q)=Q$.
\end{lemma}

\begin{proof}
Since $\Qq$ is finite and $\Psi:\Qq\to\Qq$ is bijective, $\Psi$ acts as a
permutation on $\Qq$.
Hence, every $Q\in\Qq$ is periodic under iteration of $\Psi$.
\end{proof}

%%%%%%%%%%%%%%%%%%%%%%%%%%%%%%%%%%%%%%%
\subsection{$T$-bijective substitutions}
\label{Subsec:T_bijective}
%%%%%%%%%%%%%%%%%%%%%%%%%%%%%%%%%%%%%%%

\begin{definition}
\label{Def:T-bjective}
Let $S$ be a substitution map on $\Aa^\Gamma$ with a convenient testing domain $T$.
Then $S$ is called \emph{$T$-bijective} if the map
\[
\Psi:\Aa^T\to\Aa^T, \qquad
\Psi(P) := S(P)\big|_T
\]
is bijective.
\end{definition}

For $\Gamma=\Z$ and $T=\{-1,0\}$, $T$-bijective substitutions are also called marked substitutions \cite{Fri99}.
The table tiling substitution from Section~\ref{Subsec:TableTiling} is $T$-bijective by Lemma~\ref{Lem:SuffCond_T-bijective} below.

\begin{proposition}
\label{Prop:T-bijective_FixedPoints}
Let $S$ be a substitution map on $\Aa^\Gamma$ with a convenient testing domain $T$. Let $\omega_0\in\Aa^\Gamma$ with associated patch sets $\{\Qq_j\}_{j=0}^{\ell_0-1}$ and limiting subshifts $\{\Omega_j\}_{j=0}^{\ell_0-1}$ as in Theorem~\ref{Thm:SubshiftPartialLimit}. If $S$ is $T$-bijective, then there exists for each $j\in\{0,\ldots,\ell_0-1\}$ and $Q\in\Qq_j$, a $\rho_Q\in\Omega_j$ such that
\[
\Omega_j
=
\bigcup_{Q\in\Qq_j} \overline{\Orb(\rho_Q)}.
\]
If $S$ is additionally primitive and $\Omega_j\neq \Omega(S)$, then we can replace $\Qq_j$ by $\Qq_j\setminus W(S)$.
\end{proposition}

\begin{proof}
This is an immediate consequence of Lemma~\ref{Lem:SuffCond_Q-fixed_point}, Proposition~\ref{Prop:SuffCond_Q-fixed_point} and Proposition~\ref{Prop:LimitSubshifts_FiniteUnion}.
\end{proof}

Recall from Proposition~\ref{Prop:Exist_SubstMap} that the substitution map is uniquely determined by
\begin{equation}\label{Eq:SubstitutionMap}
S:\Pat\to\Pat,
\qquad
S(P)(\gamma)
=
S_0\bigl(P(\eta_\gamma)\bigr)\bigl(D(\eta_\gamma)^{-1}\gamma\bigr),
\quad \gamma\in F\bigl(1,\supp(P)\bigr).
\end{equation}
Here $S_0:\Aa\to\Aa^F$ is the substitution rule, $D:\Gamma\to\Gamma$ is a group homomorphism satisfying \eqref{Eq:FundCell_D-F}, and for each $\gamma\in\Gamma$, the element $\eta_\gamma\in\Gamma$ is uniquely determined by the condition $\gamma\in D(\eta_\gamma)F$.

\begin{lemma}
\label{Lem:SuffCond_T-bijective}
Let $S$ be a substitution map on $\Aa^\Gamma$ with a convenient testing domain $T$.
If for each $t\in T$ the map
\[
\Aa \ni a \longmapsto S_0(a)\bigl(D(\gamma_t)^{-1}t\bigr)
\]
is bijective, then $S$ is $T$-bijective.
In particular, if for every $x\in F$ the map
\[
\Aa \ni a \longmapsto S_0(a)(x)\in\Aa
\]
is bijective, then $S$ is $T$-bijective.
\end{lemma}

\begin{proof}
Using the evaluation identity \eqref{Eq:SubstitutionMap} from Proposition~\ref{Prop:Exist_SubstMap},
\[
S(P)(t)
=
S_0\bigl(P(\gamma_t)\bigr)\bigl(D(\gamma_t)^{-1}t\bigr),
\qquad P\in\Aa^T,\ t\in T,
\]
the value $\Psi(P)(t)=S(P)(t)$ depends bijectively on the single letter
$P(\gamma_t)\in\Aa$.
Since this holds for every $t\in T$, the map $\Psi:\Aa^T\to\Aa^T$ is bijective.

The second claim follows from the first, together with the inclusion
$\{ D(\gamma_t)^{-1}t \mid t\in T \}\subseteq F$.
\end{proof}

\begin{corollary}
\label{Cor:Table_Tbijective}
The table tiling substitution defined in Section~\ref{Subsec:TableTiling} is $T$-bijective with a convenient testing domain $T=\{-1,0\}^2$.
\end{corollary}

\begin{proof}
Straightforward computations using \cite[Prop.~3.2]{BaBePoTe24} imply that $T=\{-1,0\}^2$ is a convenient testing domain. Then the statement follows from Lemma~\ref{Lem:SuffCond_T-bijective} using that the table tiling substitution $S_0:\Aa\to\Aa^F$ with $F=\{0,1\}^2$ is bijective if restricted to any component $x\in F$.
\end{proof}

In \cite{BaBePoTe24} we showed that $\overline{\Orb(S^n(\omega_0))},\, n\in\N,$ converges to $\Omega(S)$ if $W(\omega_0)_T \subseteq W(S)$. For $T$-bijective substitutions, we show next that this condition is also necessary. 
Thus, convergence of periodic approximations to $\Omega(S)$ is highly constrained by this condition.

\begin{proposition}
\label{Prop:Tbijective_convergence}
Let $S$ be a substitution map on $\Aa^\Gamma$ with a convenient testing domain $T$ and let
$\omega_0\in\Aa^\Gamma$. If $S$ is $T$-bijective, then the following
statements are equivalent:
\begin{enumerate}[(i)]
\item $\overline{\Orb(S^n(\omega_0))}$ converges to the subshift $\Omega(S)$.
\item $W(\omega_0)_T \subseteq W(S)$.
\end{enumerate}
\end{proposition}

\begin{proof}
(i)$\Rightarrow$(ii):
Assume that (ii) fails. Then there exists $\rho\in\Orb(\omega_0)$ such that
$Q:=\rho\vert_T\in W(\omega_0)_T\setminus W(S)$.
Since $S$ is $T$-bijective, there exists a $p\in\N$ such that $\Psi^p(Q)=Q$ by Lemma~\ref{Lem:SuffCond_Q-fixed_point}. 
Thus, Proposition~\ref{Prop:SuffCond_Q-fixed_point} yields that $S^{np}(\rho)\to\rho_Q$ with $\rho_Q\vert_T=Q$.
Since $Q$ is not $S$-legal, $\rho_Q\notin\Omega(S)$, and thus
$\overline{\Orb(S^n(\omega_0))}$ cannot converge to $\Omega(S)$.

(ii)$\Rightarrow$(i): This is proved in \cite[Thm.~2.14]{BaBePoTe24}.
\end{proof}

%%%%%%%%%%%%%%%%%%%%%%%%%%%%%%%%%%%
\subsection{Pure T-illegal substitutions}
\label{Subsec:Central_illegal}
%%%%%%%%%%%%%%%%%%%%%%%%%%%%%%%%%%%

\begin{definition}
\label{Def:central illegal}
Let $S$ be a substitution map on $\Aa^\Gamma$ with a convenient testing domain $T$.
Then $S$ is called \emph{purely $T$-illegal} if there exists $m_0\in\N$ such that for all $P\in\Aa^T$ and all $n\ge m_0$,
\[
Q\in\Aa^T : \, Q\prec S^{n}(P) \text{ and } Q\notin W(S)
\quad\Rightarrow\quad
Q = S^{n}(P)\big|_{T}.
\]
\end{definition}

The chair tiling substitution is purely $T$-illegal, see Lemma~\ref{Lem:Chair_CentralIllegal} below.

\begin{proposition}
\label{Prop:Central-illegal_FixedPoints}
Let $S$ be a primitive substitution map on $\Aa^\Gamma$ with convenient testing domain $T$. Let $\omega_0\in\Aa^\Gamma$ with associated patch sets $\{\Qq_j\}_{j=0}^{\ell_0-1}$ and limiting subshifts $\{\Omega_j\}_{j=0}^{\ell_0-1}$ as in Theorem~\ref{Thm:SubshiftPartialLimit}. If $S$ is purely $T$-illegal, then there exists for each $j\in\{0,\ldots,\ell_0-1\}$ and $Q\in\Qq_j\setminus W(S)$, an $S$-fixed point $\rho_Q\in\Omega_j$ such that
\[
\Omega_j
=
\bigcup_{Q\in\Qq_j\setminus W(S)} \overline{\Orb(\rho_Q)}.
\]
\end{proposition}

\begin{proof}
Let $m_0\in\N$ be the integer of Definition~\ref{Def:central illegal}. 
Fix $j\in\{0,\ldots,\ell_0-1\}$ and choose $n\in\N$ such that $n\ell_0\ge m_0$.
By Lemma~\ref{Lem:Qq-patches-stabilize}, for each
$Q\in\Qq_j\setminus W(S)$ there exists $P_Q\in\Qq_j$ with $Q\prec S^{n\ell_0}(P_Q)$.
Since $Q$ is not $S$-legal, we necessarily have $P_Q\in\Qq_j\setminus W(S)$.
As $S$ is purely $T$-illegal and $n\ell_0\ge m_0$, Lemma~\ref{Lem:Psi_and_S} yields
\[
Q = S^{n\ell_0}(P_Q)\vert_T = \Psi^{n\ell_0}(P_Q).
\]

Since the set $\Qq_j\setminus W(S)$ is finite, the restriction
\[
\Psi^{n\ell_0}:\Qq_j\setminus W(S)\to\Qq_j\setminus W(S)
\]
is bijective. Hence, for every $Q\in\Qq_j\setminus W(S)$ there exists
$p\in\N$, a multiple of $n\ell_0$, such that $\Psi^{p}(Q)=Q$.
Since $S$ is primitive, Proposition~\ref{Prop:SuffCond_Q-fixed_point} and Proposition~\ref{Prop:LimitSubshifts_FiniteUnion} finish the proof.
\end{proof}

A substitution $S$ on $\Z^d$ is called a \emph{block substitution} if its support $F\subseteq \Z^d$ is rectangular, that is,
\[
F=\{0,\ldots,k_1\}\times\cdots\times\{0,\ldots,k_d\}
\]
for some $k_1,\dots,k_d\in\N$. In this case, the group homomorphism $D:\Z^d\to\Z^d$ defined by $D(x):=\bigl((k_j+1)x_j\bigr)_{j=1}^d$
satisfies \eqref{Eq:FundCell_D-F}, and hence there exists a unique substitution map $S$ associated with $S_0$ by Proposition~\ref{Prop:Exist_SubstMap}, see also \cite{Fra05,BaBePoTe24}.

We show that $S$-fixed points of purely $T$-illegal block substitution satisfy the condition in Theorem~\ref{Thm:ChangeLebesgueMeas}.

\begin{lemma}
\label{Lem:Central-illegal_condition}
Let $T=\{-1,0\}^2$ and let $S$ be a purely $T$-illegal block substitution over $\Z^2$. Suppose $\rho\in\Aa^{\Z^2}$ satisfies $S^L(\rho)=\rho$ for some $L\in\N$. Then
\[
x\in\Z^2,\ r\ge 1,\ B(x,r)\cap T=\emptyset
\quad\Longrightarrow\quad
\rho|_{B(x,r)}\in W(S).
\]
\end{lemma}

\begin{proof}
Consider the sets
\begin{align*}
M_{1,+} := \{ (0,0), (0,-1) \},\quad
&M_{1,-} := \{ (-1,0), (-1,-1) \},\\
M_{2,+} := \{ (-1,0); (0,0) \},\quad
&M_{2,-} := \{ (-1,-1), (0,-1) \}
\end{align*}
satisfying $M_{j,\pm}\subseteq \pm e_j + T$, where $e_1,e_2\in\Z^2$ is the standard basis.
By the block structure of $S$, the supports $F(n,M_{1,+})$, $F(n,M_{1,-})$, $F(n,M_{2,+})$ and $F(n,M_{2,-})$ eventually cover the four half-spaces determined by the coordinate axes.

Let $x\in\Z^2$ and $r\ge 1$ with $B(x,r)\cap T=\emptyset$. Then $B(x,r)$ is contained in one of these half-spaces, and hence
\[
\rho|_{B(x,r)}\prec S^{nL}(\rho|_M)
\]
for some $M\in\{M_{1,+},M_{1,-},M_{2,+},M_{2,-}\}$ and suitable $n\in\N$. It therefore suffices to show $\rho|_M\in W(S)$ for all such $M$.

Let $m_0\in\N$ be as in Definition~\ref{Def:central illegal} and choose $n\in\N$ such that $nL\ge m_0$. Then
\[
\rho|_{M_{j,\pm}} \prec \rho|_{\pm e_j+T}=S^{nL}(\rho)|_{\pm e_j+T}, \qquad j\in\{1,2\}.
\]
Since $S$ is purely $T$-illegal and $nL\ge m_0$, every $T$-patch of $S^{nL}(\rho)$  with support different to $T$ is $S$-legal. Thus, $\rho|_{\pm e_j+T}\in W(S)$, and hence $\rho|_{M_{j,\pm}}\in W(S)$.
\end{proof}

%%%%%%%%%%%%%%%%%%%%%%%%%%%%%%%%%%%%%%%%%%%%%%%%%%%%%%%%%%%%%%%%%%%%%%%%%%%%%%%%
\section{$T$-patch graphs}
\label{Sec:T-patch_graphs}
%%%%%%%%%%%%%%%%%%%%%%%%%%%%%%%%%%%%%%%%%%%%%%%%%%%%%%%%%%%%%%%%%%%%%%%%%%%%%%%%

Throughout this section, we study a substitution map $S$ on $\Aa^\Gamma$ with a
convenient testing domain $T$. By introducing suitable graphs, we characterize
$T$-bijective substitutions and estimate the minimal $T$-periods of patches in
$\Aa^T$ in terms of the lengths of cycles in these graphs. This provides a
computationally illustrative way to determine the $S$-fixed points $\rho_Q$.

Let $\{G_k=(\Vv_k,\Ee_k)\}_{k=1}^m$ be directed graphs. Their \emph{graph tensor product}
$\bigotimes_{k=1}^m G_k$ is the directed graph with vertex set $\prod_{k=1}^m \Vv_k$ and
an edge
\[
\bigl((a_1,\ldots,a_m),(b_1,\ldots,b_m)\bigr)\in \Ee_{\otimes_{k=1}^m G_k}
\quad\Longleftrightarrow\quad
(a_k,b_k)\in\Ee_k \text{ for all } k.
\]

For each $t\in T$, define the directed graph $G_t(S)=(\Vv_t,\Ee_t)$ by $\Vv_t:=\Aa$ and
\[
\Ee_t
:=
\bigl\{(a,b)\in \Aa\times\Aa \;:\;
S_0(a)\bigl(D(\eta_t)^{-1}t\bigr)=b
\bigr\},
\]
where $\eta_t\in\Gamma$ is uniquely determined by the condition $t\in D(\eta_t)F$ via \eqref{Eq:FundCell_D-F}.
\begin{definition}
\label{Def:T-patchGraph}
For $T\subseteq \Gamma$ finite, the \emph{$T$-patch graph} is the tensor product
\[
G_T(S):=\bigotimes_{t\in T} G_t(S).
\]
\end{definition}

\begin{figure}[htb]
\centering
\includegraphics[width=0.9\textwidth]{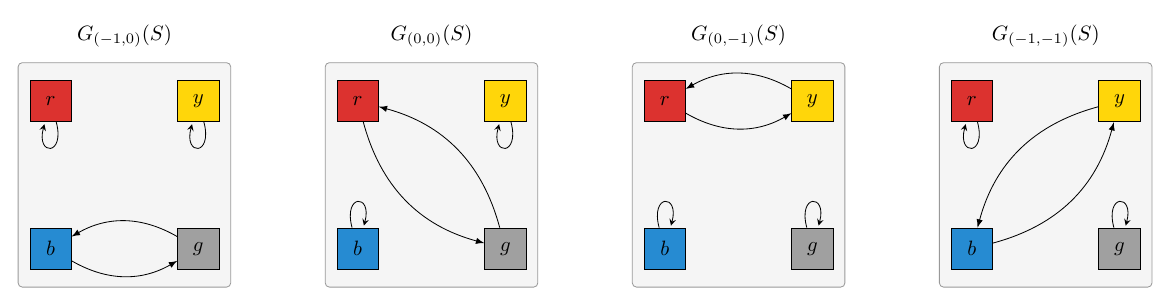}
\caption{The four graphs $G_t(S)$ for $t\in T=\{-1,0\}^2$
associated with the table tiling substitution discussed in Example~\ref{Ex:Table-T-graph}.}
\label{Fig:T-patchGraph_Table}
\end{figure}

Figure~\ref{Fig:T-patchGraph_Table} displays the four graphs corresponding to $t\in T=\{-1,0\}^2$ for the table tiling substitution.

The vertex set of $G_T(S)$ identifies canonically with $\Aa^T$ via $(a_t)_{t\in T}\mapsto P$ with
$P(t)=a_t$.

\begin{lemma}\label{Lem:Tgraph_edges}
Under the identification $\Vv(G_T(S))\cong \Aa^T$, there is a directed edge from $P$ to
$P'$ in $G_T(S)$ if and only if $P'=\Psi(P)$.
\end{lemma}

\begin{proof}
By definition of graph tensor product, there is an edge from $(a_t)$ to $(b_t)$ if and only
if $(a_t,b_t)\in\Ee_t$ for all $t$, i.e.\
\[
b_t = S_0(a_t)\bigl(D(\gamma_t)^{-1}t\bigr)\quad\text{for all } t\in T.
\]
This is exactly the coordinate form of $P'=\Psi(P)$ using \eqref{Eq:SubstitutionMap}.
\end{proof}

\begin{proposition}[Cycles and periods of $T$-patches]
\label{Prop:Tgraph_cycles_lcm}
Let $S$ be a substitution map on $\Aa^\Gamma$ with a convenient testing domain $T$.
\begin{enumerate}[(a)]
\item A patch $Q\in\Aa^T$ lies on a directed cycle in the $T$-patch graph
$G_T(S)$ if and only if $Q$ is $T$-periodic, that is, if and only if
there exists $p\in\N$ such that $\Psi^p(Q)=Q$.
\item Let $L\in\N$ be the least common multiple (lcm) of the lengths of all directed
cycles in $G_T(S)$. If $Q\in\Aa^T$ is $T$-periodic, then its minimal period divides $L$. In particular, $\Psi^L(Q)=Q$.
\end{enumerate}
\end{proposition}

\begin{proof}
\begin{enumerate}[(a)]
\item By Lemma~\ref{Lem:Tgraph_edges}, following edges in $G_T(S)$ is equivalent to
iterating $\Psi$. Hence, a directed cycle based at $Q$ is equivalent to
$\Psi^p(Q)=Q$ for some $p\in\N$.
\item This is an immediate consequence of (a). \hfill\qedhere
\end{enumerate}
\end{proof}

\begin{corollary}\label{Cor:T-bijective+graph}
Let $S$ be a substitution map on $\Aa^\Gamma$ with a convenient testing domain $T$.
The following are equivalent:
\begin{enumerate}[(i)]
\item $S$ is $T$-bijective;
\item the $T$-patch graph $G_T(S)$ is totally cyclic, that is, every vertex lies on a
directed cycle.
\end{enumerate}
\end{corollary}

\begin{proof}
By Proposition~\ref{Prop:Tgraph_cycles_lcm}~(a), a $T$-patch is periodic under
$\Psi$ if and only if it lies on a directed cycle in $G_T(S)$.
If $S$ is $T$-bijective, then $\Psi$ is a bijection on the finite set $\Aa^T$, hence
all $T$-patches are periodic and $G_T(S)$ is totally cyclic.
Conversely, if every vertex of $G_T(S)$ lies on a directed cycle, then every
$T$-patch is periodic under $\Psi$, which implies that $\Psi:\Aa^T\to\Aa^T$ is
bijective, i.e.\ $S$ is $T$-bijective.
\end{proof}

\begin{corollary}
\label{Cor:cyclelengths_general}
Let $S$ be a substitution map on $\Aa^\Gamma$ with a convenient testing domain $T$.
Let $L\in\N$ be the lcm of the lengths of all directed cycles in the associated $T$-patch graph
$G_T(S)$.

If $Q\in\Aa^T$ satisfies $\Psi^{p}(Q) = Q$ for some $p\in\N$, then for every $\rho\in\Aa^\Gamma$ with $\rho|_T=Q$, the limit
\[
\rho_Q:=\lim_{n\to\infty} S^{nL}(\rho)
\]
exists in $\Aa^\Gamma$. Moreover, $\rho_Q|_T=Q$ and $S^{L}(\rho_Q)=\rho_Q$.
\end{corollary}

\begin{proof}
This is an immediate consequence of Proposition~\ref{Prop:Tgraph_cycles_lcm}~(b) and Proposition~\ref{Prop:SuffCond_Q-fixed_point}.
\end{proof}

\begin{example}
\label{Ex:Table-T-graph}
Let $S$ be the substitution map on $\Aa^{\Z^2}$ associated with the
table tiling substitution, see Section~\ref{Subsec:TableTiling}.
By \cite[Prop.~3.2]{BaBePoTe24}, $T=\{-1,0\}^2$ is a convenient testing domain.
The associated $T$-patch graph of the table tiling substitution is
the tensor product of the four graphs $G_t(S)$ for $t\in T$
shown in Figure~\ref{Fig:T-patchGraph_Table}.
The least common multiple of the lengths of all directed cycles in
$G_T(S)$ equals $L=2$.
Hence, every $T$-patch has $T$-period $2$ by
Corollary~\ref{Cor:cyclelengths_general}.
\end{example}

For purely $T$-illegal substitutions, the least common multiple (lcm) of the cycle lengths in the $T$-patch graph also yields an upper bound on the number of possible limit points in Theorem~\ref{Thm:SubshiftPartialLimit}.

\begin{proposition}
\label{Prop:Bound_LimitPoints_CentralIllegal}
Let $S$ be a purely $T$-illegal primitive substitution with convenient testing domain $T$. Let $\ell\in\N$ be the lcm of the lengths of all directed cycles in the associated $T$-patch graph
$G_T(S)$. Then for every $\omega_0\in\Aa^\Gamma$, the iterative subshifts $\overline{\Orb(S^n(\omega_0))}, \, n\in\N$, have at most $\ell$ limit points.  
\end{proposition}

\begin{proof}
From the proof of Lemma~\ref{Lem:EventualPeriodicityT}, recall that the map 
\[
\Phi:2^{\Aa^T}\longrightarrow 2^{\Aa^T},\qquad
\Phi(\Pp)
:=
\bigl\{
Q\in\Aa^T \;:\; \exists\,P\in\Pp \text{ with } Q\prec S(P)
\bigr\}
\]
satisfies $\Pp_{n+1}=\Phi(\Pp_n)$ for all $n\in\N_0$ with $\Pp_n := W\bigl(S^n(\omega_0)\bigr)_T$.

Since $S$ is primitive, there exists an $n_0$ such that $W(S)_T\subseteq \Pp_n$ for all $n\ge n_0$. Let $m_0$ be as in Definition~\ref{Def:central illegal} and set $L_0:=\max\{n_0,m_0\}+\sharp\Aa^T$. By Remark~\ref{Rem:Period_and_Q_j}, it suffices to show $\Pp_n=\Pp_{n+\ell}$ for all $n\ge L_0$.

Fix $n\ge L_0$. Since $n\ge n_0$, it suffices to prove $\Pp_n\setminus W(S)=\Pp_{n+\ell}\setminus W(S)$.

\begin{itemize}
\item[$\subseteq$:] Let $Q\in\Pp_n\setminus W(S)$. Then $Q\prec S^n(P)$ for some $P\in W(\omega_0)_T$. Since $n\ge m_0$ and $Q\notin W(S)$, we obtain from Lemma~\ref{Lem:Psi_and_S}
		\[
			Q=S^n(P)|_T =\Psi^n(P).
		\] 
		Since $G_T(S)$ has $\sharp \Aa^T$ vertices and $n\geq L_0>\sharp \Aa^T$, Lemma~\ref{Lem:Tgraph_edges} implies that the path generated by iterating $\Psi$ on $P$ reaches a cycle, and hence $Q$ lies on a directed cycle in $G_T(S)$. By the choice of $\ell$, we obtain $\Psi^\ell(Q)=Q$ from Proposition~\ref{Prop:Tgraph_cycles_lcm}~(a). Thus, $Q=S^\ell(Q)|_T$ by Lemma~\ref{Lem:Psi_and_S}. Consequently, $Q\prec S^{n+\ell}(P)$ and hence $Q\in \Pp_{n+\ell}\setminus W(S)$.

\item[$\supseteq$:] Let $Q\in \Pp_{n+\ell}\setminus W(S)$. Then $Q\prec S^{n+\ell}(P)$ for some $P\in W(\omega_0)_T$. Since $n+\ell\ge m_0$ and $Q\notin W(S)$, we obtain from Lemma~\ref{Lem:Psi_and_S}
		\[
			Q=S^{n+\ell}(P)|_T =\Psi^{n+\ell}(P) = \Psi^\ell\bigl( \Psi^n(P) \bigr).
		\] 
		Setting $P':=\Psi^n(P)$, the same argument as above shows that $P'$ lies on a cycle in $G_T(S)$, hence $\Psi^\ell(P')=P'$. Therefore $Q=P'=\Psi^n(P)=S^n(P)|_T$, and thus $Q\in\Pp_n\setminus W(S)$. \hfill\qedhere
\end{itemize}
\end{proof}

\begin{figure}[htb]
\centering
\includegraphics[width=0.9\textwidth]{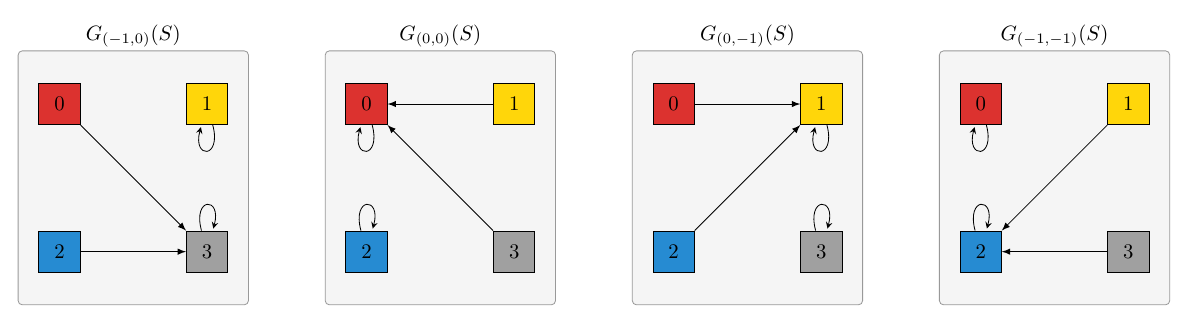}
\caption{The four graphs $G_t(S)$ for $t\in T=\{-1,0\}^2$
associated with the chair tiling substitution are presented.}
\label{Fig:T-patchGraph_Chair}
\end{figure}

\begin{example}
\label{Ex:Chair_NumberLimits}
Let $S$ be the chair tiling substitution defined in Section~\ref{Subsec:ChairTiling}. Then $T=\{-1,0\}^2$ is a convenient testing domain. The associated $T$-patch graph is given in Figure~\ref{Fig:T-patchGraph_Chair}. All directed cycles have length $1$. Thus, Proposition~\ref{Prop:Bound_LimitPoints_CentralIllegal} asserts that for every $\omega_0\in\Aa^{\Z^2}$, the sequence $(\overline{\Orb(S^n(\omega_0))})_{n\in\N}$ admits at most one limit point. In particular, the sequence converges.
\end{example}

Our previous paper \cite{BaBePoTe24} elaborates on applications of graphs for studying the convergence of dynamical systems.
The graphs which are presented here are somewhat different in nature. But the representation of graphs as tensor product may also be used for some of the graphs in \cite{BaBePoTe24}, so that one deals with smaller graphs and improves computation efficiency and analysis.

\printbibliography

\appendix

\section{Chair tiling}
\label{App:ChairTilling}

Recall the chair tiling substitution rule
$S_0:\Aa\to\Aa^F$ on the alphabet
$\Aa:=\{\letterbox{r},\letterbox{y},\letterbox{b},\letterbox{g}\}$ and 
$F=\{0,1\}^2\subseteq\Z^2=\Gamma$ as defined in Section~\ref{Subsec:ChairTiling}.
Since the chair tiling is a block substitution, $T=\{-1,0\}^2$ is a convenient testing domain \cite{BaBePoTe24}, and the $S$-legal $T$-patches are
\begin{equation}
\label{Eq:LegalPatches_Chair}
\begin{array}{cccccccccc}
\block{b}{g}{g}{b} & \block{b}{g}{g}{r} & \block{b}{g}{y}{b} & \block{b}{y}{g}{b} & \block{b}{y}{g}{r} & \block{b}{y}{y}{b} & \block{b}{y}{y}{r} & \block{g}{b}{r}{g} & 
\block{g}{r}{b}{y} & \block{r}{g}{g}{b} \\ 
\block{r}{g}{g}{r} & \block{r}{g}{y}{b} & \block{r}{g}{y}{r} & \block{r}{y}{g}{r} & \block{r}{y}{y}{b} & \block{r}{y}{y}{r} & \block{y}{b}{b}{g} & \block{y}{b}{r}{y} &
\block{y}{r}{r}{g} &  
\end{array}.
\end{equation}
This can be determined by iterating the substitution on a single letter and extracting all $T$-subpatches. Since the substitution is primitive, the collection of $T$-subpatches stabilizes after finitely many steps, and this stabilized set coincides with the set of $S$-legal $T$-patches.

\begin{lemma}
\label{Lem:Chair_CentralIllegal}
Let $T=\{-1,0\}^2$ and let $S$ be the substitution of the chair tiling. Then $S$ is purely $T$-illegal with $m_0=2$. In particular, for every $P\in\Aa^T$, every $n\ge m_0=2$, and every $Q\in\Aa^T$,
\[
Q\prec S^{n}(P),\ Q\notin W(S)
\quad\Longrightarrow\quad
Q = S^{n}(P)\big|_{T}.
\]
\end{lemma}

\begin{proof}
Let $P\in\Aa^T$ and $n\ge 2$. Recall that for $M\subseteq\Z^2$, $F(n,M)$ denotes the support of $S^n(P')$ for any patch $P'\in\Aa^M$. Since $F=\{0,1\}^2$ is the support of $S_0(a)$ for any $a\in\Aa$, the sets
\begin{equation}
\label{Eq:HalfspaceSupports}
\begin{aligned}
F(n,\{(-1,0),(0,0)\}),\quad &F(n,\{(-1,-1),(0,-1)\}),\\
F(n,\{(-1,0),(-1,-1)\}),\quad &F(n,\{(0,0),(0,-1)\}) .
\end{aligned}
\end{equation}
eventually cover the upper, lower, left, and right half-plane with respect to the coordinate axes, respectively. Hence, if a $T$-patch $Q\prec S^n(P)$ occurs at some translate $\gamma T\neq T$, then it is contained entirely in one of the four half-spaces determined by the coordinate axes.

It therefore suffices to check that, for every $2\times1$- and $1\times2$-patch $P'$, all $T$-subpatches of $S(P')$ or $S^2(P')$ are $S$-legal. This finite verification is recorded in Tables~\ref{Tab:Chair_Iterates_Horizontal}--\ref{Tab:Chair_Iterates_Vertical}. Consequently, for $n\ge 2$ the only possible $S$-illegal $T$-patch of $S^n(P)$ is the central restriction $S^n(P)|_{T}$, as claimed.
\end{proof}

\begin{table}
\caption{$S$ and $S^2$ applied to all horizontal $2\times1$-patches. If after one iteration all resulting $T$-patches are already $S$-legal, we do not display further iterates.}
\label{Tab:Chair_Iterates_Horizontal}
\centering
\begingroup
\setlength{\tabcolsep}{6pt}%
\renewcommand{\arraystretch}{1.10}%
\noindent\begin{tabular}{@{}p{0.24\linewidth}@{}p{0.24\linewidth}@{}p{0.24\linewidth}@{}p{0.24\linewidth}@{}}
\(
\mathaxisbox{
\begin{tikzpicture}[scale=0.32]
\filldraw[Tiler, line width=0.15pt] (0,-1) rectangle (1,0);
\filldraw[Tiler, line width=0.15pt] (1,-1) rectangle (2,0);
\draw[Frame] (0,-1) rectangle (2,0);
\end{tikzpicture}
}%
\mapsto%
\mathaxisbox{
\begin{tikzpicture}[scale=0.32]
\filldraw[Tiley, line width=0.15pt] (0,-1) rectangle (1,0);
\filldraw[Tiler, line width=0.15pt] (1,-1) rectangle (2,0);
\filldraw[Tiley, line width=0.15pt] (2,-1) rectangle (3,0);
\filldraw[Tiler, line width=0.15pt] (3,-1) rectangle (4,0);
\filldraw[Tiler, line width=0.15pt] (0,-2) rectangle (1,-1);
\filldraw[Tileg, line width=0.15pt] (1,-2) rectangle (2,-1);
\filldraw[Tiler, line width=0.15pt] (2,-2) rectangle (3,-1);
\filldraw[Tileg, line width=0.15pt] (3,-2) rectangle (4,-1);
\draw[Frame] (0,-2) rectangle (4,0);
\end{tikzpicture}
}
\) & \(
\mathaxisbox{
\begin{tikzpicture}[scale=0.32]
\filldraw[Tiler, line width=0.15pt] (0,-1) rectangle (1,0);
\filldraw[Tiley, line width=0.15pt] (1,-1) rectangle (2,0);
\draw[Frame] (0,-1) rectangle (2,0);
\end{tikzpicture}
}%
\mapsto%
\mathaxisbox{
\begin{tikzpicture}[scale=0.32]
\filldraw[Tiley, line width=0.15pt] (0,-1) rectangle (1,0);
\filldraw[Tiler, line width=0.15pt] (1,-1) rectangle (2,0);
\filldraw[Tiley, line width=0.15pt] (2,-1) rectangle (3,0);
\filldraw[Tileb, line width=0.15pt] (3,-1) rectangle (4,0);
\filldraw[Tiler, line width=0.15pt] (0,-2) rectangle (1,-1);
\filldraw[Tileg, line width=0.15pt] (1,-2) rectangle (2,-1);
\filldraw[Tiler, line width=0.15pt] (2,-2) rectangle (3,-1);
\filldraw[Tiley, line width=0.15pt] (3,-2) rectangle (4,-1);
\draw[Frame] (0,-2) rectangle (4,0);
\end{tikzpicture}
}
\) & \(
\mathaxisbox{
\begin{tikzpicture}[scale=0.32]
\filldraw[Tiler, line width=0.15pt] (0,-1) rectangle (1,0);
\filldraw[Tileg, line width=0.15pt] (1,-1) rectangle (2,0);
\draw[Frame] (0,-1) rectangle (2,0);
\end{tikzpicture}
}%
\mapsto%
\mathaxisbox{
\begin{tikzpicture}[scale=0.32]
\filldraw[Tiley, line width=0.15pt] (0,-1) rectangle (1,0);
\filldraw[Tiler, line width=0.15pt] (1,-1) rectangle (2,0);
\filldraw[Tileg, line width=0.15pt] (2,-1) rectangle (3,0);
\filldraw[Tileb, line width=0.15pt] (3,-1) rectangle (4,0);
\filldraw[Tiler, line width=0.15pt] (0,-2) rectangle (1,-1);
\filldraw[Tileg, line width=0.15pt] (1,-2) rectangle (2,-1);
\filldraw[Tiler, line width=0.15pt] (2,-2) rectangle (3,-1);
\filldraw[Tileg, line width=0.15pt] (3,-2) rectangle (4,-1);
\draw[Frame] (0,-2) rectangle (4,0);
\end{tikzpicture}
}
\) &
\(
\mathaxisbox{
\begin{tikzpicture}[scale=0.32]
\filldraw[Tiley, line width=0.15pt] (0,-1) rectangle (1,0);
\filldraw[Tiler, line width=0.15pt] (1,-1) rectangle (2,0);
\draw[Frame] (0,-1) rectangle (2,0);
\end{tikzpicture}
}%
\mapsto%
\mathaxisbox{
\begin{tikzpicture}[scale=0.32]
\filldraw[Tiley, line width=0.15pt] (0,-1) rectangle (1,0);
\filldraw[Tileb, line width=0.15pt] (1,-1) rectangle (2,0);
\filldraw[Tiley, line width=0.15pt] (2,-1) rectangle (3,0);
\filldraw[Tiler, line width=0.15pt] (3,-1) rectangle (4,0);
\filldraw[Tiler, line width=0.15pt] (0,-2) rectangle (1,-1);
\filldraw[Tiley, line width=0.15pt] (1,-2) rectangle (2,-1);
\filldraw[Tiler, line width=0.15pt] (2,-2) rectangle (3,-1);
\filldraw[Tileg, line width=0.15pt] (3,-2) rectangle (4,-1);
\draw[Frame] (0,-2) rectangle (4,0);
\end{tikzpicture}
}
\) \\[6pt] 

\multicolumn{2}{@{}p{0.48\linewidth}@{}}{\(
\mathaxisbox{
\begin{tikzpicture}[scale=0.32]
\filldraw[Tiler, line width=0.15pt] (0,-1) rectangle (1,0);
\filldraw[Tileb, line width=0.15pt] (1,-1) rectangle (2,0);
\draw[Frame] (0,-1) rectangle (2,0);
\end{tikzpicture}
}%
\mapsto%
\mathaxisbox{
\begin{tikzpicture}[scale=0.32]
\filldraw[Tiley, line width=0.15pt] (0,-1) rectangle (1,0);
\filldraw[Tiler, line width=0.15pt] (1,-1) rectangle (2,0);
\filldraw[Tiley, line width=0.15pt] (2,-1) rectangle (3,0);
\filldraw[Tileb, line width=0.15pt] (3,-1) rectangle (4,0);
\filldraw[Tiler, line width=0.15pt] (0,-2) rectangle (1,-1);
\filldraw[Tileg, line width=0.15pt] (1,-2) rectangle (2,-1);
\filldraw[Tileb, line width=0.15pt] (2,-2) rectangle (3,-1);
\filldraw[Tileg, line width=0.15pt] (3,-2) rectangle (4,-1);
\draw[Frame] (0,-2) rectangle (4,0);
\end{tikzpicture}
}%
\mapsto%
\mathaxisbox{
\begin{tikzpicture}[scale=0.32]
\filldraw[Tiley, line width=0.15pt] (0,-1) rectangle (1,0);
\filldraw[Tileb, line width=0.15pt] (1,-1) rectangle (2,0);
\filldraw[Tiley, line width=0.15pt] (2,-1) rectangle (3,0);
\filldraw[Tiler, line width=0.15pt] (3,-1) rectangle (4,0);
\filldraw[Tiley, line width=0.15pt] (4,-1) rectangle (5,0);
\filldraw[Tileb, line width=0.15pt] (5,-1) rectangle (6,0);
\filldraw[Tiley, line width=0.15pt] (6,-1) rectangle (7,0);
\filldraw[Tileb, line width=0.15pt] (7,-1) rectangle (8,0);
\filldraw[Tiler, line width=0.15pt] (0,-2) rectangle (1,-1);
\filldraw[Tiley, line width=0.15pt] (1,-2) rectangle (2,-1);
\filldraw[Tiler, line width=0.15pt] (2,-2) rectangle (3,-1);
\filldraw[Tileg, line width=0.15pt] (3,-2) rectangle (4,-1);
\filldraw[Tiler, line width=0.15pt] (4,-2) rectangle (5,-1);
\filldraw[Tiley, line width=0.15pt] (5,-2) rectangle (6,-1);
\filldraw[Tileb, line width=0.15pt] (6,-2) rectangle (7,-1);
\filldraw[Tileg, line width=0.15pt] (7,-2) rectangle (8,-1);
\filldraw[Tiley, line width=0.15pt] (0,-3) rectangle (1,-2);
\filldraw[Tiler, line width=0.15pt] (1,-3) rectangle (2,-2);
\filldraw[Tileg, line width=0.15pt] (2,-3) rectangle (3,-2);
\filldraw[Tileb, line width=0.15pt] (3,-3) rectangle (4,-2);
\filldraw[Tiley, line width=0.15pt] (4,-3) rectangle (5,-2);
\filldraw[Tileb, line width=0.15pt] (5,-3) rectangle (6,-2);
\filldraw[Tileg, line width=0.15pt] (6,-3) rectangle (7,-2);
\filldraw[Tileb, line width=0.15pt] (7,-3) rectangle (8,-2);
\filldraw[Tiler, line width=0.15pt] (0,-4) rectangle (1,-3);
\filldraw[Tileg, line width=0.15pt] (1,-4) rectangle (2,-3);
\filldraw[Tiler, line width=0.15pt] (2,-4) rectangle (3,-3);
\filldraw[Tileg, line width=0.15pt] (3,-4) rectangle (4,-3);
\filldraw[Tileb, line width=0.15pt] (4,-4) rectangle (5,-3);
\filldraw[Tileg, line width=0.15pt] (5,-4) rectangle (6,-3);
\filldraw[Tiler, line width=0.15pt] (6,-4) rectangle (7,-3);
\filldraw[Tileg, line width=0.15pt] (7,-4) rectangle (8,-3);
\draw[Frame] (0,-4) rectangle (8,0);
\end{tikzpicture}
}
\)} &
\multicolumn{2}{@{}p{0.48\linewidth}@{}}{\(
\mathaxisbox{
\begin{tikzpicture}[scale=0.32]
\filldraw[Tiley, line width=0.15pt] (0,-1) rectangle (1,0);
\filldraw[Tileg, line width=0.15pt] (1,-1) rectangle (2,0);
\draw[Frame] (0,-1) rectangle (2,0);
\end{tikzpicture}
}%
\mapsto%
\mathaxisbox{
\begin{tikzpicture}[scale=0.32]
\filldraw[Tiley, line width=0.15pt] (0,-1) rectangle (1,0);
\filldraw[Tileb, line width=0.15pt] (1,-1) rectangle (2,0);
\filldraw[Tileg, line width=0.15pt] (2,-1) rectangle (3,0);
\filldraw[Tileb, line width=0.15pt] (3,-1) rectangle (4,0);
\filldraw[Tiler, line width=0.15pt] (0,-2) rectangle (1,-1);
\filldraw[Tiley, line width=0.15pt] (1,-2) rectangle (2,-1);
\filldraw[Tiler, line width=0.15pt] (2,-2) rectangle (3,-1);
\filldraw[Tileg, line width=0.15pt] (3,-2) rectangle (4,-1);
\draw[Frame] (0,-2) rectangle (4,0);
\end{tikzpicture}
}%
\mapsto%
\mathaxisbox{
\begin{tikzpicture}[scale=0.32]
\filldraw[Tiley, line width=0.15pt] (0,-1) rectangle (1,0);
\filldraw[Tileb, line width=0.15pt] (1,-1) rectangle (2,0);
\filldraw[Tiley, line width=0.15pt] (2,-1) rectangle (3,0);
\filldraw[Tileb, line width=0.15pt] (3,-1) rectangle (4,0);
\filldraw[Tileg, line width=0.15pt] (4,-1) rectangle (5,0);
\filldraw[Tileb, line width=0.15pt] (5,-1) rectangle (6,0);
\filldraw[Tiley, line width=0.15pt] (6,-1) rectangle (7,0);
\filldraw[Tileb, line width=0.15pt] (7,-1) rectangle (8,0);
\filldraw[Tiler, line width=0.15pt] (0,-2) rectangle (1,-1);
\filldraw[Tiley, line width=0.15pt] (1,-2) rectangle (2,-1);
\filldraw[Tileb, line width=0.15pt] (2,-2) rectangle (3,-1);
\filldraw[Tileg, line width=0.15pt] (3,-2) rectangle (4,-1);
\filldraw[Tiler, line width=0.15pt] (4,-2) rectangle (5,-1);
\filldraw[Tileg, line width=0.15pt] (5,-2) rectangle (6,-1);
\filldraw[Tileb, line width=0.15pt] (6,-2) rectangle (7,-1);
\filldraw[Tileg, line width=0.15pt] (7,-2) rectangle (8,-1);
\filldraw[Tiley, line width=0.15pt] (0,-3) rectangle (1,-2);
\filldraw[Tiler, line width=0.15pt] (1,-3) rectangle (2,-2);
\filldraw[Tiley, line width=0.15pt] (2,-3) rectangle (3,-2);
\filldraw[Tileb, line width=0.15pt] (3,-3) rectangle (4,-2);
\filldraw[Tiley, line width=0.15pt] (4,-3) rectangle (5,-2);
\filldraw[Tiler, line width=0.15pt] (5,-3) rectangle (6,-2);
\filldraw[Tileg, line width=0.15pt] (6,-3) rectangle (7,-2);
\filldraw[Tileb, line width=0.15pt] (7,-3) rectangle (8,-2);
\filldraw[Tiler, line width=0.15pt] (0,-4) rectangle (1,-3);
\filldraw[Tileg, line width=0.15pt] (1,-4) rectangle (2,-3);
\filldraw[Tiler, line width=0.15pt] (2,-4) rectangle (3,-3);
\filldraw[Tiley, line width=0.15pt] (3,-4) rectangle (4,-3);
\filldraw[Tiler, line width=0.15pt] (4,-4) rectangle (5,-3);
\filldraw[Tileg, line width=0.15pt] (5,-4) rectangle (6,-3);
\filldraw[Tiler, line width=0.15pt] (6,-4) rectangle (7,-3);
\filldraw[Tileg, line width=0.15pt] (7,-4) rectangle (8,-3);
\draw[Frame] (0,-4) rectangle (8,0);
\end{tikzpicture}
}
\)}\\[14pt] 
\(
\mathaxisbox{
\begin{tikzpicture}[scale=0.32]
\filldraw[Tiley, line width=0.15pt] (0,-1) rectangle (1,0);
\filldraw[Tiley, line width=0.15pt] (1,-1) rectangle (2,0);
\draw[Frame] (0,-1) rectangle (2,0);
\end{tikzpicture}
}%
\mapsto%
\mathaxisbox{
\begin{tikzpicture}[scale=0.32]
\filldraw[Tiley, line width=0.15pt] (0,-1) rectangle (1,0);
\filldraw[Tileb, line width=0.15pt] (1,-1) rectangle (2,0);
\filldraw[Tiley, line width=0.15pt] (2,-1) rectangle (3,0);
\filldraw[Tileb, line width=0.15pt] (3,-1) rectangle (4,0);
\filldraw[Tiler, line width=0.15pt] (0,-2) rectangle (1,-1);
\filldraw[Tiley, line width=0.15pt] (1,-2) rectangle (2,-1);
\filldraw[Tiler, line width=0.15pt] (2,-2) rectangle (3,-1);
\filldraw[Tiley, line width=0.15pt] (3,-2) rectangle (4,-1);
\draw[Frame] (0,-2) rectangle (4,0);
\end{tikzpicture}
}
\) & 
\(
\mathaxisbox{
\begin{tikzpicture}[scale=0.32]
\filldraw[Tiley, line width=0.15pt] (0,-1) rectangle (1,0);
\filldraw[Tileb, line width=0.15pt] (1,-1) rectangle (2,0);
\draw[Frame] (0,-1) rectangle (2,0);
\end{tikzpicture}
}%
\mapsto%
\mathaxisbox{
\begin{tikzpicture}[scale=0.32]
\filldraw[Tiley, line width=0.15pt] (0,-1) rectangle (1,0);
\filldraw[Tileb, line width=0.15pt] (1,-1) rectangle (2,0);
\filldraw[Tiley, line width=0.15pt] (2,-1) rectangle (3,0);
\filldraw[Tileb, line width=0.15pt] (3,-1) rectangle (4,0);
\filldraw[Tiler, line width=0.15pt] (0,-2) rectangle (1,-1);
\filldraw[Tiley, line width=0.15pt] (1,-2) rectangle (2,-1);
\filldraw[Tileb, line width=0.15pt] (2,-2) rectangle (3,-1);
\filldraw[Tileg, line width=0.15pt] (3,-2) rectangle (4,-1);
\draw[Frame] (0,-2) rectangle (4,0);
\end{tikzpicture}
}
\) &
\(
\mathaxisbox{
\begin{tikzpicture}[scale=0.32]
\filldraw[Tileb, line width=0.15pt] (0,-1) rectangle (1,0);
\filldraw[Tiler, line width=0.15pt] (1,-1) rectangle (2,0);
\draw[Frame] (0,-1) rectangle (2,0);
\end{tikzpicture}
}%
\mapsto%
\mathaxisbox{
\begin{tikzpicture}[scale=0.32]
\filldraw[Tiley, line width=0.15pt] (0,-1) rectangle (1,0);
\filldraw[Tileb, line width=0.15pt] (1,-1) rectangle (2,0);
\filldraw[Tiley, line width=0.15pt] (2,-1) rectangle (3,0);
\filldraw[Tiler, line width=0.15pt] (3,-1) rectangle (4,0);
\filldraw[Tileb, line width=0.15pt] (0,-2) rectangle (1,-1);
\filldraw[Tileg, line width=0.15pt] (1,-2) rectangle (2,-1);
\filldraw[Tiler, line width=0.15pt] (2,-2) rectangle (3,-1);
\filldraw[Tileg, line width=0.15pt] (3,-2) rectangle (4,-1);
\draw[Frame] (0,-2) rectangle (4,0);
\end{tikzpicture}
}
\) & \(
\mathaxisbox{
\begin{tikzpicture}[scale=0.32]
\filldraw[Tileb, line width=0.15pt] (0,-1) rectangle (1,0);
\filldraw[Tiley, line width=0.15pt] (1,-1) rectangle (2,0);
\draw[Frame] (0,-1) rectangle (2,0);
\end{tikzpicture}
}%
\mapsto%
\mathaxisbox{
\begin{tikzpicture}[scale=0.32]
\filldraw[Tiley, line width=0.15pt] (0,-1) rectangle (1,0);
\filldraw[Tileb, line width=0.15pt] (1,-1) rectangle (2,0);
\filldraw[Tiley, line width=0.15pt] (2,-1) rectangle (3,0);
\filldraw[Tileb, line width=0.15pt] (3,-1) rectangle (4,0);
\filldraw[Tileb, line width=0.15pt] (0,-2) rectangle (1,-1);
\filldraw[Tileg, line width=0.15pt] (1,-2) rectangle (2,-1);
\filldraw[Tiler, line width=0.15pt] (2,-2) rectangle (3,-1);
\filldraw[Tiley, line width=0.15pt] (3,-2) rectangle (4,-1);
\draw[Frame] (0,-2) rectangle (4,0);
\end{tikzpicture}
}
\) \\[6pt]
\(
\mathaxisbox{
\begin{tikzpicture}[scale=0.32]
\filldraw[Tileb, line width=0.15pt] (0,-1) rectangle (1,0);
\filldraw[Tileb, line width=0.15pt] (1,-1) rectangle (2,0);
\draw[Frame] (0,-1) rectangle (2,0);
\end{tikzpicture}
}%
\mapsto%
\mathaxisbox{
\begin{tikzpicture}[scale=0.32]
\filldraw[Tiley, line width=0.15pt] (0,-1) rectangle (1,0);
\filldraw[Tileb, line width=0.15pt] (1,-1) rectangle (2,0);
\filldraw[Tiley, line width=0.15pt] (2,-1) rectangle (3,0);
\filldraw[Tileb, line width=0.15pt] (3,-1) rectangle (4,0);
\filldraw[Tileb, line width=0.15pt] (0,-2) rectangle (1,-1);
\filldraw[Tileg, line width=0.15pt] (1,-2) rectangle (2,-1);
\filldraw[Tileb, line width=0.15pt] (2,-2) rectangle (3,-1);
\filldraw[Tileg, line width=0.15pt] (3,-2) rectangle (4,-1);
\draw[Frame] (0,-2) rectangle (4,0);
\end{tikzpicture}
}
\) & \(
\mathaxisbox{
\begin{tikzpicture}[scale=0.32]
\filldraw[Tileb, line width=0.15pt] (0,-1) rectangle (1,0);
\filldraw[Tileg, line width=0.15pt] (1,-1) rectangle (2,0);
\draw[Frame] (0,-1) rectangle (2,0);
\end{tikzpicture}
}%
\mapsto%
\mathaxisbox{
\begin{tikzpicture}[scale=0.32]
\filldraw[Tiley, line width=0.15pt] (0,-1) rectangle (1,0);
\filldraw[Tileb, line width=0.15pt] (1,-1) rectangle (2,0);
\filldraw[Tileg, line width=0.15pt] (2,-1) rectangle (3,0);
\filldraw[Tileb, line width=0.15pt] (3,-1) rectangle (4,0);
\filldraw[Tileb, line width=0.15pt] (0,-2) rectangle (1,-1);
\filldraw[Tileg, line width=0.15pt] (1,-2) rectangle (2,-1);
\filldraw[Tiler, line width=0.15pt] (2,-2) rectangle (3,-1);
\filldraw[Tileg, line width=0.15pt] (3,-2) rectangle (4,-1);
\draw[Frame] (0,-2) rectangle (4,0);
\end{tikzpicture}
}
\) &
\(
\mathaxisbox{
\begin{tikzpicture}[scale=0.32]
\filldraw[Tileg, line width=0.15pt] (0,-1) rectangle (1,0);
\filldraw[Tiler, line width=0.15pt] (1,-1) rectangle (2,0);
\draw[Frame] (0,-1) rectangle (2,0);
\end{tikzpicture}
}%
\mapsto%
\mathaxisbox{
\begin{tikzpicture}[scale=0.32]
\filldraw[Tileg, line width=0.15pt] (0,-1) rectangle (1,0);
\filldraw[Tileb, line width=0.15pt] (1,-1) rectangle (2,0);
\filldraw[Tiley, line width=0.15pt] (2,-1) rectangle (3,0);
\filldraw[Tiler, line width=0.15pt] (3,-1) rectangle (4,0);
\filldraw[Tiler, line width=0.15pt] (0,-2) rectangle (1,-1);
\filldraw[Tileg, line width=0.15pt] (1,-2) rectangle (2,-1);
\filldraw[Tiler, line width=0.15pt] (2,-2) rectangle (3,-1);
\filldraw[Tileg, line width=0.15pt] (3,-2) rectangle (4,-1);
\draw[Frame] (0,-2) rectangle (4,0);
\end{tikzpicture}
}
\) &
\(
\mathaxisbox{
\begin{tikzpicture}[scale=0.32]
\filldraw[Tileg, line width=0.15pt] (0,-1) rectangle (1,0);
\filldraw[Tiley, line width=0.15pt] (1,-1) rectangle (2,0);
\draw[Frame] (0,-1) rectangle (2,0);
\end{tikzpicture}
}%
\mapsto%
\mathaxisbox{
\begin{tikzpicture}[scale=0.32]
\filldraw[Tileg, line width=0.15pt] (0,-1) rectangle (1,0);
\filldraw[Tileb, line width=0.15pt] (1,-1) rectangle (2,0);
\filldraw[Tiley, line width=0.15pt] (2,-1) rectangle (3,0);
\filldraw[Tileb, line width=0.15pt] (3,-1) rectangle (4,0);
\filldraw[Tiler, line width=0.15pt] (0,-2) rectangle (1,-1);
\filldraw[Tileg, line width=0.15pt] (1,-2) rectangle (2,-1);
\filldraw[Tiler, line width=0.15pt] (2,-2) rectangle (3,-1);
\filldraw[Tiley, line width=0.15pt] (3,-2) rectangle (4,-1);
\draw[Frame] (0,-2) rectangle (4,0);
\end{tikzpicture}
}
\) \\[6pt] 
\(
\mathaxisbox{
\begin{tikzpicture}[scale=0.32]
\filldraw[Tileg, line width=0.15pt] (0,-1) rectangle (1,0);
\filldraw[Tileb, line width=0.15pt] (1,-1) rectangle (2,0);
\draw[Frame] (0,-1) rectangle (2,0);
\end{tikzpicture}
}%
\mapsto%
\mathaxisbox{
\begin{tikzpicture}[scale=0.32]
\filldraw[Tileg, line width=0.15pt] (0,-1) rectangle (1,0);
\filldraw[Tileb, line width=0.15pt] (1,-1) rectangle (2,0);
\filldraw[Tiley, line width=0.15pt] (2,-1) rectangle (3,0);
\filldraw[Tileb, line width=0.15pt] (3,-1) rectangle (4,0);
\filldraw[Tiler, line width=0.15pt] (0,-2) rectangle (1,-1);
\filldraw[Tileg, line width=0.15pt] (1,-2) rectangle (2,-1);
\filldraw[Tileb, line width=0.15pt] (2,-2) rectangle (3,-1);
\filldraw[Tileg, line width=0.15pt] (3,-2) rectangle (4,-1);
\draw[Frame] (0,-2) rectangle (4,0);
\end{tikzpicture}
}
\) & \(
\mathaxisbox{
\begin{tikzpicture}[scale=0.32]
\filldraw[Tileg, line width=0.15pt] (0,-1) rectangle (1,0);
\filldraw[Tileg, line width=0.15pt] (1,-1) rectangle (2,0);
\draw[Frame] (0,-1) rectangle (2,0);
\end{tikzpicture}
}%
\mapsto%
\mathaxisbox{
\begin{tikzpicture}[scale=0.32]
\filldraw[Tileg, line width=0.15pt] (0,-1) rectangle (1,0);
\filldraw[Tileb, line width=0.15pt] (1,-1) rectangle (2,0);
\filldraw[Tileg, line width=0.15pt] (2,-1) rectangle (3,0);
\filldraw[Tileb, line width=0.15pt] (3,-1) rectangle (4,0);
\filldraw[Tiler, line width=0.15pt] (0,-2) rectangle (1,-1);
\filldraw[Tileg, line width=0.15pt] (1,-2) rectangle (2,-1);
\filldraw[Tiler, line width=0.15pt] (2,-2) rectangle (3,-1);
\filldraw[Tileg, line width=0.15pt] (3,-2) rectangle (4,-1);
\draw[Frame] (0,-2) rectangle (4,0);
\end{tikzpicture}
}
\) 
\end{tabular}
\par\endgroup
\end{table}

\begin{table}
\caption{$S$ and $S^2$ applied to all vertical $1\times2$-patches. If after one iteration all resulting $T$-patches are already $S$-legal, we do not display further iterates.}
\label{Tab:Chair_Iterates_Vertical}
\centering
\begingroup
\setlength{\tabcolsep}{4pt}%
\renewcommand{\arraystretch}{1.10}%
\noindent\begin{tabular}{@{}p{0.157\linewidth}@{}p{0.157\linewidth}@{}p{0.157\linewidth}@{}p{0.157\linewidth}@{}p{0.157\linewidth}@{}p{0.157\linewidth}@{}}
\(
\mathaxisbox{
\begin{tikzpicture}[scale=0.32]
\filldraw[Tiler, line width=0.15pt] (0,-1) rectangle (1,0);
\filldraw[Tiler, line width=0.15pt] (0,-2) rectangle (1,-1);
\draw[Frame] (0,-2) rectangle (1,0);
\end{tikzpicture}
}%
\mapsto%
\mathaxisbox{
\begin{tikzpicture}[scale=0.32]
\filldraw[Tiley, line width=0.15pt] (0,-1) rectangle (1,0);
\filldraw[Tiler, line width=0.15pt] (1,-1) rectangle (2,0);
\filldraw[Tiler, line width=0.15pt] (0,-2) rectangle (1,-1);
\filldraw[Tileg, line width=0.15pt] (1,-2) rectangle (2,-1);
\filldraw[Tiley, line width=0.15pt] (0,-3) rectangle (1,-2);
\filldraw[Tiler, line width=0.15pt] (1,-3) rectangle (2,-2);
\filldraw[Tiler, line width=0.15pt] (0,-4) rectangle (1,-3);
\filldraw[Tileg, line width=0.15pt] (1,-4) rectangle (2,-3);
\draw[Frame] (0,-4) rectangle (2,0);
\end{tikzpicture}
}
\) &
\(
\mathaxisbox{
\begin{tikzpicture}[scale=0.32]
\filldraw[Tiler, line width=0.15pt] (0,-1) rectangle (1,0);
\filldraw[Tiley, line width=0.15pt] (0,-2) rectangle (1,-1);
\draw[Frame] (0,-2) rectangle (1,0);
\end{tikzpicture}
}%
\mapsto%
\mathaxisbox{
\begin{tikzpicture}[scale=0.32]
\filldraw[Tiley, line width=0.15pt] (0,-1) rectangle (1,0);
\filldraw[Tiler, line width=0.15pt] (1,-1) rectangle (2,0);
\filldraw[Tiler, line width=0.15pt] (0,-2) rectangle (1,-1);
\filldraw[Tileg, line width=0.15pt] (1,-2) rectangle (2,-1);
\filldraw[Tiley, line width=0.15pt] (0,-3) rectangle (1,-2);
\filldraw[Tileb, line width=0.15pt] (1,-3) rectangle (2,-2);
\filldraw[Tiler, line width=0.15pt] (0,-4) rectangle (1,-3);
\filldraw[Tiley, line width=0.15pt] (1,-4) rectangle (2,-3);
\draw[Frame] (0,-4) rectangle (2,0);
\end{tikzpicture}
}
\) &
\(
\mathaxisbox{
\begin{tikzpicture}[scale=0.32]
\filldraw[Tiler, line width=0.15pt] (0,-1) rectangle (1,0);
\filldraw[Tileb, line width=0.15pt] (0,-2) rectangle (1,-1);
\draw[Frame] (0,-2) rectangle (1,0);
\end{tikzpicture}
}%
\mapsto%
\mathaxisbox{
\begin{tikzpicture}[scale=0.32]
\filldraw[Tiley, line width=0.15pt] (0,-1) rectangle (1,0);
\filldraw[Tiler, line width=0.15pt] (1,-1) rectangle (2,0);
\filldraw[Tiler, line width=0.15pt] (0,-2) rectangle (1,-1);
\filldraw[Tileg, line width=0.15pt] (1,-2) rectangle (2,-1);
\filldraw[Tiley, line width=0.15pt] (0,-3) rectangle (1,-2);
\filldraw[Tileb, line width=0.15pt] (1,-3) rectangle (2,-2);
\filldraw[Tileb, line width=0.15pt] (0,-4) rectangle (1,-3);
\filldraw[Tileg, line width=0.15pt] (1,-4) rectangle (2,-3);
\draw[Frame] (0,-4) rectangle (2,0);
\end{tikzpicture}
}
\) &
\(
\mathaxisbox{
\begin{tikzpicture}[scale=0.32]
\filldraw[Tiler, line width=0.15pt] (0,-1) rectangle (1,0);
\filldraw[Tileg, line width=0.15pt] (0,-2) rectangle (1,-1);
\draw[Frame] (0,-2) rectangle (1,0);
\end{tikzpicture}
}%
\mapsto%
\mathaxisbox{
\begin{tikzpicture}[scale=0.32]
\filldraw[Tiley, line width=0.15pt] (0,-1) rectangle (1,0);
\filldraw[Tiler, line width=0.15pt] (1,-1) rectangle (2,0);
\filldraw[Tiler, line width=0.15pt] (0,-2) rectangle (1,-1);
\filldraw[Tileg, line width=0.15pt] (1,-2) rectangle (2,-1);
\filldraw[Tileg, line width=0.15pt] (0,-3) rectangle (1,-2);
\filldraw[Tileb, line width=0.15pt] (1,-3) rectangle (2,-2);
\filldraw[Tiler, line width=0.15pt] (0,-4) rectangle (1,-3);
\filldraw[Tileg, line width=0.15pt] (1,-4) rectangle (2,-3);
\draw[Frame] (0,-4) rectangle (2,0);
\end{tikzpicture}
}
\) &
\(
\mathaxisbox{
\begin{tikzpicture}[scale=0.32]
\filldraw[Tileg, line width=0.15pt] (0,-1) rectangle (1,0);
\filldraw[Tileb, line width=0.15pt] (0,-2) rectangle (1,-1);
\draw[Frame] (0,-2) rectangle (1,0);
\end{tikzpicture}
}%
\mapsto%
\mathaxisbox{
\begin{tikzpicture}[scale=0.32]
\filldraw[Tileg, line width=0.15pt] (0,-1) rectangle (1,0);
\filldraw[Tileb, line width=0.15pt] (1,-1) rectangle (2,0);
\filldraw[Tiler, line width=0.15pt] (0,-2) rectangle (1,-1);
\filldraw[Tileg, line width=0.15pt] (1,-2) rectangle (2,-1);
\filldraw[Tiley, line width=0.15pt] (0,-3) rectangle (1,-2);
\filldraw[Tileb, line width=0.15pt] (1,-3) rectangle (2,-2);
\filldraw[Tileb, line width=0.15pt] (0,-4) rectangle (1,-3);
\filldraw[Tileg, line width=0.15pt] (1,-4) rectangle (2,-3);
\draw[Frame] (0,-4) rectangle (2,0);
\end{tikzpicture}
}
\) &
\(
\mathaxisbox{
\begin{tikzpicture}[scale=0.32]
\filldraw[Tileg, line width=0.15pt] (0,-1) rectangle (1,0);
\filldraw[Tiler, line width=0.15pt] (0,-2) rectangle (1,-1);
\draw[Frame] (0,-2) rectangle (1,0);
\end{tikzpicture}
}%
\mapsto%
\mathaxisbox{
\begin{tikzpicture}[scale=0.32]
\filldraw[Tileg, line width=0.15pt] (0,-1) rectangle (1,0);
\filldraw[Tileb, line width=0.15pt] (1,-1) rectangle (2,0);
\filldraw[Tiler, line width=0.15pt] (0,-2) rectangle (1,-1);
\filldraw[Tileg, line width=0.15pt] (1,-2) rectangle (2,-1);
\filldraw[Tiley, line width=0.15pt] (0,-3) rectangle (1,-2);
\filldraw[Tiler, line width=0.15pt] (1,-3) rectangle (2,-2);
\filldraw[Tiler, line width=0.15pt] (0,-4) rectangle (1,-3);
\filldraw[Tileg, line width=0.15pt] (1,-4) rectangle (2,-3);
\draw[Frame] (0,-4) rectangle (2,0);
\end{tikzpicture}
}
\)\\[15pt]

\(
\mathaxisbox{
\begin{tikzpicture}[scale=0.32]
\filldraw[Tiley, line width=0.15pt] (0,-1) rectangle (1,0);
\filldraw[Tiler, line width=0.15pt] (0,-2) rectangle (1,-1);
\draw[Frame] (0,-2) rectangle (1,0);
\end{tikzpicture}
}%
\mapsto%
\mathaxisbox{
\begin{tikzpicture}[scale=0.32]
\filldraw[Tiley, line width=0.15pt] (0,-1) rectangle (1,0);
\filldraw[Tileb, line width=0.15pt] (1,-1) rectangle (2,0);
\filldraw[Tiler, line width=0.15pt] (0,-2) rectangle (1,-1);
\filldraw[Tiley, line width=0.15pt] (1,-2) rectangle (2,-1);
\filldraw[Tiley, line width=0.15pt] (0,-3) rectangle (1,-2);
\filldraw[Tiler, line width=0.15pt] (1,-3) rectangle (2,-2);
\filldraw[Tiler, line width=0.15pt] (0,-4) rectangle (1,-3);
\filldraw[Tileg, line width=0.15pt] (1,-4) rectangle (2,-3);
\draw[Frame] (0,-4) rectangle (2,0);
\end{tikzpicture}
}
\) &
\(
\mathaxisbox{
\begin{tikzpicture}[scale=0.32]
\filldraw[Tiley, line width=0.15pt] (0,-1) rectangle (1,0);
\filldraw[Tiley, line width=0.15pt] (0,-2) rectangle (1,-1);
\draw[Frame] (0,-2) rectangle (1,0);
\end{tikzpicture}
}%
\mapsto%
\mathaxisbox{
\begin{tikzpicture}[scale=0.32]
\filldraw[Tiley, line width=0.15pt] (0,-1) rectangle (1,0);
\filldraw[Tileb, line width=0.15pt] (1,-1) rectangle (2,0);
\filldraw[Tiler, line width=0.15pt] (0,-2) rectangle (1,-1);
\filldraw[Tiley, line width=0.15pt] (1,-2) rectangle (2,-1);
\filldraw[Tiley, line width=0.15pt] (0,-3) rectangle (1,-2);
\filldraw[Tileb, line width=0.15pt] (1,-3) rectangle (2,-2);
\filldraw[Tiler, line width=0.15pt] (0,-4) rectangle (1,-3);
\filldraw[Tiley, line width=0.15pt] (1,-4) rectangle (2,-3);
\draw[Frame] (0,-4) rectangle (2,0);
\end{tikzpicture}
}
\) &
\(
\mathaxisbox{
\begin{tikzpicture}[scale=0.32]
\filldraw[Tiley, line width=0.15pt] (0,-1) rectangle (1,0);
\filldraw[Tileb, line width=0.15pt] (0,-2) rectangle (1,-1);
\draw[Frame] (0,-2) rectangle (1,0);
\end{tikzpicture}
}%
\mapsto%
\mathaxisbox{
\begin{tikzpicture}[scale=0.32]
\filldraw[Tiley, line width=0.15pt] (0,-1) rectangle (1,0);
\filldraw[Tileb, line width=0.15pt] (1,-1) rectangle (2,0);
\filldraw[Tiler, line width=0.15pt] (0,-2) rectangle (1,-1);
\filldraw[Tiley, line width=0.15pt] (1,-2) rectangle (2,-1);
\filldraw[Tiley, line width=0.15pt] (0,-3) rectangle (1,-2);
\filldraw[Tileb, line width=0.15pt] (1,-3) rectangle (2,-2);
\filldraw[Tileb, line width=0.15pt] (0,-4) rectangle (1,-3);
\filldraw[Tileg, line width=0.15pt] (1,-4) rectangle (2,-3);
\draw[Frame] (0,-4) rectangle (2,0);
\end{tikzpicture}
}
\) &
\(
\mathaxisbox{
\begin{tikzpicture}[scale=0.32]
\filldraw[Tileb, line width=0.15pt] (0,-1) rectangle (1,0);
\filldraw[Tiley, line width=0.15pt] (0,-2) rectangle (1,-1);
\draw[Frame] (0,-2) rectangle (1,0);
\end{tikzpicture}
}%
\mapsto%
\mathaxisbox{
\begin{tikzpicture}[scale=0.32]
\filldraw[Tiley, line width=0.15pt] (0,-1) rectangle (1,0);
\filldraw[Tileb, line width=0.15pt] (1,-1) rectangle (2,0);
\filldraw[Tileb, line width=0.15pt] (0,-2) rectangle (1,-1);
\filldraw[Tileg, line width=0.15pt] (1,-2) rectangle (2,-1);
\filldraw[Tiley, line width=0.15pt] (0,-3) rectangle (1,-2);
\filldraw[Tileb, line width=0.15pt] (1,-3) rectangle (2,-2);
\filldraw[Tiler, line width=0.15pt] (0,-4) rectangle (1,-3);
\filldraw[Tiley, line width=0.15pt] (1,-4) rectangle (2,-3);
\draw[Frame] (0,-4) rectangle (2,0);
\end{tikzpicture}
}
\) &
\(
\mathaxisbox{
\begin{tikzpicture}[scale=0.32]
\filldraw[Tileb, line width=0.15pt] (0,-1) rectangle (1,0);
\filldraw[Tileb, line width=0.15pt] (0,-2) rectangle (1,-1);
\draw[Frame] (0,-2) rectangle (1,0);
\end{tikzpicture}
}%
\mapsto%
\mathaxisbox{
\begin{tikzpicture}[scale=0.32]
\filldraw[Tiley, line width=0.15pt] (0,-1) rectangle (1,0);
\filldraw[Tileb, line width=0.15pt] (1,-1) rectangle (2,0);
\filldraw[Tileb, line width=0.15pt] (0,-2) rectangle (1,-1);
\filldraw[Tileg, line width=0.15pt] (1,-2) rectangle (2,-1);
\filldraw[Tiley, line width=0.15pt] (0,-3) rectangle (1,-2);
\filldraw[Tileb, line width=0.15pt] (1,-3) rectangle (2,-2);
\filldraw[Tileb, line width=0.15pt] (0,-4) rectangle (1,-3);
\filldraw[Tileg, line width=0.15pt] (1,-4) rectangle (2,-3);
\draw[Frame] (0,-4) rectangle (2,0);
\end{tikzpicture}
}
\) &
\(
\mathaxisbox{
\begin{tikzpicture}[scale=0.32]
\filldraw[Tileb, line width=0.15pt] (0,-1) rectangle (1,0);
\filldraw[Tileg, line width=0.15pt] (0,-2) rectangle (1,-1);
\draw[Frame] (0,-2) rectangle (1,0);
\end{tikzpicture}
}%
\mapsto%
\mathaxisbox{
\begin{tikzpicture}[scale=0.32]
\filldraw[Tiley, line width=0.15pt] (0,-1) rectangle (1,0);
\filldraw[Tileb, line width=0.15pt] (1,-1) rectangle (2,0);
\filldraw[Tileb, line width=0.15pt] (0,-2) rectangle (1,-1);
\filldraw[Tileg, line width=0.15pt] (1,-2) rectangle (2,-1);
\filldraw[Tileg, line width=0.15pt] (0,-3) rectangle (1,-2);
\filldraw[Tileb, line width=0.15pt] (1,-3) rectangle (2,-2);
\filldraw[Tiler, line width=0.15pt] (0,-4) rectangle (1,-3);
\filldraw[Tileg, line width=0.15pt] (1,-4) rectangle (2,-3);
\draw[Frame] (0,-4) rectangle (2,0);
\end{tikzpicture}
}
\)\\[15pt]

\multicolumn{4}{@{}p{0.57\linewidth}@{}}{\(
\mathaxisbox{
\begin{tikzpicture}[scale=0.32]
\filldraw[Tiley, line width=0.15pt] (0,-1) rectangle (1,0);
\filldraw[Tileg, line width=0.15pt] (0,-2) rectangle (1,-1);
\draw[Frame] (0,-2) rectangle (1,0);
\end{tikzpicture}
}%
\mapsto%
\mathaxisbox{
\begin{tikzpicture}[scale=0.32]
\filldraw[Tiley, line width=0.15pt] (0,-1) rectangle (1,0);
\filldraw[Tileb, line width=0.15pt] (1,-1) rectangle (2,0);
\filldraw[Tiler, line width=0.15pt] (0,-2) rectangle (1,-1);
\filldraw[Tiley, line width=0.15pt] (1,-2) rectangle (2,-1);
\filldraw[Tileg, line width=0.15pt] (0,-3) rectangle (1,-2);
\filldraw[Tileb, line width=0.15pt] (1,-3) rectangle (2,-2);
\filldraw[Tiler, line width=0.15pt] (0,-4) rectangle (1,-3);
\filldraw[Tileg, line width=0.15pt] (1,-4) rectangle (2,-3);
\draw[Frame] (0,-4) rectangle (2,0);
\end{tikzpicture}
}%
\mapsto%
\mathaxisbox{
\begin{tikzpicture}[scale=0.32]
\filldraw[Tiley, line width=0.15pt] (0,-1) rectangle (1,0);
\filldraw[Tileb, line width=0.15pt] (1,-1) rectangle (2,0);
\filldraw[Tiley, line width=0.15pt] (2,-1) rectangle (3,0);
\filldraw[Tileb, line width=0.15pt] (3,-1) rectangle (4,0);
\filldraw[Tiler, line width=0.15pt] (0,-2) rectangle (1,-1);
\filldraw[Tiley, line width=0.15pt] (1,-2) rectangle (2,-1);
\filldraw[Tileb, line width=0.15pt] (2,-2) rectangle (3,-1);
\filldraw[Tileg, line width=0.15pt] (3,-2) rectangle (4,-1);
\filldraw[Tiley, line width=0.15pt] (0,-3) rectangle (1,-2);
\filldraw[Tiler, line width=0.15pt] (1,-3) rectangle (2,-2);
\filldraw[Tiley, line width=0.15pt] (2,-3) rectangle (3,-2);
\filldraw[Tileb, line width=0.15pt] (3,-3) rectangle (4,-2);
\filldraw[Tiler, line width=0.15pt] (0,-4) rectangle (1,-3);
\filldraw[Tileg, line width=0.15pt] (1,-4) rectangle (2,-3);
\filldraw[Tiler, line width=0.15pt] (2,-4) rectangle (3,-3);
\filldraw[Tiley, line width=0.15pt] (3,-4) rectangle (4,-3);
\filldraw[Tileg, line width=0.15pt] (0,-5) rectangle (1,-4);
\filldraw[Tileb, line width=0.15pt] (1,-5) rectangle (2,-4);
\filldraw[Tiley, line width=0.15pt] (2,-5) rectangle (3,-4);
\filldraw[Tileb, line width=0.15pt] (3,-5) rectangle (4,-4);
\filldraw[Tiler, line width=0.15pt] (0,-6) rectangle (1,-5);
\filldraw[Tileg, line width=0.15pt] (1,-6) rectangle (2,-5);
\filldraw[Tileb, line width=0.15pt] (2,-6) rectangle (3,-5);
\filldraw[Tileg, line width=0.15pt] (3,-6) rectangle (4,-5);
\filldraw[Tiley, line width=0.15pt] (0,-7) rectangle (1,-6);
\filldraw[Tiler, line width=0.15pt] (1,-7) rectangle (2,-6);
\filldraw[Tileg, line width=0.15pt] (2,-7) rectangle (3,-6);
\filldraw[Tileb, line width=0.15pt] (3,-7) rectangle (4,-6);
\filldraw[Tiler, line width=0.15pt] (0,-8) rectangle (1,-7);
\filldraw[Tileg, line width=0.15pt] (1,-8) rectangle (2,-7);
\filldraw[Tiler, line width=0.15pt] (2,-8) rectangle (3,-7);
\filldraw[Tileg, line width=0.15pt] (3,-8) rectangle (4,-7);
\draw[Frame] (0,-8) rectangle (4,0);
\end{tikzpicture}
}
\)
\quad
\(
\mathaxisbox{
\begin{tikzpicture}[scale=0.32]
\filldraw[Tileb, line width=0.15pt] (0,-1) rectangle (1,0);
\filldraw[Tiler, line width=0.15pt] (0,-2) rectangle (1,-1);
\draw[Frame] (0,-2) rectangle (1,0);
\end{tikzpicture}
}%
\mapsto%
\mathaxisbox{
\begin{tikzpicture}[scale=0.32]
\filldraw[Tiley, line width=0.15pt] (0,-1) rectangle (1,0);
\filldraw[Tileb, line width=0.15pt] (1,-1) rectangle (2,0);
\filldraw[Tileb, line width=0.15pt] (0,-2) rectangle (1,-1);
\filldraw[Tileg, line width=0.15pt] (1,-2) rectangle (2,-1);
\filldraw[Tiley, line width=0.15pt] (0,-3) rectangle (1,-2);
\filldraw[Tiler, line width=0.15pt] (1,-3) rectangle (2,-2);
\filldraw[Tiler, line width=0.15pt] (0,-4) rectangle (1,-3);
\filldraw[Tileg, line width=0.15pt] (1,-4) rectangle (2,-3);
\draw[Frame] (0,-4) rectangle (2,0);
\end{tikzpicture}
}%
\mapsto%
\mathaxisbox{
\begin{tikzpicture}[scale=0.32]
\filldraw[Tiley, line width=0.15pt] (0,-1) rectangle (1,0);
\filldraw[Tileb, line width=0.15pt] (1,-1) rectangle (2,0);
\filldraw[Tiley, line width=0.15pt] (2,-1) rectangle (3,0);
\filldraw[Tileb, line width=0.15pt] (3,-1) rectangle (4,0);
\filldraw[Tiler, line width=0.15pt] (0,-2) rectangle (1,-1);
\filldraw[Tiley, line width=0.15pt] (1,-2) rectangle (2,-1);
\filldraw[Tileb, line width=0.15pt] (2,-2) rectangle (3,-1);
\filldraw[Tileg, line width=0.15pt] (3,-2) rectangle (4,-1);
\filldraw[Tiley, line width=0.15pt] (0,-3) rectangle (1,-2);
\filldraw[Tileb, line width=0.15pt] (1,-3) rectangle (2,-2);
\filldraw[Tileg, line width=0.15pt] (2,-3) rectangle (3,-2);
\filldraw[Tileb, line width=0.15pt] (3,-3) rectangle (4,-2);
\filldraw[Tileb, line width=0.15pt] (0,-4) rectangle (1,-3);
\filldraw[Tileg, line width=0.15pt] (1,-4) rectangle (2,-3);
\filldraw[Tiler, line width=0.15pt] (2,-4) rectangle (3,-3);
\filldraw[Tileg, line width=0.15pt] (3,-4) rectangle (4,-3);
\filldraw[Tiley, line width=0.15pt] (0,-5) rectangle (1,-4);
\filldraw[Tileb, line width=0.15pt] (1,-5) rectangle (2,-4);
\filldraw[Tiley, line width=0.15pt] (2,-5) rectangle (3,-4);
\filldraw[Tiler, line width=0.15pt] (3,-5) rectangle (4,-4);
\filldraw[Tiler, line width=0.15pt] (0,-6) rectangle (1,-5);
\filldraw[Tiley, line width=0.15pt] (1,-6) rectangle (2,-5);
\filldraw[Tiler, line width=0.15pt] (2,-6) rectangle (3,-5);
\filldraw[Tileg, line width=0.15pt] (3,-6) rectangle (4,-5);
\filldraw[Tiley, line width=0.15pt] (0,-7) rectangle (1,-6);
\filldraw[Tiler, line width=0.15pt] (1,-7) rectangle (2,-6);
\filldraw[Tileg, line width=0.15pt] (2,-7) rectangle (3,-6);
\filldraw[Tileb, line width=0.15pt] (3,-7) rectangle (4,-6);
\filldraw[Tiler, line width=0.15pt] (0,-8) rectangle (1,-7);
\filldraw[Tileg, line width=0.15pt] (1,-8) rectangle (2,-7);
\filldraw[Tiler, line width=0.15pt] (2,-8) rectangle (3,-7);
\filldraw[Tileg, line width=0.15pt] (3,-8) rectangle (4,-7);
\draw[Frame] (0,-8) rectangle (4,0);
\end{tikzpicture}
}
\)} &

\(
\mathaxisbox{
\begin{tikzpicture}[scale=0.32]
\filldraw[Tileg, line width=0.15pt] (0,-1) rectangle (1,0);
\filldraw[Tileg, line width=0.15pt] (0,-2) rectangle (1,-1);
\draw[Frame] (0,-2) rectangle (1,0);
\end{tikzpicture}
}%
\mapsto%
\mathaxisbox{
\begin{tikzpicture}[scale=0.32]
\filldraw[Tileg, line width=0.15pt] (0,-1) rectangle (1,0);
\filldraw[Tileb, line width=0.15pt] (1,-1) rectangle (2,0);
\filldraw[Tiler, line width=0.15pt] (0,-2) rectangle (1,-1);
\filldraw[Tileg, line width=0.15pt] (1,-2) rectangle (2,-1);
\filldraw[Tileg, line width=0.15pt] (0,-3) rectangle (1,-2);
\filldraw[Tileb, line width=0.15pt] (1,-3) rectangle (2,-2);
\filldraw[Tiler, line width=0.15pt] (0,-4) rectangle (1,-3);
\filldraw[Tileg, line width=0.15pt] (1,-4) rectangle (2,-3);
\draw[Frame] (0,-4) rectangle (2,0);
\end{tikzpicture}
}
\) 
 & \(
\mathaxisbox{
\begin{tikzpicture}[scale=0.32]
\filldraw[Tileg, line width=0.15pt] (0,-1) rectangle (1,0);
\filldraw[Tiley, line width=0.15pt] (0,-2) rectangle (1,-1);
\draw[Frame] (0,-2) rectangle (1,0);
\end{tikzpicture}
}%
\mapsto%
\mathaxisbox{
\begin{tikzpicture}[scale=0.32]
\filldraw[Tileg, line width=0.15pt] (0,-1) rectangle (1,0);
\filldraw[Tileb, line width=0.15pt] (1,-1) rectangle (2,0);
\filldraw[Tiler, line width=0.15pt] (0,-2) rectangle (1,-1);
\filldraw[Tileg, line width=0.15pt] (1,-2) rectangle (2,-1);
\filldraw[Tiley, line width=0.15pt] (0,-3) rectangle (1,-2);
\filldraw[Tileb, line width=0.15pt] (1,-3) rectangle (2,-2);
\filldraw[Tiler, line width=0.15pt] (0,-4) rectangle (1,-3);
\filldraw[Tiley, line width=0.15pt] (1,-4) rectangle (2,-3);
\draw[Frame] (0,-4) rectangle (2,0);
\end{tikzpicture}
}
\)\\
\end{tabular}
\par\endgroup
\end{table}

\end{document}